\documentclass[11pt,a4paper]{article}

\usepackage[utf8]{inputenc} 
\usepackage[T1]{fontenc} 

\usepackage{lmodern} 

\usepackage[backend=bibtex,style=trad-alpha,natbib=true,maxalphanames=99,minalphanames=99]{biblatex} 

\usepackage[autostyle=true]{csquotes} 

\usepackage{subfiles}
\usepackage{geometry}

\usepackage{etex}

\usepackage{amsmath}
\allowdisplaybreaks
\usepackage{amssymb}
\usepackage{amsthm}
\usepackage{amsfonts}

\usepackage{mathtools,bm}
\usepackage{comment}
\usepackage{stmaryrd}
\usepackage[shortlabels]{enumitem}

\usepackage{mathrsfs}
\usepackage{verbatim}
\usepackage{cancel}
\usepackage[x11names]{xcolor}

\usepackage{graphicx}
\usepackage{wrapfig}
\usepackage{sansmath}
\usepackage{dsfont}
\usepackage{todonotes}
\usepackage{enumitem}
\usepackage{soul}
\usetikzlibrary{shapes.geometric}
\usepackage{orcidlink}
\usepackage[nameinlink,capitalize]{cleveref}

\usepackage[title,titletoc]{appendix}

\usepackage{xurl}

\usepackage[title]{appendix}

\newcommand{\R}{\mathbb{R}}
\newcommand{\E}{\mathbb{E}}
\newcommand{\rd}{\mathrm{d}}

\newcommand{\esssup}{\mathop{\mathrm{ess\,sup}}}

\newcommand{\be}{\begin{equation}}
\newcommand{\ee}{\end{equation}}
\newcommand{\bea}{\begin{eqnarray}}
\newcommand{\eea}{\end{eqnarray}}

\newtheorem{defn}{Definition}[section]
\theoremstyle{plain}

\newtheorem{rmk}[defn]{Remark}

\newtheorem{thm}[defn]{Theorem}

\newtheorem*{thm*}{Theorem}
\newtheorem{lemma}[defn]{Lemma}
\newtheorem{coroll}[defn]{Corollary}
\newtheorem{prop}[defn]{Proposition}

\newtheorem*{assum*}{Assumptions}
\newtheorem{assumptions}[defn]{Assumptions}

\title{Malliavin Calculus for rough stochastic differential equations}
\author{Fabio Bugini\thanks{
Institut für Mathematik, Technische Universität Berlin, Straße des 17.\ Juni 136, 10623 Berlin (fabio.bugini@tu-berlin.de)} , Michele Coghi\thanks{Dipartimento di Matematica, Università degli Studi di Trento, via Sommarive 14, 38123 Povo (Trento) (michele.coghi@unitn.it)} , Torstein Nielssen\thanks{Department of Mathematical Sciences, University of Agder, Postboks 422, 4604 Kristiansand, Norway (torstein.nilssen@uia.no)}}

\begin{document}

\maketitle

\begin{abstract}
In this work we show that rough stochastic differential equations (RSDEs), as introduced in \cite{FHL21}, are Malliavin differentiable.
We use this to prove existence of a density when the diffusion coefficients satisfies standard ellipticity assumptions. Moreover, when the coefficients are smooth and the diffusion coefficients satisfies a H\"{o}rmander condition, the density is shown to be smooth. 

The key ingredient is to develop a comprehensive theory of linear rough stochastic differential equations, which could be of independent interest. 
\end{abstract}



\tableofcontents

\section{Introduction}
On a complete filtered probability space $(\Omega,\mathcal{F},\{\mathcal{F}_t\}_{t \in [0,T]}, \mathbb{P})$, we consider the following mixed rough path and stochastic differential equation on $\R^d$, for $t\in [0,T]$,
\begin{equation} \label{eq:RSDE_intro}
dX_t=b(X_t)dt+\sigma(X_t)dB_t+\beta(X_t)d\mathbf{Z}_t, \qquad X_0 = x_0 \in \R^d.
\end{equation} 
Here, $B=(B_t)_{t \in [0,T]}$ is an $m-$dimensional $\{\mathcal{F}_t\}$-Brownian motion and $\mathbf{Z}=(Z,\mathbb{Z})$ is a deterministic geometric $\alpha$-rough path, with $\alpha \in \left(\frac{1}{3},\frac{1}{2}\right]$.  
The coefficients of the equation are $b:\R^d\to\R^d$, $\sigma:\R^d\to\mathscr{L}(\R^m,\R^d)\equiv \R^{d \times m}$ and $\beta:\R^d\to\mathscr{L}(\R^n,\R^d)$.


A solution theory for equation \eqref{eq:RSDE_intro} can be found in \cite{FHL21}: Equation \eqref{eq:RSDE_intro} admits a unique solution $(X_t)_{t\in[0,T]}$, under the assumptions on the coefficients one would expect from considering the rough path and the stochastic equation separately; 
When $\beta\equiv 0$ equation \eqref{eq:RSDE_intro} is a classical stochastic differential equation (SDE) which is well-posed if the coefficients are Lipschitz. On the other hand, when $b,\sigma \equiv 0$, rough path theory (see \cite{Lyo98}) tells us that equation \eqref{eq:RSDE_intro} is well-posed when $\beta$ is three times differentiable, see also \cite{FH20} for a general introduction on rough paths.

In this paper, we study the Malliavin differentiability of the solution $X$ and, consequently, the existence (and smoothness) of the density with respect to the Lebesgue measure of the law of $X_t$, $t\in[0,T]$. 


One motivation for studying equations of the type \eqref{eq:RSDE_intro} or \eqref{eq:DSDE} comes from McKean-Vlasov models with common noise. 
When the deterministic rough path $\mathbf{Z}$ in equation \eqref{eq:RSDE_intro} is replaced by a stochastic process $W=(W_t)_{t\in[0,T]}$, independent of $B$, we obtain the following so called doubly stochastic differential equation
\begin{equation} \label{eq:DSDE}
dX_t=b(X_t)dt+\sigma(X_t)dB_t+\beta(X_t)dW_t, \qquad X_0 = x_0 \in \R^d.
\end{equation} 
Consider a system of $N$ (interacting) particles subject to both independent Brownian motions $B^1, \dots, B^N$ (one for each particle) and a common noise $W$. The empirical measure of the system is then expected to converge,  when $N\to \infty$, to a random measure \linebreak $\mu_t = \mathcal{L}(X_t\mid\sigma(W_s, \ s\leq t))$, which is the conditional law of a process $X$ solution to \eqref{eq:DSDE} given the common noise $W$. This is a phenomenon of conditional propagation of chaos, see \cite{kurtz1999particle} or \cite{coghi2016propagation} for a case without independent noise. In the case with interaction, some (or all) the coefficients of \eqref{eq:DSDE} might depend on $\mu$ itself, this case is beyond the scope of the current work. An important generalization of mean field theory is the theory of mean field games with common noise, which arises when the dynamics of the particles (in this case players or agents) are controlled and each agent tries to maximise a value function \cite{CARMONADELARUE_volumeII}. 

Let us emphasize that \eqref{eq:RSDE_intro} is more general than \eqref{eq:DSDE}, since, in general $\mathbf{Z}$ can be sampled from any stochastic process $W$ with continuous trajectories of Hölder regularity in $(\frac{1}{3},\frac{1}{2}]$, for which there exists a rough path lift. In particular $\mathbf{Z}$ can be sampled from a fractional Brownian motion, a more general Gaussian process or a Brownian bridge.

A second motivation for studying equations of the type \eqref{eq:DSDE} comes from filtering theory: in many applications one wants to have information about a stream of data called the signal, say the process $(X_t)_{t\in[0,T]}$, but only has access to the process $(W_t)_{t\in[0,T]}$, called the observation. See \cite{BC09} for a general introduction on stochastic filtering. Also in this context, the object of interest is the so called filtering measure, which is the distribution of $X$ given the observation $W$. As it turns out, under suitable assumptions and the application of a Girsanov tranformation, $X$ can be seen as a diffusion process which depends on the observation $W$ and on other independent fluctuations, as described by equation \eqref{eq:DSDE}. One fundamental question is the robustness of the model, or how the filtering measure depends on the observation or even if this dependence is continuous. Notoriously, the solution of a rough differential equation depends continuously on the driver and rough paths can be employed to construct robust filters, as in \cite{CDFO13} or \cite{CNNR23}.

Another motivation comes from linear rough PDEs as follows.
At least at the formal level, if we define $u_t := \mathcal{L}(X_t)$, then $u$ satisfies the Fokker-Planck equation
\begin{equation}
\label{eq: rough FP}
    \rd u = \mathcal{L}^\ast \ u \rd t + \nabla \cdot (\beta u) \rd \mathbf{Z}_t, \qquad u_0 = \delta_{x_0}
\end{equation}
where $\mathcal{L}^\ast$ is the adjoint operator of
\begin{equation*}
    \mathcal{L} = \sum_{i=1}^d b^i \frac{\partial}{\partial x_i}
    + \frac{1}{2} \sum_{i,j=1}^d a^{i,j}(x)\frac{\partial^2}{\partial x_i x_j},
\end{equation*}
and $a(x) := \sigma(x)\sigma(x)^\top$. 
If one shows that any solution of \eqref{eq: rough FP} can be represented by \eqref{eq:RSDE_intro} by the Feynman-Kac formula  (in a similar vein of \cite{DFS17}), then the results of the present work could be used to show smoothing properties in the solution of \eqref{eq: rough FP}.



The central result of our paper is the following

\begin{thm*}[see Theorem \ref{thm:malliavincalculusforRSDEs}]
If $b$ and $\sigma$ are differentiable and bounded and $\beta$ is three times differentiable and bounded, then the solution $X$ to equation \eqref{eq:RSDE_intro} is Malliavin differentiable.
\end{thm*}

Using the notation from \cite{malliavin2015stochastic} and \cite{nualart2006malliavin} we will actually prove that $X$ belongs to the space $\mathbb{D}^{1,p}$ for every $p\geq 1$.
The strategy of the proof is based on an approximation argument as follows. We replace $\mathbf{Z}$ in equation \eqref{eq:RSDE_intro} with the canonical lift of a smooth path, say $\mathbf{Z}^n$, where $(\mathbf{Z}^n)_{n\in \mathbb{N}}$ converges to $\mathbf{Z}$ in the rough path metric.
When equation \eqref{eq:RSDE_intro} is driven by $\mathbf{Z}^n$, its solution, $X^n$, is well known to be Malliavin differentiable and its derivative satisfies a linearized equation. 
Provided we can show good a priori estimates on this linearized equation, 
Malliavin differentiability of $X$ now follows by $L^2(\Omega)$ stability of $\mathbf{Z} \mapsto X$ in the rough path metric coupled with the fact that the Malliavin derivative is a closed operator.

Towards this goal, we will give a meaning and prove well-posedness to solutions of equations of the form
\begin{align} \begin{aligned}
     dY_t &= dF_t +  G_tY_tdt + S_tY_tdB_t + f_t Y_td\mathbf{Z}_t, \quad t \in [0,T] ,\\
     Y_0&=\xi,
\end{aligned} \label{eq:linearRSDE_model intro}
\end{align}
on any finite dimensional Hilbert space.
Indeed, as in the classical case, we expect the Malliavin derivative $Y^{\theta}_t = \mathbb{D}_{\theta}X_t \in \R^{d \times m}$, $t\ge \theta \in [0,T]$, to solve the following linear rough stochastic equation,
\begin{equation} \label{eq:linearRSDEforthemalliavinderivative intro}
\begin{aligned} &dY^\theta_t=Db(X_t)Y^\theta_t dt+D\sigma(X_t)Y^\theta_t dB_t+D\beta(X_t) Y^\theta_t d\mathbf{Z}_t, \ t \in [\theta,T] \\ &Y^\theta_\theta=\sigma(X_\theta).
\end{aligned}
\end{equation}

In \eqref{eq:linearRSDE_model intro} the coefficients $G$ and $S$ are random linear vector fields and $\xi \in~L^p(\Omega;W).$ The linear term, $f$, and the forcing term, $F$, are assumed to be \emph{stochastically controlled vector fields}, as introduced in \cite{friz2020existence}; This is a generalization of controlled paths in the sense of Gubinelli \cite{gubinelli2004controlling} (see Section \ref{sec: linear rough stochastic equations} below).

Since the vector fields of \eqref{eq:linearRSDE_model intro} are linear, hence not bounded, we cannot apply the well-posedness results of \cite{FHL21}. Our main technical contribution is thus the following. 
\begin{thm*}[see Theorem \ref{thm:wellposednesslinearRSDEs}]
    Equation \eqref{eq:linearRSDE_model intro} admits a unique solution. Moreover, the solution is continuous in all the inputs $(F,G,S,f,\mathbf{Z},\xi)$.
\end{thm*}
The solution of \eqref{eq:linearRSDE_model intro} is understood in the sense of Davie \cite{davie2008differential} (see Definition \ref{def:solutionlinearRSDEs} below) as in \cite{FHL21}. A priori estimates are proved using techniques introduced in \cite{bailleul2017unbounded} and stochastic sewing techniques.
When the coefficients of equation \eqref{eq:RSDE_intro} are smooth, we establish (infinite) Malliavin differentiability for the solutions.

\bigskip

Next, we apply these results to study the density of \eqref{eq:RSDE_intro} as follows. In the regularity class of \cite{FHL21} and under a uniform ellipticity condition on $\sigma$ {(cf.\ Remark \ref{rmk:ellipticity_vs_Hoermander})} we can show that the solution of \eqref{eq:RSDE_intro} admits a density with respect to Lebesgue measure. 
Smoothness of the density is shown when the coefficients are smooth and satisfy an appropriate Hörmander condition as described in the following.

Define the space of vector fields $\mathcal{S}_0 := \{\sigma^k \mid k=1,\dots,m\}$ and recursively 
\begin{equation}
\label{eq: def of S spaces}
\begin{aligned}
\mathcal{S}_{i+1} :=  \mathcal{S}_{i} &\cup
    \{
        [F,V] \mid F\in\{\tilde b , \sigma^1,\dots,\sigma^m,\beta^1,\dots,\beta^n\},
        V\in\mathcal{S}_{i}
    \} 
    \quad
    i\in \mathbb{N},
\end{aligned}  
\end{equation}
as well as $\mathcal{S} = \bigcup_{i} \mathcal{S}_{i}$. Here,
we have denoted by \begin{equation} \label{eq:Stratonovich corrected drift}
    \tilde b(x) = b(x) -\frac12 \sum_{k=1}^m D\sigma^k(x) \sigma^k(x), \qquad x \in \R^d
\end{equation} the Stratonovich corrected drift
\footnote{
Please note that the accepted journal version of this paper contains an inaccurate Stratonovich correction term. This arXiv version provides the exact formulation and, accordingly, the corrected proofs of Theorem \ref{thm: existence under hoermander} and Theorem \ref{thm: smoothness of densities}.
}
and we have used the commutator notation between two vector fields $F,G$, viz
\begin{equation*}
    \label{eq: lie brackets}
    [F,G](x) = DG(x)F(x)-DF(x)G(x), \qquad x\in \R^d.
\end{equation*}
For $x \in \R^d$ we define $\mathcal{S}_{i}(x) = \text{span}\{V(x) | V \in \mathcal{S}_{i} \}$ and $\mathcal{S}(x) = \operatorname{span}\{V(x) \mid V\in\mathcal{S}\}$. 

Our Hörmander condition reads as follows.
\begin{assumptions}[]
\label{def: heormander condition}
    Assume \begin{equation*}
        \text{$b \in {C}^{\infty}_b(\R^d;\R^d)$, $\sigma \in {C}^{\infty}_b(\R^d;\mathscr{L}(\R^m,\R^d))$, and
    $\beta \in {C}^{\infty}_b(\R^d;\mathscr{L}(\R^n,\R^d))$.}
    \end{equation*} 
We say that \textit{H\"ormander's condition} is satisfied {at $x_0 \in \R^d$} if $\mathcal{S}(x_0) = \R^d$.
\end{assumptions}


Our main results are summarized in the following theorem. See Theorems \ref{thm: existence under hoermander} and \ref{thm: smoothness of densities} for precise results and Definitions \ref{def: truly rough} and \ref{def: theta hoelder rough} for the definition of true roughness and $\theta$-Hölder roughness. Heuristically, an $\alpha$-Hölder continuous path, $Z$, is truly rough if it is \emph{not} $2\alpha$-Hölder continuous at every time and this notion yields a suitable rough path version of the Doob-Meyer decomposition. The notion of $\theta$-Hölder roughness is a further quantitative measure on this roughness and yields a suitable rough path version of Norris lemma. 
Both properties are shared, for instance, by the Brownian motion and more in general by a class of Gaussian processes which  include the fractional Brownian motion.
\begin{thm*}
    Suppose the coefficients of \eqref{eq:RSDE_intro} satisfy Assumption \ref{def: heormander condition} and $Z$ is truly rough. 
    Then, for each fixed $t$,  the solution $X_t$ has a density w.r.t. Lebesgue measure.

    If, in addition, $Z$ is $\theta$-Hölder rough, then the density is smooth. 
\end{thm*}

Note that the roughness of $Z$ spans the diffusivity in more directions via each of the vector fields $\beta^1, \dots, \beta^n$ than would be the case if $Z$ was a smooth path. To see this, consider the case when $Z_t^j = t$ for each $j =1, \dots, n$ for which we have classical diffusion
$$
dX_t = \big( b(X_t) + \sum_{j=1}^n \beta^j(X_t) \big) dt + \sigma(X_t) dB_t.
$$
In this setting the first two sets of vector fields would be
$$
\hat{\mathcal{S}}_0 = \{\sigma^k \mid k=1,\dots,m\} = \mathcal{S}_0
$$
and 
\begin{align*}
\hat{\mathcal{S}}_1 = \hat{\mathcal{S}}_0 &\cup \{
        [F,V] \mid F\in\{\tilde b+ \sum_{j=1}^n \beta^j , \sigma^1,\dots,\sigma^m\},
        V\in \hat{\mathcal{S}}_{0}
    \} \
\end{align*}
which clearly gives $\hat{\mathcal{S}}_1(x) \subset \mathcal{S}_1(x)$.

\begin{rmk} \label[rmk]{rmk:ellipticity_vs_Hoermander}
    It is worth noting that, in accordance with the classical case, when the diffusion coefficient $\sigma$ is uniformly elliptic, Hörmander condition is immediately satisfied and there exists a density. In this case the coefficients of the equation do not need to be smooth (this is apparent from the proof of existence of densities). 
\end{rmk}



Malliavin differentiability for solutions to rough differential equations of the form
\begin{equation} \label{eq:purely rough path eq}
dX_t = b(X_t) dt + \xi(X_t) d\mathbf{W}_t
\end{equation}
has been first studied in \cite{cass2010densities} for the case without drift, i.e. $b\equiv 0$, and when $\textbf{W}$ is a centered Gaussian process, with suitable additional properties (those include the fractional Brownian motion of Hurst parameter $H>\frac{1}{4}$). In particular, it is shown that, under Hörmander condition, the solution admits a density.

Later, in \cite{hairer2013regularity} the authors proved smoothness of densities under Hörmander condition for equations of the form \eqref{eq:purely rough path eq} with drift, when the driver is a fractional Brownian motion of Hurst parameter $H\in (\frac{1}{3},\frac{1}{2}]$. A less general case was studied in \cite{hu2013smooth}: in this work smoothness of densities has been proven under the additional assumption that the drift $b$ is $0$ and the columns of $\xi$ are $n$-nilpotent. The existence of smooth densities for the very special case of the rough path lift of a $2$-dimensional fractional Brownian motion with Hurst parameter $H\in (\frac{1}{3},\frac{1}{2}]$ (equation \eqref{eq:purely rough path eq} with constant coefficient $\xi$ and no drift) has been earlier proved in \cite{driscoll2010smoothness}.
The previous results have been generalized in \cite{cass2015smoothness}. The authors consider as driver $\mathbf{W}$ a centered Gaussian process satisfying some additional assumptions (which include, among others, the fractional Brownian motions with Hurst parameter $H>\frac{1}{4}$ and the Brownian bridge) and they prove that the solution to equation \eqref{eq:purely rough path eq} admits a smooth density if the coefficients are smooth and satisfy Hörmander condition.

In principle, one can use the classical approach of joining the deterministic rough path with Brownian motion (see for instance \cite{DFS17}, \cite{DOR13}, \cite{friz2020existence} or \cite{CN21}) to obtain a purely rough equation, as described below. 
Consider the joint rough path 
$$
W_{s,t} := 
\left(
\begin{array}{c}
B_{s,t} \\
Z_{s,t} \\
\end{array}
\right), 
\qquad 
\mathbb{W}_{s,t} := 
\left(
\begin{array}{cc}
\mathbb{B}_{s,t} & \int_s^t Z_{s,r} \otimes dB_r \\
\int_s^t B_{s,r} \otimes dZ_r & \mathbb{Z}_{s,t} \\
\end{array}
\right) 
$$
and the vector fields
$$
\xi_j := 
\left\{
\begin{array}{cc}
\sigma_j & \textrm{ for } j = 1, \dots, m \\
\beta_{j-m} & \textrm{ for } j = m+1, \dots, m+n. \\
\end{array}
\right.
$$
We can then consider equation \eqref{eq:purely rough path eq}, where $\xi : \R^d \rightarrow \mathscr{L}(\R^{m+n};\R^d)$.
However, if we try to fit equation \eqref{eq:RSDE_intro} in the pure rough framework, we would incur in a couple of drawbacks;
First, the pure rough approach requires $\xi$ (and, hence, $\sigma$) to be $3$-times differentiable with bounded derivatives of all orders. As it has been shown in \cite{FHL21} this assumption can be weakened for rough stochastic differential equations of the form \eqref{eq:RSDE_intro}, thanks to the stochastic sewing approach, first developed in \cite{LE20}.

Second, to invoke a Norris type argument for $W = (B,Z)$ in the rough path framework one would need to know that the sample paths of this process are almost surely $\theta$-Hölder rough. 
However, to the best of the authors knowledge, it is not possible to conclude $\theta$-Hölder roughness of $W$ based on knowledge of $B$ and $Z$ separately; showing this property using existing techniques (see e.g. \cite[Proposition 6.11]{FH20}) relies on probabilistic arguments on the full path. With this in mind, such a condition would need to be \emph{assumed} for $W$
\footnote{On the other hand, if $Z$ itself is a sample path of the Brownian motion, then one could argue that with probability 1 the path is indeed $\theta$-Hölder rough, since then $W$ itself is just a sample path of the $m+n$-dimensional Brownian motion.}.
Finally, notice that in previous works on densities for rough differential equations, like \cite{cass2010densities}, \cite{hairer2011malliavins} and \cite{cass2015smoothness}, one must assume that $(B,Z)$ is centered Gaussian, which is not the case when $Z$ is a deterministic path.


The paper is organized as follows.
In Section \ref{sec: notation} we introduce some notation and we state some preliminary results in stochastic rough analysis. In Section \ref{sec: linear rough stochastic equations} we prove existence, uniqueness and stability for rough stochastic linear equations. Section \ref{sec: malliavin} contains the main result about Malliavin differentiability of the solutions to rough stochastic differential equations. In section \ref{sec: hoermander} we prove existence and smoothness of the density under Hörmander condition.

\subsection*{Acknowledgment}
F.B.\ acknowledges support from the Deutsche Forschungsgemeinschaft (DFG) - Project-ID 410208580 - IRTG2544 (”Stochastic Analysis in Interaction”).

M.C.\ thanks Istituto Nazionale di
Alta Matematica, group GNAMPA, through the project ‘Fluidodinamica Stocastica’ - E53C23001670001.

The authors would like to thank Peter Friz for useful insights on joint rough paths lifts.

\section{Notation and Preliminaries}
\label{sec: notation}

\subsection{Notation} We write $a \lesssim b$ to express that there exists a non-negative constant $C\ge 0$ such that $a \le Cb$. If such a constant depends on some parameters $\theta_1,\dots,\theta_N$ we write $a \lesssim_{\theta_1,\dots,\theta_N} b$.
For any $\pi=\{a=t_0 < t_1< \dots <t_N=b \}$ partition of $[a,b]$, we define its mesh size as \begin{equation*}
    |\pi| := \max_{i=0,\dots,N-1} |t_{i+1}-t_i|.
\end{equation*}
For any $x \in \R^d$, we denote by $x^i$ its $i$-th component ($i=1,\dots,d)$. 

 Let $(W,|\cdot|_W)$ and $(\bar{W},|\cdot|_{\bar{W}})$ be two Banach spaces. We denote by $\mathscr{L}(W,\bar{W})$ the space of linear and continuous functions from $W$ to $\bar{W}$. The latter is a Banach space, if endowed with the norm $|T|_{\mathscr{L}(W,\bar{W})} := \sup_{x \in W, |x|_W\le 1} |Tx|_{\bar{W}}$. For any $T \in \mathscr{L}(W,\bar{W})$ and for any $x \in W$, we have that \begin{equation*}
    |Tx|_{\bar{W}} \le |T|_{\mathscr{L}(W,\bar{W})}|x|_W.
\end{equation*} 
We write $T^2$ to indicate the composition $T \circ T$, i.e.\ $T^2 x = T(Tx)$.
If $W$ and $\bar{W}$ are finite dimensional Hilbert spaces, say $W \equiv \R^n$ and $\bar{W} \equiv \R^m$, we often identify the tensor product $W \otimes \bar{W}$ with the matrix space $\R^{n\times m}$.
Given any other finite dimensional Hilbert space $(V,|\cdot|_V)$, the following identifications hold: \begin{equation*}
    \mathscr{L}(W \otimes \bar{W},V) \equiv \mathscr{L}(W,\mathscr{L}(\bar{W},V)) \equiv \mathscr{L}^2(W \times \bar{W},V),
\end{equation*} where $\mathscr{L}^2(W \times \bar{W},V)$ denotes the space of bilinear and continuous maps from $W \times \bar{W}$ to $V$.

Given any function $F:W \to \bar{W}$, we denote by $|F|_\infty := \sup_{x \in W} |F(x)|_{\bar{W}}$. If $|F|_\infty<~+\infty$, the function $F$ is said to be bounded. Moreover, given $n \in \mathbb{N}$, we denote by ${C}^n_b(W;\bar{W})$ the space of bounded and continuous functions from $W$ to $\bar{W}$ which are $n$-times continuously differentiable with bounded derivatives.
If $F$ is differentiable in $W$, we consider its derivative as a map $DF:W \to \mathscr{L}(W,\bar{W})$. 
The space ${C}^n_b(W;\bar{W})$ is a Banach space if endowed with the norm \begin{equation*}
    |F|_{{C}_b^n} := |F|_\infty + |DF|_\infty + \dots + |D^nF|_\infty,
\end{equation*} 
where $D^nF$ denotes the $n$-th derivative of $F$.
Any function $F \in {C}^1_b(W,\bar{W})$ satisfies the following mean value theorem: \begin{equation*}
    F(x+h) - F(x) = \int_0^1 DF(x+\theta h) d\theta \  h
\end{equation*} for any $x,h \in W$, where the integral can be intended as a componentwise integral, namely \begin{equation*}
    F^i(x+h) - F^i(x) = \sum_{j=1}^n \big(\int_0^1 \frac{\partial F^i}{\partial x^j}(x+\theta h) d\theta \big) h^j \qquad i=1,\dots,n.
\end{equation*}
Similarly, for any $F \in {C}^2_b(W,\bar{W})$ and for any $x,h \in W$, \begin{equation*}
    F(x+h) - F(x) - DF(x) h = \frac{1}{2} \int_0^1 (1-\theta) D^2F(x +\theta h) d\theta \ h \otimes h.
\end{equation*}

Given a probability space $(\Omega,\mathcal{F},\mathbb{P})$ and an integrability esponent $p\in [1,\infty)$, we denote by $\|\cdot\|_{p;W}$ - or simply $\|\cdot\|_p$, if clear from the context - the usual norm on the Lebesgue space $L^p(\Omega;W)$ of $W$-valued random variables. A $W$-valued stochastic process $X=(X_t)_t$ is said to be $L^p$-integrable if any $X_t$ belongs to $L^p(\Omega;W)$. If $W=\R$, we sometimes write $L^p(\Omega)$ instead of $L^p(\Omega;\R)$. For a given sub-$\sigma$-field $\mathcal{G}\subseteq \mathcal{F}$, the space $L^p(\Omega,\mathcal{G};W)$ is the space of $\mathcal{G}$-measurable random variables in $L^p(\Omega;W)$.
The norm $\|\cdot\|_\infty$ denote the essential supremum norm. 

\subsection{Rough paths theory}
Let $(E,|\cdot|_E)$ be any Banach space. If clear from the context, we denote $|\cdot|_E$ simply by $|\cdot|$.
For any $T>0$ we define \begin{equation*}
    \Delta_{[0,T]} := \{ (s,t) \in [0,T]^2 \ | \ s\le t\}
\end{equation*} and \begin{equation*}
    \Delta^2_{[0,T]} := \{ (s,u,t) \in [0,T]^3 \ | \ s\le u \le t\}. 
\end{equation*}
For any path $[0,T]\ni t \mapsto X_t \in E$, we define its increment as the two parameter function
\begin{equation*}
    \delta X_{s,t} := X_t - X_s \qquad (s,t) \in \Delta_{[0,T]}.
\end{equation*}
Similarly, given a two-parameter function  $A:\Delta_{[0,T]} \to E$, $(s,t) \mapsto A_{s,t}$, we define \begin{equation*}
    \delta A_{s,u,t}:= A_{s,t} - A_{s,u} - A_{u,t} \qquad (s,u,t) \in \Delta_{[0,T]}^2. 
\end{equation*}

\begin{defn} Let $T>0$ and $\alpha \in (0,1]$. An $\alpha$-Hölder continuous path with values in $E$ is a map $X:[0,T] \to E$ for which there exists a constant $C>0$ such that \begin{equation*}
    |\delta X_{s,t}|_E \le C|t-s|^\alpha \qquad \text{for any $(s,t) \in \Delta_{[0,T]}$.}
\end{equation*} 
We write $X \in {C}^\alpha([0,T];E)$ and we denote by \begin{equation*}
    |\delta X|_{\alpha;E} := \sup_{0\le s < t \le T} \frac{|X_t-X_s|_E}{|t-s|^\alpha}.
\end{equation*}  
An $\alpha$-Hölder continuous two-parameter path with values in $E$ is a map $A:\Delta_{[0,T]} \to E$ for which there exists a constant $C>0$ such that \begin{equation*}
    |A_{s,t}|_E \le C|t-s|^\alpha \qquad \text{for any $(s,t) \in \Delta_{[0,T]}$.}
\end{equation*} 
We write $A \in {C}_2^\alpha(\Delta_{[0,T]};E)$ and we denote by \begin{equation*}
    |A|_{\alpha;E} := \sup_{\substack{(s,t)\in\Delta_{[0,T]}\\ s \ne t}} \frac{|A_{s,t}|_E}{|t-s|^\alpha}.
\end{equation*}
If $X \in {C}^\alpha([0,T];E)$, then $\delta X \in {C}_2^\alpha(\Delta{[0,T]};E)$. The converse is a priori not true.
If clear from the context, we write $|\cdot|_\alpha$ instead of $|\cdot|_{\alpha;E}$.
\end{defn}

\begin{defn} Let $\alpha \in (\frac13,\frac12]$.  An $\alpha$-Hölder rough path on $[0,T]$ with values in $\R^n$ is a pair $\mathbf{Z}=(Z,\mathbb{Z})$ such that \begin{enumerate}
    \item[a)] $Z \in {C}^\alpha([0,T];\R^n)$;
    \item[b)] $\mathbb{Z} \in {C}_2^{2\alpha}(\Delta_{[0,T]};\R^n \otimes \R^n) $;
    \item[c)] (Chen's relation) for any $(s,u,t) \in \Delta_{[0,T]}^2$, it holds \begin{equation*}
        \mathbb{Z}_{s,t}-\mathbb{Z}_{s,u}-\mathbb{Z}_{u,t} = \delta Z_{s,u}\otimes\delta Z_{u,t}
    \end{equation*} or, equivalently, $\mathbb{Z}_{s,t}^{ij}-\mathbb{Z}_{s,u}^{ij}-\mathbb{Z}_{u,t}^{ij} = \delta Z_{s,u}^i \delta Z_{u,t}^j$ for any $i,j=1,\dots,n$.
\end{enumerate}

We write $\mathbf{Z}=(Z,\mathbb{Z}) \in \mathscr{C}^\alpha([0,T];\R^n)$, and we define \begin{align*}
    |\mathbf{Z}|_\alpha := |\delta Z|_\alpha + |\mathbb{Z}|_{2\alpha}.
\end{align*}
\end{defn}

Given any pair of rough paths $\mathbf{Z}=(Z,\mathbb{Z}), \ \bar{\mathbf{Z}}=(\bar{Z},\bar{\mathbb{Z}}) \in \mathscr{C}^\alpha([0,T];\R^n)$,  we define their distance as \begin{equation*} \begin{aligned}
    \rho_{\alpha}(\mathbf{Z},\bar{\mathbf{Z}}) &:= |\delta Z -\delta\bar{Z}|_\alpha + |\mathbb{Z}-\bar{\mathbb{Z}}|_{2\alpha} =  \\ &= \sup_{0\le s<t\le T} \frac{|\delta Z_{s,t} - \delta \bar{Z}_{s,t}|_{\R^n}}{|t-s|^\alpha} + \sup_{0\le s<t\le T} \frac{|\mathbb{Z}_{s,t} -  \bar{\mathbb{Z}}_{s,t}|_{\R^n \otimes \R^n}}{|t-s|^{2\alpha}} .
\end{aligned}
\end{equation*}

Let us now consider a smooth path $Z \in {C}^\infty([0,T];\R^n)$. We denote by $\dot{Z}$ its time derivative. It is possible to show that such a path can be canonically lifted to a $\alpha$-rough path for any $\alpha \in (0,1]$, by defining  
\begin{equation*}
    \mathbb{Z}_{s,t} := \int_s^t \delta Z_{s,r} \otimes dZ_r = \int_s^t \delta Z_{s,r} \otimes \dot{Z}_r \ dr \qquad \text{for any $(s,t)\in\Delta_{[0,T]}$}.
\end{equation*}
We write $\mathbf{Z}=(Z,\mathbb{Z}) \in \mathscr{L}({C}^\infty([0,T];\R^n))$ and we call it a smooth rough path. 

\begin{defn} Let $\alpha \in (\frac13,\frac12]$. A weakly geometric $\alpha$-rough path with values in $\R^n$ is an element  $\mathbf{Z}=(Z,\mathbb{Z}) \in \mathscr{C}^\alpha([0,T];\R^n)$ such that, for any $(s,t) \in \Delta_{[0,T]}$, \begin{equation*}
Sym(\mathbb{Z}_{s,t}):=\frac{\mathbb{Z}_{s,t}+\mathbb{Z}_{s,t}^\top}{2} = \frac{1}{2}\delta Z_{s,t} \otimes \delta Z_{s,t}.
\end{equation*}
We write $\mathbf{Z}\in \mathscr{C}^\alpha_g([0,T];\R^n)$.
\end{defn}



\begin{defn} \label{def:geometricroughpaths}
Let $\alpha \in (\frac13,\frac12]$. A geometric $\alpha$-rough path with values in $\R^n$ is an element of the closure of $\mathscr{L}({C}^\infty([0,T];\R^n))$ in $(\mathscr{C}^\alpha([0,T];\R^n),\rho_\alpha)$. Namely, $\mathbf{Z}=(Z,\mathbb{Z})$ is a geometric rough path - and we write $\mathbf{Z} \in \mathscr{C}^{0,\alpha}_g([0,T];\R^n) $ - if and only if it belongs to $\mathscr{C}^\alpha([0,T];\R^n)$ and there exists a sequence $(\mathbf{Z}^n)_n \subseteq  \mathscr{L}({C}^\infty([0,T];\R^n))$ of smooth rough paths such that \begin{equation*}
    \rho_\alpha(\mathbf{Z}^n,\mathbf{Z}) \to 0 \qquad \text{as $n \to +\infty$.}
\end{equation*}     
\end{defn}

\subsection{Preliminaries on Malliavin Calculus} For a complete discussion about Malliavin calculus and its application to stochastic differential equations we refer to \cite{nualart2006malliavin}. 
The general framework to develop Malliavin calculus consists in a complete probability space $(\Omega,\mathcal{F},\mathbb{P})$, a separable real Hilbert space $(H,\langle \cdot,\cdot \rangle_H)$ and an isonormal Gaussian process on $H$. The latter
 is a family $\{W(h), h \in H\}$ of real-valued centered Gaussian random variables on $(\Omega,\mathcal{F},\mathbb{P})$ such that \begin{equation*}
        \E(W(h)W(g)) = \langle h,g \rangle_H 
    \end{equation*} 
for any $h,g \in H$. The mapping $H \ni h \mapsto W(h) \in L^2(\Omega)$ provides a linear isometry of $H$ onto a closed subspace of $L^2(\Omega,\mathcal{F},\mathbb{P})$. \\
We are interested in developing Malliavin calculus on a complete filtered probability space $(\Omega,\mathcal{F},\{\mathcal{F}_t\}_{t \in [0,T]},\mathbb{P})$, on which it is defined an $m$-dimensional Brownian motion \linebreak $B=((B^1_t,\dots,B^m_t))_{t \in [0,T]}$. 
Therefore it is classical to choose \begin{equation*}
    H=L^2(\mathcal{T},\mathscr{B},\mu),
\end{equation*} 
where $(\mathcal{T},\mathscr{B},\mu)$ is the finite measure space defined by 
\begin{equation*}
    \mathcal{T}:= [0,T] \times \{1,\dots,m\}, \quad \mathscr{B}:= \mathcal{B}([0,T])\otimes \mathcal{P}(\{1,\dots,m\}), \quad \mu:= \lambda_{|_{[0,T]}} \otimes \#
\end{equation*} with $\lambda$ denoting the one-dimensional Lebesgue measure, $\#$ denoting the counting measure on the discrete set $\{1,\dots,m\}$ and $\mathcal{B}([0,T])$ the Borel $\sigma$-field on $[0,T]$. Standard functional analysis results lead to the following identifications (i.e.\ the spaces are isometrically isomorphic): \begin{itemize}
    \item $L^2(\mathcal{T}) \equiv L^2([0,T];\R^m)$;
    \item $L^2(\Omega;L^2(\mathcal{T})) \equiv L^2([0,T]\times\Omega;\R^m)$;
    \item $L^2(\mathcal{T})^{\otimes{k}} \equiv L^2([0,T]^k; (\R^{m})^{k}), \ k\in \mathbb{N}_{\ge 1}$.
\end{itemize}
The scalar product on $L^2(\mathcal{T})$ can therefore be written as \begin{equation*}
    \langle h,g \rangle_{L^2(\mathcal{T})} = \int_{\mathcal{T}} h_\tau g_\tau \mu(d\tau) = \sum_{j=1}^m \int_0^T h_t^j g_t^j dt.
\end{equation*} and, for any $h\in L^2(\mathcal{T})$, we define 

\begin{equation*}
    W(h) := \sum_{j=1}^m \int_0^T h^j_r dB^j_r, 
\end{equation*} 
where the integral of a deterministic function on $[0,T]$ against a Brownian motion is intended as a Wiener integral. The map $W:L^2(\mathcal{T}) \to L^2(\Omega)$ satisfies all the requirements to be isonormal Gaussian process on $L^2(\mathcal{T})$. 
In this setting, for any fixed $p\in [1,\infty)$, the Malliavin derivative operator is a linear operator \begin{equation*}
    \mathbb{D}: \mathbb{D}^{1,p} \subseteq L^p(\Omega) \to L^p(\Omega;L^2([0,T];\R^m)).
\end{equation*}
The case $p=2$ is particularly interesting, since the Malliavin derivative of a random variable can be interpreted as a $\lambda$-a.e.\ defined stochastic process taking values in $\R^m$. Namely, for any  $F \in \mathbb{D}^{1,2}$, $\mathbb{D}F$ is an element of $L^2([0,T]\times \Omega;\R^m)$, and, for $\lambda$-almost every $t \in [0,T]$ and for any $k=1,\dots,m$, we write \begin{equation*}
    \mathbb{D}_t F := \mathbb{D}F (t,\omega).
\end{equation*} 
If we consider a random vector $X=(X^1,\dots,X^d)$ such that $X^i\in \mathbb{D}^{1,2}$ for any $i=1,\dots,d$, we often define its Malliavin derivative as \begin{equation*}
    \mathbb{D}X := (\mathbb{D}X^1,\dots,\mathbb{D}X^d)^T \in [L^2(\Omega \times [0,T];\R^m)]^d \equiv L^2(\Omega \times [0,T];\R^{d \times m})
\end{equation*}  
and we identify it with the $\R^{d\times m}$-valued process given by $\mathbb{D}_t X = ((\mathbb{D}_t X^i)^k)_{i=1,\dots,d; k=1,\dots,m}$. 
Similarly, the $k$-th order Malliavin derivative operator is defined as a map \begin{equation*}
    \mathbb{D}^k: \mathbb{D}^{k,p} \subseteq L^p(\Omega) \to L^p(\Omega;L^2([0,T]^k;(\R^{m})^k)) \quad k\in \mathbb{N}_{\ge 1}, \ p\in [1,\infty). 
\end{equation*}  
Given $F\in \mathbb{D}^{k,2}$, its Malliavin derivative $\mathbb{D}^kF$ can be seen as an element of \linebreak $L^2([0,T]^k\times \Omega; (\R^{m})^k)$ and we write \begin{equation*}
    \mathbb{D}^k_{t_1,\dots,t_k}F := \mathbb{D}^k F ((t_1,\dots,t_k),\omega) \qquad (t_1,\dots,t_k) \in [0,T]^k.
\end{equation*}
Given a random vector $X=(X^1,\dots,X^d)$ such that $X^i\in \mathbb{D}^{k,2}$ for any $i=1,\dots,d$, we consider its $k$-th Malliavin derivative as \begin{equation*}
    \mathbb{D}^k X := (\mathbb{D}^kX^1,\dots,\mathbb{D}^kX^d)^\top \in L^2([0,T]^k\times \Omega;\R^{d\times km}).
\end{equation*}

\subsection{Rough stochastic analysis}
In this section we summarize the main concepts and results on rough stochastic differential equations (RSDEs) from \cite{FHL21}. We begin by recalling the definition of stochastic $L^{p,q}$-norms that will be used throughout.

\begin{defn}[] Let $(\Omega,\mathcal{F},\mathbb{P})$ be a probability space and let $(E,|\cdot|_E)$ be a Banach space. Let $p \in [1,\infty)$ and let $\mathcal{G}\subseteq\mathcal{F}$ be a sub-$\sigma$-field. Given  an $E$-valued random variable $\xi$, we define - if it exists - its conditional $L^p$-norm with respect to  $\mathcal{G}$ as the (unique) $\mathcal{G}$-measurable $\R$-valued random variable given by \begin{equation*}
    \|\xi|\mathcal{G}\|_{p;E} := \E\left(|\xi|_E^p|\mathcal{G}\right)^{\frac{1}{p}}.
\end{equation*} 
We often abbreviate the notation as $\|\xi|\mathcal{G}\|_p$.  
In particular, for any $q\in[p,\infty]$, the mixed $L^{p,q}$-norm $\|\|\xi |\mathcal{G}\|_p\|_q$ denotes the $L^q(\Omega)$-norm of $\|\xi|\mathcal{G}\|_p$.
\end{defn}

Let us fix a complete filtered probability space $(\Omega,\mathcal{F},\{\mathcal{F}_t\}_{t \in [0,T]},\mathbb{P})$ and let $(W,|\cdot|)$ be a finite dimensional real Hilbert space.
{ We denote by $\E_t (\cdot) = \E( \cdot \mid \mathcal{F}_t)$.} 
We define the class of stochastic controlled paths, for which a meaningful notion of rough stochastic integration against a rough path exists (cf.\ \cite[Theorem 3.4]{FHL21}).

\begin{defn}\label{def:stochasticcontrolledpaths}
(\cite[Definition 3.1]{FHL21})
Let $p\in [1,\infty)$ and $\ q \in [p,\infty]$. Let $\gamma \in (0,1]$ and let $Z \in {C}^\alpha([0,T];\R^n)$ be an $\alpha$-Hölder continuous deterministic path, for a certain $\alpha \in (0,1]$.  
A pair $(X,X')$ is said to be a $L^{p,q}$-integrable $W$-valued \textit{stochastic controlled rough path} with $\gamma$-Hölder regularity - and we write $(X,X') \in \mathscr{D}_Z^{2\gamma} L^{p,q}(W)$ - if 
\begin{enumerate} \renewcommand{\labelenumi}{\alph{enumi})}
\item $X=(X_t)_{t \in [0,T]}$ is an adapted $W$-valued stochastic processes such that $\delta X$ is $L^p$-integrable and \begin{equation*}
    \|\delta X\|_{\gamma;p,q} := \sup_{0\le s < t\le T} \frac{\|\|\delta X_{s,t}|\mathcal{F}_s\|_p\|_q}{|t-s|^\gamma} < +\infty.
\end{equation*} 
In such a case, we often write $X\in {C}_2^{\gamma} L^{p,q}(W)$, and when $p=q$ we write simply $X\in {C}_2^{\gamma} L^{p}(W)$;
\item $X'=(X'_t)_{t \in [0,T]}$ is an adapted $\mathscr{L}(\R^n,W)$-valued stochastic process such that
\begin{equation*} 
    \sup_{t \in [0,T]} \|X'_t\|_q < +\infty \qquad \text{and} \qquad 
    \|\delta X'\|_{\gamma;p,q} := \sup_{0\le s < t\le T} \frac{\|\|\delta X'_{s,t}|\mathcal{F}_s\|_p\|_q}{|t-s|^{\gamma}} < +\infty;
\end{equation*}
\item denoting by $R^X_{s,t}=\delta X_{s,t} - X'_s \delta Z_{s,t} $ for any $(s,t)\in\Delta_{[0,T]}$, it holds that the process $\E_\cdot R^X=(\E_s(R^X_{s,t}))_{(s,t)\in\Delta_{[0,T]}}$ satisfies
\begin{equation*}
    \|\E_\cdot R^X\|_{2\gamma;q}:=\sup_{0\le s < t \le T} \frac{\|\E_s(R^X_{s,t})\|_q}{|t-s|^{2\gamma}} < + \infty.
\end{equation*} 
\end{enumerate}
The process $X'$ is called the (stochastic) Gubinelli derivative of $X$ (with respect to $Z$). In the case $p=q$, we simply write $(X,X') \in \mathscr{D}_Z^{2\gamma} L^{p}(W)$. We denote by \begin{align*}
    \|(X,X')\|_{\mathscr{D}_Z^{2\gamma}L^{p,q}} &:= \|\delta X\|_{\gamma;p,q} + \sup_{t \in [0,T]} \|X'_t\|_q + \|\delta X'\|_{\gamma;p,q} + \|\E_\cdot R^X\|_{2\gamma;q} \\
    \|(X,X')\|_{W;2\gamma;p,q} &:= \sup_{t \in [0,T]} \|X_t\|_q + \|(X,X')\|_{\mathscr{D}_Z^{2\gamma}L^{p,q}}.
\end{align*}
{ 
Given another $\bar Z \in C^\alpha([0,T];\R^n)$ and $(\bar X,\bar X') \in \mathscr{D}_{\bar Z}^{2\gamma} L^{p,q}(W)$ we define \begin{align*}
    \| X, X' ; \bar X, \bar X' \|_{Z,\bar Z; 2\gamma;p} :=&  \| \delta (X - \bar X)\|_{\gamma;p} + \sup_{t \in [0,T]} \|X_t' - \bar X'_t\|_p + \| \delta (X' - \bar X')\|_{\gamma;p} \\
    & + \| \E_\cdot R^X - \E_\cdot \bar R ^{\bar X}\|_{2\gamma ;p}
\end{align*}
where $\bar R ^{\bar X}_{s,t} := \delta \bar X_{s,t} - \bar X'_s \delta \bar Z_{s,t}$.
}
\end{defn}

{
Following \cite{FHL21}, given $\mathbf{Z}=(Z,\mathbb{Z}) \in \mathscr{C}^\alpha([0,T];\R^n)$ with $\alpha \in (\frac13,\frac12]$ and a stochastic controlled rough path $(X,X') \in \mathscr{D}_Z^{2\gamma}L^{p,q}(\mathscr{L}(\R^n,W))$ for some $\gamma \in [0,1]$ with $\alpha + \gamma > \frac12,$ $\alpha + 2\gamma > 1$, the rough stochastic integral $\int_0^\cdot (X_r,X'_r) d\mathbf{Z}_r = (\int_0^t (X_r,X'_r) d\mathbf{Z}_r )_{t \in [0,T]}$ is defined as the unique continuous and adapted stochastic process such that \begin{equation} \label{eq:roughstochasticintegral}
    \begin{aligned}
    \|\| \int_s^t (X_r,X'_r) d\mathbf{Z}_r - X_s \delta Z_{s,t} - X'_s \mathbb{Z}_{s,t} \mid \mathcal{F}_s \|_p \|_q \lesssim |t-s|^{\frac12 + \varepsilon} \\
    \|\E_s( \int_s^t (X_r,X'_r) d\mathbf{Z}_r - X_s \delta Z_{s,t} - X'_s \mathbb{Z}_{s,t} ) \|_q \lesssim |t-s|^{1 + \varepsilon}
\end{aligned}
\end{equation} 
for any $(s,t) \in \Delta_{[0,T]}$ and for some $\varepsilon >0$. Here

$$\int_s^t (X_r,X'_r) d\mathbf{Z}_r  := \int_0^t (X_r,X'_r) d\mathbf{Z}_r  - \int_0^s (X_r,X'_r) d\mathbf{Z}_r. $$ 
}
We collect some technical results on stability of stochastic controlled rough paths. 
\begin{lemma} \label{lemma:technicalconvergence}
   Let $(A^n)_n$ be a sequence of $L^p$-integrable two-parameter stochastic processes such that $\sup_n \|A^n\|_{\beta ; p, q} < + \infty$, and let $A = (A_{s, t})_{(s, t) \in\Delta_{[0, T]}}$ be $L^p$-integrable and such that \begin{equation*}
      \sup_{(s, t) \in \Delta_{[0, T]}} \|\|A^n_{s, t} - A_{s, t} |
     \mathcal{F}_s \|_p \|_q \to 0 \quad \text{as $n
     \to + \infty$} .
  \end{equation*}
  Then $\|A\|_{\beta ; p, q} \le \sup_n \|A^n \|_{\beta ; p, q}$ and, for any $\alpha < \beta$, \begin{equation*}
      \|A^n - A\|_{\alpha ; p, q} \to 0 \quad \text{as $n
  \to + \infty$.}
  \end{equation*}  
\end{lemma}

\begin{proof}
  Let $(s, t) \in \Delta_{[0, T]}$. Then we can write \begin{equation*}
      \|\|A_{s, t} | \mathcal{F}_s \|_p \|_q = \lim_{n \to + \infty} \|\|A^n_{s, t} | \mathcal{F}_s \|_p \|_q \le (\sup_n \|A^n \|_{\beta ; p, q}) | t - s |^{\beta}.
  \end{equation*} Moreover,
  \begin{align*}
    \|\|A^n_{s, t} - A_{s, t} | \mathcal{F}_s \|_p \|_q & = 
    \|\|A^n_{s, t} - A_{s, t} | \mathcal{F}_s  \|_p
    \|_q^{\frac{\alpha}{\beta}} \|\|A^n_{s, t} - A_{s, t} | \mathcal{F}_s  \|_p \|_q^{1 - \frac{\alpha}{\beta}}\\
    & \le (2 \sup_n \|A^n \|_{\beta ; p, q}  | t - s
    |^{\beta})^{\frac{\alpha}{\beta}}  \|\|A^n_{s, t} - A_{s, t} | \mathcal{F}_s \|_p \|_q^{1 - \frac{\alpha}{\beta}}\\
    & \lesssim  \left( \sup_n \|A^n \|_{\beta ; p, q}^{\frac{\alpha}{\beta}}\right) | t - s |^{\alpha}  \sup_{(s, t) \in \Delta_{[0, T]}} \|\|A^n_{s, t} - A_{s, t} | \mathcal{F}_s  \|_p \|_q^{1 - \frac{\alpha}{\beta}} .
  \end{align*}
\end{proof}

\begin{coroll} \label{coroll:technicalconvergence}
    For any $n\in \mathbb{N}$, let $\mathbf{Z}^n = (Z^n, \mathbb{Z}^n) \in \mathscr{C}^{\alpha}([0,T];\R^n)$ be a rough path and let $( Y^n, (Y^n)') \in \mathscr{D}^{2 \alpha}_{Z^n} L^{p}(\R^d)$ with
    \begin{equation*}
        \sup_n \| ( Y^n, (Y^n)')\|_{\mathscr{D}^{2\alpha}_{Z^n} L^{p}} < + \infty . 
    \end{equation*}
  Let $\mathbf{Z} = (Z, \mathbb{Z}) \in \mathscr{C}^{\alpha}([0,T];\R^n)$ be a rough path such that $\rho_{\alpha} (\mathbf{Z}^n, \mathbf{Z})\rightarrow 0$ as $n \rightarrow + \infty$.
  Let $(Y, Y') \in \mathscr{D}^{2
  \alpha}_Z L^{p, q}$ and assume $$\sup_{(s, t) \in \Delta_{[0, T]}} \|
  \delta Y^n_{s, t} - \delta Y_{s, t} \|_p
  \rightarrow 0 \quad \text{and} \quad \sup_{t \in [0, T]} \left\| (Y^n)'_t - Y_t' \right\|_p
  \rightarrow 0$$ as $n \rightarrow + \infty$. Then, for any $\gamma < \alpha$,
  \begin{equation*}
      \|\delta (Y^n-Y)\|_{\gamma;p} + \|\delta ((Y^n)'-Y')\|_{\gamma;p} + \|\E_\cdot (R^{n,Y^n}-R^Y)\|_{2\gamma;p} \to 0
  \end{equation*} as $n \to +\infty$, where $R^{n,Y^n}_{s,t}:=\delta Y^n_{s,t}-(Y^n)'_s\delta Z^n_{s,t}$.
\end{coroll} 
\begin{proof}   
    The convergence to zero of the first two addends comes straightforwardly from Lemma \ref{lemma:technicalconvergence}. In addition notice that, uniformly in $(s,t)\in \Delta_{[0,T]}$, \begin{equation*} \begin{aligned}
        &\|\E_s(\delta Y^n_{s,t} - (Y^n)'_s\delta Z^n_{s,t} - \delta Y_{s,t} + Y_s'\delta Z_{s,t})\|_p \\
        &\le \|\delta Y^n_{s,t}-\delta Y_{s,t}\|_p + \|Y'_s\|_p |\delta Z^n_{s,t}-\delta Z_{s,t}| + \|(Y^n)'_s-Y'_s\|_p |\delta Z^n_{s,t}| \rightarrow 0
    \end{aligned}
    \end{equation*} as $n\to +\infty$.    
\end{proof}

Let $B=(B_t)_{t\in [0,T]}$ be an $m$-dimensional Brownian motion on $(\Omega,\mathcal{F},\{\mathcal{F}_t\}_{t\in [0,T]},\mathbb{P})$ and let $\mathbf{Z}=(Z,\mathbb{Z})\in \mathscr{C}^\alpha([0,T];\R^n)$ for $\alpha\in (\frac{1}{3},\frac{1}{2}]$. 
Let $b:\R^d \to \R^d$ and $\sigma:\R^d \to \mathscr{L}(\R^m;\R^d)$ be two bounded functions, let $\beta: \R^d \to \mathscr{L}(\R^n;\R^d)$ and let $x_0 \in \R^d$ be given. We define solutions of
\begin{equation} \label{eq:RSDE}
dX_t=b(X_t)dt+\sigma(X_t)dB_t+\beta(X_t)d\mathbf{Z}_t, \qquad X_0 = x_0 \in \R^d.
\end{equation} 
as follows (cf.\ \cite[Definition 4.2]{FHL21}).

\begin{defn} \label{def:solutionRSDEs}
    Let $p\in [2,\infty), \ q \in [p,\infty]$ and let $x_0\in \R^d$. An $L^{p,q}$-integrable solution to \eqref{eq:RSDE}, starting from $x_0$, is a continuous and adapted $\R^d$-valued stochastic process satisfying the following: \begin{enumerate}
    \item[i)] $X_0=x_0$;
    \item[ii)] $(\beta(X),D\beta(X)\beta(X)) \in \mathscr{D}_Z^{2\alpha}L^{p,q}(\mathscr{L}(\R^n,\R^d))$;
    \item[iii)] there exist $C_1,C_2 \ge 0$ such that, for any $(s,t)\in \Delta_{[0,T]}$, \begin{equation*}
            \|\|X^\natural_{s,t}|\mathcal{F}_s\|_p\|_q \le C_1 |t-s|^{2\alpha} \quad \text{and} \quad \|\|\E_s(X^\natural_{s,t})|\mathcal{F}_s\|_p\|_q \le C_2 |t-s|^{3\alpha},
        \end{equation*} where \begin{equation*}
            X^\natural_{s,t} := \delta X_{s,t} - \int_s^tb(X_r)dr -\int_s^t \sigma(X_r)dB_r - \beta(X_s)\delta Z_{s,t} - D\beta(X_s)\beta(X_s)\mathbb{Z}_{s,t}.
        \end{equation*}
    \end{enumerate}
    It is equivalent to replace iii) with \begin{itemize}
    \item[iii')] $\mathbb{P}$-a.s.\ and for any $t\in [0,T]$, \begin{equation*}
        X_t = x +\int_0^tb(X_r)dr + \int_0^t \sigma(X_r)dB_r + \int_0^t (\beta(X_r),D\beta(X_r)\beta(X_r))d\mathbf{Z}_r. 
    \end{equation*}
    where the last integral on the right-hand side is a rough stochastic integral in the sense of \cite[Theorem 3.4]{FHL21}.
\end{itemize}
\end{defn}

We only consider deterministic initial conditions, since it is the proper setting to develop Malliavin calculus. 
In the case $q=\infty$ it is possible to prove the following existence-and-uniqueness result (see \cite[Theorem 4.6]{FHL21}). 
\begin{thm}
    Assume $b\in C^1_b(\R^d;\R^d), \ \sigma\in C^1_b(\R^d;\mathscr{L}(\R^m;\R^d))$, $\beta \in C^3_b(\R^d;\mathscr{L}(\R^n;\R^d)).$ Then, for any $p\in [2,\infty)$ and for any $x_0 \in \R^d$, there exists a unique $L^{p,\infty}$-integrable solution $X=(X_t)_{t\in [0,T]}$ to \eqref{eq:RSDE} starting from $x_0$. In particular, $(X,\beta(X))\in \mathscr{D}_Z^{2\alpha}L^{p,\infty}(\R^d)$.
\end{thm}
We refer to \cite{FHL21} to generalizations and further results, concerning for example a priori bounds and stability estimates.

\subsection{Rough It\^o formula for functions of polynomial growth}

Throughout this section we fix a weakly geometric rough path $\mathbf{Z} = (Z, \mathbb{Z})\in \mathscr{C}^{\alpha}_g([0,T],\R^n)$ with $\alpha \in (\frac{1}{3},\frac{1}{2}]$. 
We prove an It\^o-type formula for \textit{rough It\^o processes} and functions of polynomial growth. The idea is based on the rough It\^o's formula for $C^2$-bounded functions proved in \cite[Section 4.4]{FHL21}, which cannot be directly applied in our case, as we need a chain rule for the product of two rough It\^o processes, which is not bounded. 


As a first step to prove the rough It\^o formula we must show that the composition of a $C^2$ function with polynomial growth with a stochastic controlled rough path is still a stochastic controlled rough path.
\begin{prop}
\label{pro: chain rule C3 function}
    Let $U,W$ be finite dimensional Hilbert spaces. Let $f\in C^2(U;W)$ such that there exists $c>0$ and $q\geq 1$ such that for $k=1,2$,
    \begin{equation*}
        |D^kf(x)| \leq c (1+|x|^q).
    \end{equation*}
    Moreover, assume that, for every $p\geq 2$ ,  and that
    \begin{equation*}
    \label{eq: X stoch controlled and bounded}
        (X,X^{\prime})\in \mathscr{D}^{2\alpha}_{Z}L^p(U), \qquad
        \sup_{t\in[0,T]}\|X\|_{p} < \infty.
    \end{equation*}
    Then,
    \begin{equation*}
        (f(X), Df(X)X^{\prime}) \in \mathscr{D}^{2\alpha}_{Z}L^p(W)
    \end{equation*}
\end{prop}

\begin{proof}
    We start the proof by expanding the increment of $f(X_t)$ as follows
    \begin{equation*}
        \delta f(X)_{s,t} = Df(X_s)X_s^{\prime} \delta Z_{s,t} + R^{f}_{s,t},
    \end{equation*}
    where, for $R^{X}_{s,t}:= \delta X_{s,t} - X^{\prime}_s \delta Z_{s,t}$,
    \begin{equation*}
        R^f_{s,t} := Df(X_s)R^{X}_{s,t}
        + \int_{0}^{1} \left(Df(X_s + \theta \delta X_{s,t}) - Df(X_s) \right) \rd \theta \delta X_{s,t}.
    \end{equation*}
    We must prove that Definition \ref{def:stochasticcontrolledpaths} is satisfied by $f(X)$. Since $(X_t)_{t\in[0,T]}$ and $(X^{\prime}_t)_{t\in[0,T]}$ are $\mathcal{F}_t$-adapted, so are also the processes $(f(X_t))_{t\in[0,T]}$ and $(Df(X_t)X^\prime_t)_{t\in [0,T]}$.
    We show that $\delta f(X) \in C^{\alpha}_2 L^p(W)$ by using \eqref{eq: X stoch controlled and bounded} and applying Cauchy-Schwarz inequality
    \begin{align*}
        \|\delta f(X)_{s,t} \|_{p}
        & = \| \int_0^1 Df(X_s+\theta \delta X_{s,t})\rd \theta \delta X_{s,t}\|_{p}\\
        & \leq c \|(1+|X_s|^q+|\delta X_{s,t}|^q \|_{{2p}}
        \| \delta X_{s,t} \|_{2p}\\
        & \lesssim \sup_{u\in[0,T]}\|X_u\|_{{2pq}}^q \|X\|_{\alpha;2p} |t-s|^{\alpha}.
    \end{align*}
    This shows that $f(X)$ satisfies point a) in Definition \ref{def:stochasticcontrolledpaths}.
    With similar arguments one shows that 
    \begin{align*}
        \| \delta(Df(X)X^\prime)_{s,t}
        \|_{p}
        & \lesssim (1+\sup_{u\in[0,T]}\|X_u\|_{{2pq}}^q + \sup_{u\in[0,T]}\|X^{\prime}_u\|_{{2pq}}^q)
        (\|X^{\prime}\|_{\alpha;2p} + \|X\|_{\alpha;2p}) |t-s|^{\alpha}\\
        \sup_{t\in[0,T]}\|Df(X_t)X^{\prime}_t\|_{p} 
        & \lesssim (1+\sup_{u\in[0,T]}\|X_u\|_{{2pq}}^q) \sup_{u\in[0,T]}\|X^{\prime}_u\|_{{2p}}\\
        \|\mathbb{E}_sR^f_{s,t}\|_{p} 
        &\lesssim (1+\sup_{u\in[0,T]}\|X_u\|_{{2pq}}^q)
        (\|X\|_{\alpha;4p}^2+ \|\mathbb{E}_{\cdot}R^x\|_{2\alpha;2p}) |t-s|^{2\alpha}.
    \end{align*}
    This complete the proof that $f(X)$ satisfies Definition \ref{def:stochasticcontrolledpaths} points b) and c).
\end{proof}
We are now ready to introduce the definition of rough It\^o process on a finite dimensional Hilbert space $U$.
\begin{defn}
    \label{def: rough ito process}
    Let $\mathbf{Z} = (Z, \mathbb{Z})\in \mathscr{C}^{\alpha}_g([0,T],\R^n)$, with $\alpha \in (\frac{1}{3},\frac{1}{2}]$.
    Assume that
     \begin{align*}
         (b_t)_{t\in [0,T]}\in U,
         \qquad
         (\sigma_t)_{t\in [0,T]}\in\mathscr{L}(\R^m,U) 
     \end{align*}
     are progressive processes such that, for every $p\geq 2$,
     \begin{equation}
         \label{eq: boundedness in Lp coefficients}
         \sup_{t\in [0,T]}\|b_t\|_{p},
         \sup_{t\in [0,T]}\|\sigma_t\|_{p} < \infty.
     \end{equation}
     Moreover, for every $p\geq 2$, let $(X^{\prime},X^{\prime\prime}) \in \mathscr{D}_Z^{2\alpha} L^p(\mathscr{L}(\R^{n},U))$ such that
     \begin{equation}
     \label{eq: boundedness in Lp derivative}
          \sup_{t\in [0,T]}\|X^{\prime}_t\|_{p}
         < \infty.
     \end{equation}
     For every $p\geq 2$ assume that $X_0\in L^p(\Omega;U)$ is $\mathcal{F}_0$-measurable.
     
     We say that $(X_t)_{t\in[0,t]} \in U$ is a \textit{rough It\^o process} if it is a continuous and adapted process such that
     \begin{align*}
         X_t &= X_0 + \int_0^t b_s \rd s + \int_0^t\sigma_s \rd B_s + \int_0^t (X^{\prime}, X^{\prime\prime})_s \rd \mathbf{Z}_s.
     \end{align*}
\end{defn}

\begin{prop}
\label{pro: rough ito process is stoch controlled}
    If $X$ is a rough It\^o process, then
    $X\in \mathscr{D}_Z^{2\alpha} L^p(U)$ and
        \begin{equation}
        \label{eq: boundedness rough ito}
        \sup_{t\in[0,T]}\|X_t\|_{p} < \infty;
        \end{equation}
\end{prop}
    \begin{proof}
It follows from the definition of rough stochastic integral and keeping in mind assumption \eqref{eq: boundedness in Lp derivative} that $X\in \mathscr{D}_Z^{2\alpha} L^p(U)$. The bounds \eqref{eq: boundedness rough ito} follow from standard considerations using Jensen and Burkholder-Davis-Gundy inequalities on the drift and the diffusion terms and \cite[Lemma 2.12]{FHL21} on the rough integral.
\end{proof}
We are now ready to state and prove the rough It\^o formula for functions of polynomial growth.
\begin{prop}
     \label{prop: product formula}  
     Let $U$ be a finite dimensional Hilbert spaces and $f\in C^3(U;W)$ such that, there exists $q\geq1$ and $c>0$ such that, for $k=1,2,3$
    \begin{equation*}
        |D^kf(x)| \leq c(1+|x|^q),
        \qquad
        x\in U.
    \end{equation*}
    Assume that $X$ is a rough It\^o process on $U$.
    
    Then $(Y,Y^\prime) := (Df(X)X^{\prime}, D^2f(X)(X^{\prime})^2 + Df(X)X^{\prime})\in \mathscr{D}_Z^{2\alpha} L^p(\mathscr{L}(\R^n,W))$, for every $p\geq 2$, and the rough It\^o formula holds: For every $t\in [0,T]$, $\mathbb{P}$-a.s.,
        \begin{align}
        \label{eq: product formula}
        \begin{aligned}
            f(X_t) - f(X_0)
            = & \int_{0}^{t} Df(X_u) b_u \rd u
            + \int_{0}^{t} Df(X_u) \sigma_u \rd B_u
            + \int_{0}^{t} (Y,Y^{\prime})_u \rd \mathbf{Z}_u \\ & + \frac{1}{2} \int_{s}^{t} \operatorname{trace}(\sigma_u \sigma^{\top}_u D^2f(X_u)) \rd u.
                    \end{aligned}
        \end{align}
\end{prop}

\begin{proof}
    First note that $(Y,Y^{\prime}) \in \mathscr{D}_{Z}^{2\alpha}L^p(\mathscr{L}(\R^n,W))$ thanks to Proposition \eqref{pro: chain rule C3 function} applied to the function $U\times \mathscr{L}(\R^n,U) \ni (x,x^\prime) \mapsto Df(x)x^{\prime} \in \mathscr{L}(\R^n,W)$, which is $2$-times differentiable with derivatives of polynomial growth.

    We apply now Taylor expansion to compute the increment of $f(X)$. Let $0\leq s \leq t \leq T$, we have $\delta f(X)_{s,t}
        = A_{s,t} + R_{s,t}$,
        where
    \begin{align*}
        A_{s,t} &:= Df(X_s)\delta X_{s,t} + \frac{1}{2} D^2f(X_s) (\delta X_{s,t})^{\otimes 2}\\
        R_{s,t} &:= \int_{0}^{1}\int_{0}^{1}\int_0^1
            D^3f(X_s + \xi \theta \zeta \delta X_{s,t}) \rd \xi \theta \rd \theta  \rd \zeta (\delta X_{s,t})^{\otimes 3}.
    \end{align*}
    Using the polynomial growth property of $D^3 f$ and Cauchy-Schwarz inequality, we have that $R\in C^{3\alpha}L^2(W)$. Indeed, there exists $c>0$ such that
    \begin{equation*}
        \|R_{s,t}\|_{2}
        \le c (1+\sup_{u\in[0,T]}\|X\|_{{4q}}^q) \|\delta X\|_{\alpha;12}^3 |t-s|^{3\alpha}.
    \end{equation*}
    This immediately implies that
    \begin{equation}
    \label{eq: L2 limit of A}
        \delta f(X)_{s,t} = L^2-\lim_{n\to\infty} \sum_{[u,v]\in \pi_n} A_{u,v},
    \end{equation}
    where $\pi_n$ is a sequence of partitions of $[s,t]$ whose mesh goes to $0$ as $n\to \infty$.

    We now define $I_{t}$ as the right-hand side of \eqref{eq: product formula} evaluated for $s=0$ and $t=t$. By definition $(I_t)_{t\in [0,T]}$ is an $\mathcal{F}_t$-adapted stochastic process with continuous paths and $I_t \in~L^p(W)$ for every $t\in[0,T]$ and $p\geq 2$. We want to now apply the uniqueness part of the stochastic sewing lemma to see that $\delta I_{s,t}$ is the $L^2$ limit of the partial sums over $A_{s,t}$ and thus coincides with $\delta f(X)_{s,t}$ almost surely.
    Using the definition of $I$, $A$, $(Y,Y^\prime)$ expanding $\delta X_{s,t}$ and adding and subtracting the relevant terms, we have that
    \begin{align}
        \delta I_{s,t} - A_{s,t} 
        = &\int_{s}^{t} \delta Df(X)_{s,u} b_u \rd u
        + \int_{s}^{t} \delta Df(X)_{s,u} \sigma_u \rd B_u
        \label{eq: expansion 1}\\
        &+ \int_{s}^{t} (Y,Y^{\prime})_u\rd \mathbf{Z}_u
        - Y_s \delta Z_{s,t} - Y_s^{\prime} \mathbb{Z}_{s,t}
        \label{eq: expansion 2}\\
        & - Df(X_s) \left(\int_{s}^t (X^\prime, X^{\prime\prime})_u \rd \mathbf{Z}_u - X^{\prime}_s \delta Z_{s,t} - X^{\prime \prime}_s \mathbb{Z}_{s,t}
        \right)
        \label{eq: expansion 3}
        \\
        & -\frac{1}{2} D^2f(X_s) \left( (\delta X_{s,t})^{\otimes 2} - 2 (X^{\prime}_s)^2 \operatorname{Symm}(\mathbb{Z}_{s,t})\right)
        \label{eq: expansion 4}\\
         &+ \frac{1}{2} \int_{s}^{t}\operatorname{trace}(\sigma_u \sigma^{\top}_u D^2f(X_u)) \rd u.
         \label{eq: expansion 5}
    \end{align}
    By standard arguments of stochastic calculus one sees that $\|\eqref{eq: expansion 1}\|_{2} \lesssim |t-s|^{\frac{1}{2} + \alpha}$ and $\|\mathbb{E}_s\eqref{eq: expansion 1}\|_{2} \lesssim |t-s|^{1 + \alpha}$. It is immediate to notice that terms \eqref{eq: expansion 2} and \eqref{eq: expansion 3} are rough stochastic integrals minus their local expansions. By \eqref{eq:roughstochasticintegral}, we immediately have
    \begin{equation*}
        \|\eqref{eq: expansion 2}\|_{2} \lesssim |t-s|^{2\alpha}, \qquad \|\mathbb{E}_s\eqref{eq: expansion 2}\|_{2} \lesssim |t-s|^{3\alpha},
    \end{equation*}
    and similar bounds hold for \eqref{eq: expansion 3}, using Hölder inequality.

    For term \eqref{eq: expansion 4} we have
    \begin{align*}
        (\delta X_{s,t})^{\otimes 2} - 2 (X^{\prime}_s)^2 \operatorname{Symm}(\mathbb{Z}_{s,t})
        &= (\delta X_{s,t})^{\otimes 2} - (X^{\prime}_s)^2 (\delta Z_{s,t})^{\otimes 2}\\
        & = \left(\int_s^t \sigma_u \rd B_u\right)^{\otimes 2} + \left(\int_s^t X^{\prime}_u \rd \mathbf{Z}_u \right)^{\otimes 2}
        + \bar{R}_{s,t} - (X^{\prime}_s)^{2}(\delta Z_{s,t})^{\otimes 2},
    \end{align*}
    where $\bar{R}_{s,t}$ is such that
    \begin{equation*}
        \|\bar{R}_{s,t}\|_{2} \lesssim |t-s|^{2\alpha},
        \qquad
        \|\mathbb{E}_s\bar{R}_{s,t}\|_{2} \lesssim |t-s|^{3\alpha}.
    \end{equation*}
    Using again H\"older inequality and the fact that $X,X^\prime$ have all moments bounded uniformly in time, we obtain
    \begin{align*}
        &\|\left(\int_s^t X^{\prime}_u \rd \mathbf{Z}_u \right)^{\otimes 2}
         - (X^{\prime}_s)^{2}(\delta Z_{s,t})^{\otimes 2}\|_{2}\\
        & \qquad \leq
        \|\left(\int_s^t \delta X^{\prime}_{s,u} \rd \mathbf{Z}_u \right)\otimes \left(\int_s^t \delta X^{\prime}_{u} \rd \mathbf{Z}_u \right)\|_{2}
         + \| X^{\prime}_s\delta Z_{s,t}\otimes 
         \left(\int_s^t \delta X^{\prime}_{s,u} \rd \mathbf{Z}_u
         \right)\|_{2}
         \\
         & \qquad \lesssim |t-s|^{3 \alpha}.
    \end{align*}
    In order to have the required bounds on $\eqref{eq: expansion 4}+\eqref{eq: expansion 5}$, we are only left with computing the following
    \begin{align*}
        S_{s,t} := &\int_s^t \operatorname{trace}(\sigma_u\sigma_u^{\top} D^2 f(X_u)) \rd u - D^2f(X_s)\left(\int_s^t \sigma_u \rd B_u\right)^{\otimes 2}
    \end{align*}
    From standard stochastic calculus, the polynomial growth of $D^kf$ and the assumptions on the moments of $X$, one readily obtains
    \begin{equation*}
        \|\mathbb{E}_s S_{s,t}\|_{2} \lesssim |t-s|^{1+\alpha},
        \qquad
        \|S_{s,t}\|_{2}
        \lesssim |t-s|.
    \end{equation*}
    Putting all previous estimates together, we have that there exists $\epsilon_1, \epsilon_2 > 0$ such that
    \begin{equation*}
        \|\mathbb{E}_s \delta I_{s,t}-A_{s,t}\|_{2} \lesssim |t-s|^{1+\epsilon_1},
        \qquad
        \|\delta I_{s,t}-A_{s,t}\|_{2}
        \lesssim |t-s|^{1+\epsilon_2}.
    \end{equation*}
    By the uniqueness part of the stochastic sewing lemma \cite[Theorem 2.9]{FHL21}
    we have that $\delta I_{s,t}$ coincides with the $L^2$ limit in \eqref{eq: L2 limit of A} and it is almost surely equal to $\delta f(X)_{s,t}$.
\end{proof}

\section{Linear rough stochastic differential equations}
\label{sec: linear rough stochastic equations}
In this section we develop a theory of RSDEs with linear coefficients on a complete filtered probability space $(\Omega,\mathcal{F},\{\mathcal{F}_t\}_{t \in [0,T]}, \mathbb{P})$. We consider an $m$-dimensional Brownian motion $B=(B_t)_{t \in [0,T]}$ on $(\Omega,\mathcal{F},\{\mathcal{F}_t\}_{t \in [0,T]}, \mathbb{P})$ and a deterministic rough path $\mathbf{Z}=(Z,\mathbb{Z})$ in $\mathscr{C}^{\alpha}([0,T];\R^n)$, with $\alpha \in (\frac{1}{3},\frac{1}{2}]$. We assume that solutions to linear RSDEs live in the finite dimensional real Hilbert spaces which we denote by $W$ and $\bar{W}$, typically they are $\R^d$ or $\R^{d\times m}$ for integers $d$ and $m$.

\begin{defn}\label{def:stochasticlinearvectorfields}
A stochastic linear vector field from $W$ to $\bar{W}$ is a progressively measurable map $Q:\Omega \times [0,T] \times W \to \bar{W}$ such that \begin{itemize}
    \item for $\mathbb{P}$-almost every $\omega$ and for any $t \in [0,T]$, the mapping $y \mapsto Q_t y := Q(\omega,t,y)$ is linear (and continuous);
    \item there is a constant $\|Q\|_{\infty} >0$ such that \begin{equation*}
       \esssup_{\omega} \sup_{t \in [0,T]} | Q_t |_{\mathscr{L}(W,\bar{W})} \le \|Q\|_\infty <+ \infty.  \end{equation*}
\end{itemize} 
We often denote by $\|Q\|_{\mathscr{L};p} := \sup_{t \in [0,T]} \| | Q_t |_{\mathscr{L}(W,\bar{W})} \|_p$, for any $p \in [1,\infty)$.
\end{defn}

\begin{defn}[Stochastic controlled linear vector fields] \label{def:stochasticcontrolledroughpaths_chapter2} 
Let $Z \in {C}^\alpha([0,T];\R^n)$ be a deterministic $\alpha$-Hölder continuous path with $\alpha \in (0,1]$. Let $p\in[2,\infty)$, $q \in [p,\infty]$ and let $\gamma\in(0,1]$. A pair $(f,f')$ is said to be a $L^{p,q}$-integrable stochastic controlled linear vector field from $W$ to $\bar{W}$ with $\gamma$-H\"older regularity - and we write $(f,f') \in \mathbf{D}^{2\gamma}_ZL^{p,q} \mathscr{L} (W,\bar{W})$ or simply $(f,f') \in \mathbf{D}^{2\gamma}_ZL^{p,q} \mathscr{L}$ - if \begin{itemize}
    \item[a)] $f:\Omega \times [0,T] \times W \to \bar{W}$ is a stochastic linear vector field with \begin{equation*} \label{eq:normofcontrolledvectorfields1}
        \|\delta f\|_{\gamma;p,q}:=\sup_{0\le s<t\le T} \frac{\|\||f_t-f_s|_{\mathscr{L}(W,\bar{W})}|\mathcal{F}_s\|_p\|_{q}}{|t-s|^{\gamma}}<+\infty;
    \end{equation*}
    \item[b)] $f':\Omega \times [0,T] \times W \to \mathscr{L}(\R^n, \bar{W})$ is a stochastic linear vector field with \begin{equation*} \label{eq:normofcontrolledvectorfields2}
        \|\delta f'\|_{\gamma;p,q}:=\sup_{0\le s<t\le T} \frac{\|\||f'_t-f'_s|_{\mathscr{L}(W,\mathscr{L}(\R^n,\bar{W}))}|\mathcal{F}_s\|_p\|_{q}}{|t-s|^{\gamma}}<+\infty;
    \end{equation*}
    \item[c)] denoting by $R_{s,t}^f:=f_t-f_s-f'_s\delta Z_{s,t}$ for any $(s,t)\in \Delta_{[0,T]}$, it holds \begin{equation*} \label{eq:remainderofstochasticcontrolledlinearvectorfields} \|\E_{\cdot}R^f\|_{2\gamma;q}:=\sup_{0\le s<t\le T} \frac{\||\E_s(R_{s,t}^f)|_{\mathscr{L}(W,\bar{W})}\|_{q}}{|t-s|^{2\gamma}}<+\infty.
    \end{equation*}  
\end{itemize} 
For such a stochastic controlled linear vector field we write \begin{align*}
    \|(f,f')\|_\infty &:= \|f\|_\infty + \|f'\|_\infty \\ 
    \|(f,f')\|_{\mathbf{D}_Z^{2\gamma}L^{p,q} \mathscr{L}} &:= \|\delta f\|_{\gamma;p,q} + \|\delta f'\|_{\gamma;p,q} + \|\E_{\cdot}R^f\|_{2\gamma;q} \\
\end{align*}
{ Given another $\bar Z \in C^\alpha([0,T];\R^n)$ and $(\bar f,\bar f') \in \mathbf{D}^{2\gamma}_{\bar Z}L^{p,q} \mathscr{L}(W,\bar W)$ we denote by \begin{align*}
    \| f, f' ; \bar f , \bar f' \|_{\mathscr{L};p} &:= \| f - \bar f\|_{\mathscr{L};p} + \| f' - \bar f'\|_{\mathscr{L};p} \\
    \| f, f' ; \bar f , \bar f' \|_{Z, \bar Z; 2\gamma;p;\mathscr{L}} &:= \| \delta (f - \bar f) \|_{\gamma;p} + \| \delta (f' - \bar f') \|_{\gamma;p} + \| \E_\cdot R^f - \E_\cdot \bar R^{\bar f} \|_{2\gamma;p}
\end{align*}
where $\bar R ^{\bar f}_{s,t} := \bar f_t - \bar f_s - \bar f'_s \delta \bar Z_{s,t}$.
}
\end{defn}

The following is a stability result for stochastic controlled rough paths under the composition with stochastic controlled linear vector fields.

\begin{prop} \label{prop:stabilityundercompositionlinearvectorfields2} 
Let $Z \in {C}^\alpha([0,T];\R^n)$ and let $p \in [2,\infty), \ \gamma \in (0,\alpha]$. 
Consider a stochastic controlled rough path $(Y,Y') \in \mathscr{D}_Z^{2\gamma}L^{p,q}(W)$ with $\sup_{t \in [0,T]} \|Y_t\|_p < +\infty$ and let $(f,f')\in  \mathbf{D}^{2\gamma}_ZL^{p,\infty}\mathscr{L}(W,\bar{W})$. Then \begin{equation*}
    (f,f')Y := (f Y,f' Y + f Y') \in \mathscr{D}_Z^{2\gamma}L^{p}(\bar{W})
\end{equation*} and \begin{equation*}
    \|(f,f')Y\|_{\mathscr{D}_Z^{2\gamma}L^p} \le \|(Y,Y')\|_{W;2\gamma;p} \left(\|(f,f')\|_\infty + \|(f,f')\|_{\mathbf{D}_Z^{2\gamma}L^{p,\infty} \mathscr{L}}\right).
\end{equation*}
\end{prop}

\begin{proof} For any $(s,t) \in \Delta_{[0,T]}$, $f_tY_t$ and $f'_tY_t+f_tY'_t$ are $\mathcal{F}_t$-measurable and
\begin{equation} \label{eq:esttimateforapriori1}
    \|f'_tY_t+f_tY'_t\|_p \le \|f'\|_\infty \sup_{t\in [0,T]}\|Y_t\|_p + \|f\|_\infty \sup_{t\in [0,T]}\|Y'_t\|_p < +\infty.
\end{equation} 
It is also easy to verify that \begin{equation}\begin{aligned} \label{eq:esttimateforapriori2}
     \|\delta(f_\cdot Y_\cdot)_{s,t} \|_p &\le \|f_t\delta Y_{s,t} \mid \mathcal{F}_s \|_p + \|\delta f_{s,t} Y_s\|_p \\ &\le \big(\|f\|_\infty \|\delta Y\|_{\gamma;p}  + \|\delta f\|_{\gamma;p,\infty} \sup_{t\in [0,T]}\|Y_t\|_p\big) |t-s|^\gamma
\end{aligned}
\end{equation} and, similarly, \begin{equation} \label{eq:esttimateforapriori3}
    \begin{aligned}
        \|\delta(f'_\cdot Y_\cdot + f_\cdot Y'_\cdot)_{s,t}\|_p 
        \le & \Big(\|f'\|_\infty \|\delta Y\|_{\gamma;p} + \|\delta f'\|_{\gamma;p,\infty} \sup_{t\in [0,T]}\|Y_t\|_p +  \|f\|_\infty \|\delta Y'\|_{\gamma;p}  \\ & + \|\delta f\|_{\gamma;p,\infty} \sup_{t\in [0,T]}\|Y'_t\|_p\Big) |t-s|^\gamma.
    \end{aligned}
\end{equation} Moreover, by linearity we can write \begin{equation} \label{eq:remaindercompositionwithstochasticcontrolledlinearvectorfields}
    \begin{aligned}
        R_{s,t}^{f_\cdot Y_\cdot}&:= f_tY_t-f_sY_s-(f'_s Y_s+f_sY'_s)\delta Z_{s,t} \\ 
        &= f_tY_t-f_sY_t+f_s(Y_t-Y_s-Y'_s\delta Z_{s,t})-f'_sY_s\delta Z_{s,t} \\  
        &= \delta f_{s,t} \delta Y_{s,t} +f_sR^Y_{s,t} +R^{f}_{s,t}Y_s.
    \end{aligned}
\end{equation}
This quite simple computation is one of the crucial points where linearity really makes a difference, with respect to the case of continuous and bounded coefficients treated in \cite{FHL21}.
Let $\bar{p} \in (1,2]$ be the conjugate exponent of $p$. By means of the conditional Hölder inequality we deduce that \begin{equation*}
    \begin{aligned}
        \|\E_s(\delta f_{s,t}\delta Y_{s,t})\|_p &\le \E(\E_s(|\delta f_{s,t}\delta Y_{s,t}|)^p)^{\frac{1}{p}} \le \E(\|\delta f_{s,t}|\mathcal{F}_s\|_{\bar{p}}^p \|\delta Y_{s,t}|\mathcal{F}_s\|_p^p)^{\frac{1}{p}}  \\
        &\le \|\|\delta f_{s,t}|\mathcal{F}_s\|_{\bar{p}}\|_\infty \|\delta Y_{s,t}\|_p \le \|\delta f\|_{\gamma;p,\infty}\|\delta Y\|_{\gamma;p} |t-s|^{2\gamma},
    \end{aligned}
\end{equation*} 
and we also have \begin{align*}
    \|\E_s(f_sR^Y_{s,t})\|_p &\le  \|f\|_\infty \|\E_\cdot R^Y\|_{2\gamma;p} |t-s|^{2\gamma} \\ 
    \|\E_s(R^{f}_{s,t}Y_s)\|_p &\le  \|\E_\cdot R^f\|_{2\gamma;p} \sup_{t\in [0,T]}\|Y_t\|_p |t-s|^{2\gamma}.
\end{align*}
\end{proof}

\begin{rmk}
    From the proof of Proposition \ref{prop:stabilityundercompositionlinearvectorfields2}, it is clear that we have stability of the composition only if we take a $L^{p,q}$-integrable stochastic controlled linear vector field with $q=\infty$. Otherwise, due to the application of the (possibly conditional) Hölder inequality, we lose some integrability and we can only conclude that $(f,f')Y \in \mathscr{D}_Z^{2\gamma}L^\frac{p}{2}(\bar{W})$.
\end{rmk}

\begin{prop} \label{prop:stabilityundercomposition}
Let $Z,\bar{Z} \in C^\alpha([0,T];\R^n)$, with $\alpha \in (0,1]$, let $p \in [2,\infty)$, $\gamma \in (0,\alpha]$. \linebreak
Let $(Y,Y') \in \mathscr{D}_Z^{2\gamma}L^{2p}(W)$ with $\|Y_0\|_{2p} <+\infty$ and let $(f,f') \in \mathbf{D}_Z^{2\gamma}L^{2p}\mathscr{L}(W,\bar W)$.
Let $(\bar{Y},\bar{Y}') \in \mathscr{D}_{\bar{Z}}^{2\gamma}L^{2p}(W)$ with $\|\bar{Y}_0\|_{2p} <+\infty$ and let $(\bar{f},\bar{f}') \in \mathbf{D}_{\bar{Z}}^{2\gamma}L^{2p}\mathscr{L}(W,\bar W)$.
Let $M>0$ be any constant such that \begin{equation*}
    \begin{aligned}
        &\| (f,f') \|_{\mathscr{L};2p} + \| (f,f') \|_{Z;\delta,\delta';2p;\mathscr{L}} + \| (\bar{f},\bar{f}') \|_{\mathscr{L};2p} + \| (\bar{f},\bar{f}') \|_{\bar{Z};\delta,\delta';2p;\mathscr{L}} + \\
        &+ \|Y_0\|_{2p} + \|(Y,Y')\|_{\mathbf{D}_Z^{\gamma,\gamma'}L_{2p}} + \|\bar{Y}_0\|_{2p} +\|(\bar{Y},\bar{Y}')\|_{\mathbf{D}_{\bar{Z}}^{\gamma,\gamma'}L_{2p}} \le M.
    \end{aligned}
\end{equation*}
Denote by $(\phi,\phi')=(fY,f'Y+fY')$ and, similarly, $(\bar{\phi},\bar{\phi}')=(\bar{f}\bar{Y},\bar{f}'\bar{Y}+\bar{f}\bar{Y}')$.

\noindent Then it holds that $(\phi,\phi') \in \mathscr{D}_Z^{2\gamma}L^p(\bar W)$, $ (\bar{\phi},\bar{\phi}') \in \mathscr{D}_{\bar{Z}}^{2\gamma}L^p(\bar W)$, and
\begin{align*}
        \|\phi,\phi';\bar{\phi},\bar{\phi}'\|_{Z,\bar{Z};2\gamma;p} \lesssim_M &  \| f,f';\bar{f},\bar{f}' \|_{\mathscr{L};2p} + \| f,f';\bar{f},\bar{f}' \|_{Z,\bar{Z};2\gamma;2p;\mathscr{L}} \\*
        & + \sup_{t \in [0,T]} \|Y_t - \bar Y_t\|_{2p} + 
        \|Y,Y';\bar{Y},\bar{Y}'\|_{Z,\bar{Z};2\gamma;2p}.
    \end{align*}
\end{prop}

\begin{proof}
    The fact that both $(\phi,\phi')$ and $(\bar{\phi},\bar{\phi}')$ are $L^p$-integrable stochastic controlled rough paths follows from Proposition \ref{prop:stabilityundercompositionlinearvectorfields2}. 
    Let $(s,t) \in \Delta_{[0,T]}$. 
    For sake of compact notation we denote by $\tilde{Y}=Y-\bar{Y}$ and $\tilde{f}=f-\bar{f}$.
    One trivially has that \begin{equation*}
        f_tY_t -\bar{f}_t\bar{Y}_t =  f_t\tilde{Y}_t + \tilde{f}_t\bar{Y}_t
    \end{equation*} and \begin{equation*}
        f_t'Y_t+f_tY_t' - \bar{f}_t'\bar{Y}_t-\bar{f}_t\bar{Y}'_t = f_t' \tilde{Y}_t + \tilde{f}'_t \bar{Y}_t + f_t \tilde{Y}'_t + \tilde{f}_t \bar{Y}_t.
    \end{equation*}
    Moreover, we can write \begin{equation*}
        \delta (f Y-\bar{f} \bar{Y})_{s,t} = \delta (f \tilde{Y}+\tilde{f} \bar{Y})_{s,t} =  f_t \delta \tilde{Y}_{s,t} + \delta f_{s,t}\tilde{Y}_s+\tilde{f}_t\delta \bar{Y}_{s,t}+\delta \tilde{f}_{s,t}\bar{Y}_s
    \end{equation*}
    and \begin{equation*}
        \begin{aligned}
        &\delta (f' Y + f Y' - \bar{f}' \bar{Y} - \bar{f} \bar {Y}' )_{s,t} = \delta (f' \tilde{Y} +  \tilde{f}'\bar{Y}+ f \tilde{Y}' + \tilde{f} \bar{Y}')_{s,t}  \\ 
        &\quad = f'_t\delta\tilde{Y}_{s,t} + \delta f'_{s,t}\tilde{Y}_s+\tilde{f}_t'\delta \bar{Y}_{s,t} + \delta \tilde{f}'_{s,t}\bar{Y}_s+f_t\delta\tilde{Y}'_{s,t}+\delta f_{s,t}\tilde{Y}'_s + \tilde{f}_t \delta \bar{Y}_{s,t}' + \delta \tilde{f}_{s,t} \bar{Y}'_s.
    \end{aligned}
    \end{equation*}
    By linearity, it is possible to write \begin{equation*}
        \begin{aligned}
            &{R}_{s,t}^{{f} {Y}} - \bar{R}_{s,t}^{\bar{f} \bar{Y}} = \delta (f Y)_{s,t} - (f'_sY_s+f_sY'_s)\delta{W}_{s,t} - \delta (\bar{f} \bar{Y})_{s,t} + (\bar{f}_s\bar{Y}_s+\bar{f}_s\bar{Y}'_s)\delta{\bar{W}}_{s,t}  \\ 
            &\quad = \delta f_{s,t} \delta\tilde{Y}_{s,t} + f_s (R^Y_{s,t}-\bar{R}^{\bar{Y}}_{s,t}) + (R^f_{s,t} - \bar{R}^{\bar{f}}_{s,t}) \bar{Y}_s + R^f_{s,t} \tilde{Y}_s + \delta \tilde{f}_{s,t} \delta \bar{Y}_{s,t} + \tilde{f}_s \bar{R}^{\bar{Y}}_{s,t}.
        \end{aligned}
    \end{equation*}
    The conclusion follows by applying H\"older's inequality to all the previous identities. As an example, we notice that \begin{equation*}
        \|\E_s(\delta f_{s,t} \delta\tilde{Y}_{s,t})\|_p \le \|\delta f_{s,t}\|_{2p} \|\delta\tilde{Y}_{s,t}\|_{2p} \lesssim_M \|Y,Y';\bar{Y},\bar{Y}'\|_{Z,\bar{Z};2\gamma;2p} |t-s|^{2\gamma}.
    \end{equation*}
\end{proof}

The following lemma gives an example on how it is possible to construct stochastic linear vector fields, starting from solutions of RSDEs.

\begin{lemma} \label{lemma:consistencyandaprioriforstochasticlinearvectorfields}
    Let $\mathbf{Z}=(Z,\mathbb{Z})\in\mathscr{C}^\alpha([0,T];\R^n)$, with $\alpha\in (\frac{1}{3},\frac{1}{2}]$, be such that $|\mathbf{Z}|_\alpha \le C$ for some $C>0$, and let $b\in {C}^1_b(\R^d;\R^d)$, $\sigma\in C^1_b(\R^d;\R^{d\times m})$ and $\beta \in {C}^3_b(\R^d;\R^{d\times n})$. Let $X=(X_t)_{t\in [0,T]}$ be the solution of the following RSDE, as introduced in Definition \ref{def:solutionRSDEs}: \begin{align*}
        dX_t = b(X_t)dt + \sigma(X_t)dB_t + (\beta(X_t),D\beta(X_t)\beta(X_t))d\mathbf{Z}_t.
    \end{align*}
    Then $(Db(X_t))_{t \in [0,T]}$ and $(D\sigma(X_t))_{t \in [0,T]}$ are stochastic linear vector fields from $\R^d$ to $\R^d$ and $\mathscr{L}(\R^m,\R^d)$, respectively. Moreover, for any $p \in [2,\infty)$, \begin{equation*}
     (f, f') := (D\beta(X),D^2\beta(X)\beta(X)) \in \mathbf{D}^{2\alpha}_ZL^{p,\infty}\mathscr{L}(\R^d,\R^{d \times n})
    \end{equation*} and there exists a positive constant $K_C=K_C(|b|_\infty,|\sigma|_\infty,|\beta|_{{C}^3_b},C,\alpha,p,T)$ such that  \begin{equation} \label{eq:estimateinaprioriforlinearvectorfields}
        \|(D\beta(X),D^2\beta(X)\beta(X))\|_{ \mathbf{D}^{2\alpha}_ZL^{p,\infty}\mathscr{L}} \le K_C.
    \end{equation}
\end{lemma}
\begin{proof} 
    The first part of the statement easily follows from the assumptions on $b$ and $\sigma$ and from the fact that $X$ is continuous and adapted.
    By definition of Fréchet derivative, \linebreak $D\beta: \R^d \to \mathscr{L}(\R^d,\R^{d \times n})$, while $D^2\beta$ is a map from $\R^d$ to $\mathscr{L}(\R^d,\mathscr{L}(\R^d,\R^{d \times n})$. 
    Being $X$ continuous and adapted, we have that $f$ and $f'$ are progressively measurable. Recall that, from the general theory about RSDEs, $(X,\beta(X)) \in \mathscr{D}_Z^{2\alpha}L^{p,\infty}(\R^d)$ for any $p \in [2,\infty)$. For any $t \in [0,T]$ and $\mathbb{P}$-almost surely, we have $|f_t|\le|D\beta|_\infty$ and \begin{equation*} \label{eq:infinitynormoff'}|f'_t|=|D^2\beta(X_t)\beta(X_t)| \le |D^2\beta|_\infty |\beta|_\infty < +\infty.
    \end{equation*} 
    Let now $(s,t) \in \Delta_{[0,T]}$. Then,  \begin{equation*}
        \begin{aligned}
            \|f_t-f_s|\mathcal{F}_s\|_p &= \|D\beta(X_t)-D\beta(X_s)|\mathcal{F}_s\|_p 
            \\
            &\le |D^2\beta|_\infty \|\|\delta X_{s,t}|\mathcal{F}_s\|_p\|_\infty \le |D^2\beta|_\infty \|\delta X\|_{\alpha;p,\infty}|t-s|^\alpha \qquad \text{$\mathbb{P}$-a.s.},
        \end{aligned}
    \end{equation*} and, similarly, \begin{equation*}
        \begin{aligned}
            \|f'_t-f'_s|\mathcal{F}_s\|_p 
            &\le \big(|D^2\beta|_\infty|\beta|_\infty  + |D^3\beta|_\infty|\beta|_\infty\big) \|\delta X\|_{\alpha;p,\infty} |t-s|^\alpha \qquad \text{$\mathbb{P}$-a.s.}
        \end{aligned}
    \end{equation*}  
    The previous estimates prove that both $\|\delta f\|_{\alpha;p,\infty}$ and $\|\delta f'\|_{\alpha;p,\infty}$ are finite. By linearity and applying again the mean value theorem, we can write
    
    \begin{equation*}
        \begin{aligned}
            R^f_{s,t} &= D\beta(X_t)-D\beta(X_s) - D^2\beta(X_s)X'_s \delta Z_{s,t} \\
            &= D\beta(X_t)-D\beta(X_s) - D^2\beta(X_s)\delta{X}_{s,t} + D^2\beta(X_s)(\delta{X}_{s,t}- X'_s \delta Z_{s,t}) \\ 
            &= \frac{1}{2}  \int_0^1 (1-\theta) D^3\beta(X_s +\theta \delta{X}_{s,t}) d\theta (\delta{X}_{s,t})^{\otimes 2}  + D^2\beta(X_s)R^X_{s,t}. 
        \end{aligned}
    \end{equation*} 
    By applying the conditional Hölder inequality we have \begin{equation*}
        \begin{aligned}
            |\E_s(\int_0^1 (1-\theta) D^3\beta(X_s +\theta \delta{X}_{s,t}) d\theta (\delta{X}_{s,t})^{\otimes 2} ) | 
            \le |D^3\beta|_\infty \|\delta{X}\|_{\alpha;p,\infty}^2 |t-s|^{2\alpha}   \quad \text{$\mathbb{P}$-a.s.},
        \end{aligned}
    \end{equation*} and by definiton of stochastic controlled rough path \begin{equation*}
        \begin{aligned}
            |\E_s(D^2\beta(X_s)R^X_{s,t})| &= |D^2\beta(X_s)\E_s(R^X_{s,t})|
            \le  |D^2\beta|_\infty \|\E_\cdot R^X\|_{2\alpha;\infty}|t-s|^{2\alpha} \qquad \text{$\mathbb{P}$-a.s.}
        \end{aligned}
    \end{equation*} 
    Thus we deduce that \begin{equation*} \label{eq:normoftheremainderinconsistencyforvecotrfields}
        \|\E_s(R^f_{s,t})\|_\infty \le (|D^3\beta|_\infty \|\delta{X}\|_{\alpha;p,\infty}^2+|D^2\beta|_\infty \|\E_\cdot R^X\|_{2\alpha;\infty})|t-s|^{2\alpha}
    \end{equation*}
    In conclusion, \eqref{eq:estimateinaprioriforlinearvectorfields} follows by applying the a priori estimate contained in \cite[Proposition 4.5]{FHL21} to the previous computations.   
\end{proof}

From this point till the end of this section, let $G$ be a stochastic linear vector field from $W$ to $W$ and let $S$ be a stochastic linear vector field from $W$ to $\mathscr{L}(\R^m,W)$. 
Let us fix $p \in [2,\infty)$ and $\gamma \in [0,\alpha]$ such that $\alpha + \gamma > \frac12$ and $\alpha + 2\gamma > 1$ \footnote{At first reading, take $\gamma = \alpha$.}. 
Let $(f,f') \in~\mathbf{D}^{2\gamma}_ZL^{p,\infty}\mathscr{L}(W,\mathscr{L}(\R^n,W))$. We also consider a stochastic controlled rough path $(F,F') \in \mathscr{D}^{2\gamma}_Z L^p(W)$ such that $F$ is $\mathbb{P}$-a.s.\ continuous and $\sup_{t \in [0,T]} \|F_t\|_p <+\infty$. Such a path is often called the forcing term.
We want to give a meaning to a solution of \begin{equation} \begin{aligned}
     dY_t &= dF_t +  G_tY_tdt + S_tY_tdB_t + (f_t,f'_t)Y_td\mathbf{Z}_t, \quad t \in [0,T] .
\end{aligned} \label{eq:linearRSDE_model}
\end{equation}

\begin{defn}[Solution in the sense of Davie] \label{def:solutionlinearRSDEs}
Let $\xi \in L^p(\Omega,\mathcal{F}_0;W)$.  An $L^p$-integrable solution in the sense of Davie of \eqref{eq:linearRSDE_model}, starting from $\xi$, is a continuous and adapted $W$-valued stochastic process $Y=(Y_t)_{t \in [0,T]}$ with \begin{equation*}
        \sup_{t \in [0,T]} \|Y_t\|_p<+\infty
    \end{equation*} 
    and such that 
    \begin{enumerate}
        \item[i)] $Y_0 = \xi$;
        \item[ii)] there exist $C_1,C_2\ge0$ and $\varepsilon_1,\varepsilon_2 >0$ satisfying \begin{equation*}
            \|Y_{s,t}^{\natural}\|_p \le C_1 |t-s|^{\frac{1}{2}+\varepsilon_1} \quad \text{and} \quad \|\E_s(Y_{s,t}^{\natural})\|_p \le C_2 |t-s|^{1+\varepsilon_2}
        \end{equation*} for any $(s,t) \in \Delta_{[0,T]}$, where \begin{multline*} \label{davieexpansion}
            Y_{s,t}^{\natural} :=  \delta Y_{s,t} - \delta F_{s,t} - \int_s^t G_rY_r dr - \int_s^t S_rY_rdB_r - f_sY_s\delta Z_{s,t} \\  -(f'_sY_s+f_s(f_sY_s+F'_s))\mathbb{Z}_{s,t}.   
\end{multline*}
    \end{enumerate}
\end{defn}

\begin{rmk}
    There is no loss of generality in assuming $F_0=0$. Indeed, from the notion of solution that we have in Definition \ref{def:solutionlinearRSDEs}, it is sufficient to replace $(F_t)_{t \in [0,T]}$ with \linebreak $(F_t-F_0)_{t\in [0,T]}$, and $\xi$ with $\xi + F_0$.
\end{rmk}

\begin{lemma} \label{lemma:solutionsoflinearRSDEsarecontrolled}
    Let $Y=(Y_t)_{t \in [0,T]}$ be an adapted $W$-valued stochastic process such that $\sup_{t \in [0,T]} \|Y_t\|_p < +\infty$ and satisfying condition $ii)$ of Definition \ref{def:solutionlinearRSDEs}.
    Then \begin{equation*}
        (Y,fY+F')\in \mathscr{D}^{2\gamma}_ZL^p(W).  
    \end{equation*}
\end{lemma}
\begin{proof}
    Let $(s,t)  \in \Delta_{[0,T]}$. Both $Y_t$ and $f_tY_t+F'_t$ are $\mathcal{F}_t$-measurable. Moreover,
    \begin{equation*} 
    \|f_tY_t+F'_t\|_p \le \|f\|_\infty \sup_{t\in [0,T]}\|Y_t\|_p + \sup_{t \in [0,T]}\|F'_t\|_p<+\infty.
\end{equation*} 
From the Davie-type expansion in Definition \ref{def:solutionlinearRSDEs} and applying Bochner inequality for the Bochner integral and BDG inequality for the It\^o integral, we have \begin{equation} \label{eq:solutionlinearRSDEsfirstholderestimate} \begin{aligned}
    &\|\delta Y_{s,t}\|_p \\
    &\le \|\delta F_{s,t}\|_p + \|\int_s^t G_rY_r dr\|_p + \|\int_s^t S_rY_rdB_r\|_p + \|f_sY_s\delta Z_{s,t}\|_p \\ & \quad    + \|(f'_sY_s+f_s(f_sY_s+F'_s))\mathbb{Z}_{s,t}\|_p +  \|Y^\natural_{s,t}\|_p 
    \\    
    &\lesssim_p \|\delta F\|_{\gamma;p}|t-s|^\gamma + \|G\|_\infty \sup_{t \in [0,T]} \|Y_t\|_p |t-s| + \|S\|_\infty \sup_{t \in [0,T]} \|Y_t\|_p |t-s|^\frac{1}{2}  \\
    & \quad + \|f\|_\infty |\delta Z|_\alpha \sup_{t \in [0,T]} \|Y_t\|_p |t-s|^\alpha  \\
    & \quad +  \big( (\|f'\|_\infty+\|f\|^2_\infty)  \sup_{t \in [0,T]} \|Y_t\|_p  +  \|f\|_\infty \sup_{t \in [0,T]} \|F'_t\|_p \big) |\mathbb{Z}|_{2\alpha} |t-s|^{2\alpha} +  \|Y^\natural\|_{\alpha+\gamma;p} |t-s|^{\alpha + \gamma},
\end{aligned} 
\end{equation} 
from which we deduce that $\|\delta Y\|_{\gamma;p} <+\infty$. Moreover, thanks to the martingale property of the It\^o integral and the contraction property of the conditional expectation, 
\begin{equation} \label{eq:remainderofthesolution_linearRSDEs}
    \begin{aligned}
        &\|\E_s(\delta Y_{s,t} - (f_sY_s+F'_s) \delta Z_{s,t})\|_p 
        \lesssim_{\alpha,\gamma,T} \Big( \|\E_\cdot R^F\|_{2\gamma;p} + \|G\|_\infty \sup_{t \in [0,T]} \|Y_t\|_p   \\ &\quad +(\|f'\|_\infty+\|f\|^2_\infty)|\mathbb{Z}|_{2\alpha} \sup_{t \in [0,T]} \|Y_t\|_p + \|f\|_\infty \sup_{t \in [0,T]} \|F'_t\|_p |\mathbb{Z}|_{2\alpha} + \| Y^\natural\|_{\alpha + \gamma;p} \Big) |t-s|^{2\gamma}.
    \end{aligned}
\end{equation} 
Finally, it is easy to verify that \begin{equation*} \begin{aligned}
    \|\delta(f_\cdot Y_\cdot+F'_\cdot )_{s,t}\|_p &\le (\|f\|_\infty \|\delta Y\|_{\gamma;p}  + \|\delta f\|_{\gamma;p,\infty} \sup_{t \in [0,T]}\|Y_t\|_p + \|\delta F'\|_{\gamma;p}) |t-s|^\gamma.  
\end{aligned}
\end{equation*} 
\end{proof} 

Solutions to linear RSDEs can be characterized according to the following proposition. 
\begin{prop} \label{prop:characterizationsolutionslinearRSDEs} 
    Let $Y=(Y_t)_{t \in [0,T]}$ be an adapted $W$-valued stochastic process such that
        $\sup_{t \in [0,T]} \|Y_t\|_p<+\infty$ and let $\xi \in L^p(\Omega,\mathcal{F}_0;W)$.  The following are equivalent:
    \begin{enumerate} 
        \item[(i)] $Y$ satisfies conditions i) and ii) of Definition \ref{def:solutionlinearRSDEs};
        \item[(ii)] $(fY,f'Y+f(f Y+F'))\in \mathscr{D}^{2\gamma}_ZL^p(\mathscr{L}(\R^n,W))$ and, $\mathbb{P}$-a.s.\ and for any $t \in [0,T]$, \begin{equation*}\label{eq:RSDEintegralform}
            Y_t = \xi + F_t +\int_0^t G_rY_rdr +\int_0^tS_rY_rdB_r + \int_0^t (f_rY_r,f'_rY_r+f_r(f_rY_r+F'_r)) d\mathbf{Z}_r.
        \end{equation*} 
    \end{enumerate}
    where the last integral on the right-hand side is a rough stochastic integral in the sense of \cite{FHL21} (cf.\ \eqref{eq:roughstochasticintegral}). 
    In particular, for any $(s,t)\in\Delta_{[0,T]}$ it holds that \begin{equation*}
        \|Y^\natural_{s,t}\|_p \lesssim |t-s|^{\alpha+\gamma} \qquad \text{and} \qquad \|\E_s(Y^\natural_{s,t})\|_p \lesssim |t-s|^{\alpha+2\gamma}
    \end{equation*}
\end{prop}
\begin{proof} $(i)\Rightarrow(ii) \ $ The fact that $(f_{\cdot}Y_{\cdot},f'_{\cdot}Y_{\cdot}+f_{\cdot}(f_\cdot Y_{\cdot}+F'_\cdot))\in \mathscr{D}^{2\gamma}_ZL^p(\mathscr{L}(\R^n,W))$ follows combining Lemma \ref{lemma:solutionsoflinearRSDEsarecontrolled} and Proposition \ref{prop:stabilityundercompositionlinearvectorfields2}.  For any $(s,t) \in \Delta_{[0,T]}$, we denote by \begin{equation*}
    A_{s,t}=f_sY_s\delta Z_{s,t} +(f'_sY_s+f_s(f_sY_s+F'_s))\mathbb{Z}_{s,t} = f_sY_s\delta Z_{s,t} +(f'_sY_s+f_sY'_s)\mathbb{Z}_{s,t}.
\end{equation*} 
In view of \eqref{eq:roughstochasticintegral}, it makes sense to define \begin{equation*} \begin{aligned}
    R_{s,t} &:=\delta Y_{s,t}-\delta F_{s,t}-\int_s^tG_rY_rdr-\int_s^tS_rY_rdB_r - \int_s^t (f_r,f'_r)Y_r d\mathbf{Z}_r \\
    &= A_{s,t}-\int_s^t (f_r,f'_r)Y_r d\mathbf{Z}_r + Y_{s,t}^\natural 
\end{aligned}
\end{equation*} 
The previous equality leads to 
to $\|R_{s,t}\|_p\lesssim|t-s|^{\alpha+\gamma}$ and $\|\E_s(R_{s,t})\|_p\lesssim|t-s|^{\alpha+2\gamma}$. On noting that $\delta R \equiv 0$ (i.e.\ $R$ is additive), this is enough to get that $R_{s,t}=0$ for any $(s,t)\in\Delta_{[0,T]}$. In particular, for any $t \in [0,T]$ and recalling that $F_0=0$, it holds \begin{equation*}
    Y_t-\xi=\delta Y_{0,t}= F_t + \int_0^tG_rY_rdr+\int_0^tS_rY_rdB_r + \int_0^t (f_r,f'_r)Y_r d\mathbf{Z}_r \qquad \text{$\mathbb{P}$-a.s.}
\end{equation*} and the conclusion follows by the a.s.\ continuity of the involved processes.  \\

\noindent $(ii)\Rightarrow(i) \ $ Trivially one has $Y_0=\xi$ and we can write
\begin{equation*}
    \begin{aligned}
        Y_{s,t}^\natural &= \delta Y_{s,t} - \delta F_{s,t} - \int_s^tG_rY_rdr -\int_s^t S_rY_rdB_r - f_sY_s\delta Z_{s,t} -(f'_sY_s+f_s(f_sY_s+F'_s))\mathbb{Z}_{s,t}   \\  &=  \int_s^t (f_r,f'_r)Y_r d\mathbf{Z}_r - A_{s,t}.
    \end{aligned}
\end{equation*}
From stochastic sewing lemma \cite[Theorem 2.9]{FHL21} we conclude that $\|Y_{s,t}^\natural\|_p\lesssim|t-s|^{\alpha+\gamma}$ and $\|\E_s(Y_{s,t}^\natural)\|_p\lesssim|t-s|^{\alpha+2\gamma}$.
\end{proof}

We present and prove some result about linear rough stochastic differential equations. Namely, we show some a priori bounds for the solutions (cf. Theorem \ref{thm:aprioriestimatelinearRSDEs}), an existence-and-uniqueness result (cf. Theorem \ref{thm:wellposednesslinearRSDEs}) and some stability estimates (cf. Theorem \ref{thm:stabilityforlinearRSDEs}). In the proofs, we make use of the class weighted norms introduced in Appendix \ref{appendix weighted norms}, in the spirit of \cite{bailleul2017unbounded}. 

\begin{thm}[A priori estimates] \label{thm:aprioriestimatelinearRSDEs} Let $\xi \in L^p(\Omega,\mathcal{F}_0;W)$ and let $Y=(Y_t)_{t \in [0,T]}$ be an $L^p$-integrable solution to \eqref{eq:linearRSDE_model} starting from $\xi$. Let $M>0$ be any constant satisfying \begin{equation*} \label{eq:constantMinaprioriestimatesforlinearRSDEs}
    \|G\|_\infty + \|S\|_\infty + \|(f,f')\|_\infty + \|(f,f')\|_{\mathbf{D}_Z^{2\gamma}L^{p,\infty
    }\mathscr{L}} + |\mathbf{Z}|_\alpha  \le M.
\end{equation*} 
Then 
\begin{align*}
        \|(Y,fY+F')\|_{Z;2\gamma;p} \lesssim_{M,\alpha,\gamma,p,T} \|\xi\|_p + \|(F,F')\|_{\mathscr{D}^{2\gamma}_ZL^{p}}.
    \end{align*} 
\end{thm}

\begin{proof} 
From the proof of Lemma \ref{lemma:solutionsoflinearRSDEsarecontrolled}, it is clear that we need to obtain a suitable estimate on $\|Y^\natural\|_{\alpha+\gamma;p}$, and the linear nature of the problem allows us to directly employ the estimates provided by the stochastic sewing lemma. Indeed, by construction, we have that $Y^\natural_{t,t}=0$ and $Y^\natural_{s,t}$ is $\mathcal{F}_t$-measurable. Moreover, taking into account the explicit form of $Y^\natural$,
    for any $(s,u,t)\in\Delta^2_{[0,T]}$ it holds that  \begin{equation*} \label{eq:deltaYnaturalinaprioriestimates}
        \begin{aligned}
            \delta Y^\natural_{s,u,t} &= Y^\natural_{s,t} - Y^\natural_{s,u} - Y^\natural_{u,t}   \\
            &= \big(\delta (f Y)_{s,u} -  (f'_sY_s+f_s(f_sY_s+F'_s))\delta Z_{s,u}\big) \delta Z_{u,t} + \delta (f' Y +f(f Y+ F'))_{s,u} \mathbb{Z}_{u,t}  \\ &= R_{s,u}^{f Y} \delta Z_{u,t} + \delta (f'_\cdot Y_\cdot +f_\cdot^2 Y_\cdot)_{s,u} \mathbb{Z}_{u,t} + \delta (f_\cdot F'_\cdot)_{s,u} \mathbb{Z}_{u,t}.
        \end{aligned}
    \end{equation*}
    We deduce that 
    \begin{equation} \label{eq:deltaYnaturalestimate1_firstversion}
        \begin{aligned}
            \|\delta Y^\natural_{s,u,t}\|_p &\le \|\delta (f_\cdot Y_\cdot)_{s,u} \|_p |\delta Z_{u,t}| + \|f'_sY_s+f_s(f_sY_s+F'_s)\|_p|\delta Z_{s,u}|| \delta Z_{u,t}| \\& \quad  + \|\delta (f'_\cdot Y_\cdot +f_\cdot (f_\cdot Y_\cdot + F'_\cdot))_{s,u} \|_p |\mathbb{Z}_{u,t}|  \\
            &\lesssim_{|\mathbf{Z}|_\alpha,\alpha,\gamma,T} \|(f Y,f'Y+f(f Y+F'))\|_{\mathscr{D}_Z^{2\gamma}L^p} |t-s|^{\alpha+\gamma}
        \end{aligned}
    \end{equation} and \begin{equation} \label{eq:deltaYnaturalestimate2_secondversion} \begin{aligned}
        \|\E_s(\delta Y^\natural_{s,u,t})\|_p &\le \|\E_s(R_{s,u}^{f Y})\|_p |\delta Z_{u,t}| + \|\delta (f' Y +f^2 Y)_{s,u}\|_p |\mathbb{Z}_{u,t}| + \|\delta (f F')_{s,u}\|_p |\mathbb{Z}_{u,t}| \\
        &\lesssim_{|\mathbf{Z}|_\alpha,\alpha,\gamma,T} \|(f Y,f'Y+f(f Y+F'))\|_{\mathscr{D}_Z^{2\gamma}L^p} |t-s|^{\alpha+2\gamma}.
    \end{aligned}
    \end{equation}
    We therefore showed that the two-parameter process $Y^\natural$ satisfies the assumptions of the stochastic sewing lemma in \cite[Theorem 2.9]{FHL21} 
    with $m=n=p$. By the uniqueness part, we also get that the process $\mathcal{Y}=(\mathcal{Y}_t)_{t\in [0,T]}$ defined by the stochastic sewing lemma is necessarily the zero process. Equivalently, recalling the sewing map defined in Proposition \ref{prop:sewingmap}, we write \begin{equation}\label{eq:sewingofynatural}\Lambda(Y^\natural) \equiv -Y^\natural.\end{equation} 
    In order to be able to apply estimates coming from Proposition \ref{prop:sewingmap}, we introduce the parameter $\lambda>0$, and in the rest of the proof we may need to assume it to be smaller than some constants depending on $M,\alpha,\gamma,p,T$.
    Let $(s,u,t) \in \Delta^2_{[0,T]}$ with $|t-s|\le \lambda$. 
    By taking into account \eqref{eq:esttimateforapriori1},\eqref{eq:esttimateforapriori2},\eqref{eq:esttimateforapriori3} and by recalling the definition of weighted norms, from \eqref{eq:deltaYnaturalestimate1_firstversion} it is standard to deduce that 
    \begin{equation} \label{eq:deltaYnaturalestimate1}
        \begin{aligned}
            \|\delta Y^\natural_{s,u,t}\|_p 
            &\lesssim_{M,\alpha,\gamma,p,T} \big( \|(F,F')\|_{\mathscr{D}^{2\gamma}_ZL^{p}}  + (|\delta Y|)_{\gamma;p;\lambda} + (|Y_\cdot|)_{p;\lambda} \big) e^{t/\lambda} |t-s|^{\alpha+\gamma}.
        \end{aligned}
    \end{equation} 
    Writing $R_{s,u}^{f_\cdot Y_\cdot}$ as in \eqref{eq:remaindercompositionwithstochasticcontrolledlinearvectorfields} and arguing as in \eqref{eq:remainderofthesolution_linearRSDEs}, we obtain \begin{equation*}\begin{aligned}
        \|\E_s(R_{s,u}^{f_\cdot Y_\cdot})\|_p 
        \lesssim \big( (|\delta Y|)_{\gamma;p;\lambda} + \|(F,F')\|_{\mathscr{D}^{2\gamma}_ZL^{p}} + (|Y^\natural|)_{\alpha+\gamma;p;\lambda} + (|Y_\cdot|)_{p;\lambda} \big)  e^{t/\lambda} |t-s|^{2\gamma},
    \end{aligned}
    \end{equation*} 
    for some implicit constant depending on $M,\alpha,\gamma,T$. 
    Hence \eqref{eq:deltaYnaturalestimate2_secondversion} leads to
    \begin{equation} \label{eq:deltaYnaturalestimate2}
        \begin{aligned}
            &\|\E_s(\delta Y^\natural_{s,u,t})\|_p \\
            &\quad \lesssim_{M,\alpha,\gamma,p,T} \big( \|(F,F')\|_{\mathscr{D}^{2\gamma}_ZL^{p}} +(|Y^\natural|)_{\alpha+\gamma;p;\lambda}  + (|\delta Y|)_{\gamma;p;\lambda} +  (|Y_\cdot|)_{p;\lambda}\big) e^{t/\lambda} |t-s|^{\alpha+2\gamma}.
        \end{aligned}
    \end{equation}
    Applying \eqref{eq:sewingmap},  we can now obtain the desired estimate for $(|Y^\natural|)_{\alpha+\gamma;p;\lambda}$. Indeed, \begin{equation} \label{eq:corollaryaprioiestimates2}
        \begin{aligned}
             (|Y^\natural|)_{\alpha+\gamma;p;\lambda} &= (|\Lambda(Y^\natural)|)_{\alpha+\gamma;p;\lambda}  \\
             &\lesssim_{\alpha,\gamma,p} \lambda^\gamma(|\E_\cdot \delta Y^\natural|)_{\alpha+2\gamma;p;\lambda}+(|\delta Y^\natural|)_{\alpha+\gamma;p;\lambda}  \\ 
             &\lesssim_{M,\alpha,\gamma,p,T} \|(F,F')\|_{\mathscr{D}^{2\gamma}_ZL^{p}} + \lambda^\gamma (| Y^\natural|)_{\alpha+\gamma;p;\lambda} + (|\delta Y|)_{\gamma;p;\lambda} + (|Y_\cdot|)_{p;\lambda}  \\ 
             &\lesssim \|(F,F')\|_{\mathscr{D}^{2\gamma}_ZL^{p}} + (|\delta Y|)_{\gamma;p;\lambda} +  (|Y_\cdot|)_{p;\lambda} 
        \end{aligned}
    \end{equation}
    where the last inequality holds provided that we choose $\lambda=\lambda(M,\alpha,\gamma,p,T)$ sufficiently small.
    From Proposition \ref{prop:comparison} we have that \begin{equation} \label{eq:corollaryaprioriestimates4}
            (|Y_\cdot|)_{p;\lambda} \lesssim \|\xi\|_p + \lambda^\gamma (|\delta Y|)_{\gamma;p;\lambda}.
        \end{equation}  
    On the other hand, arguing as in \eqref{eq:solutionlinearRSDEsfirstholderestimate}, we get \begin{equation} \label{eq:aprioriestimatedeltaY}
        \begin{aligned}
            (|\delta Y|)_{\gamma;p;\lambda} 
            &\lesssim_{M,\alpha,\gamma,p,T} \|(F,F')\|_{\mathscr{D}^{2\gamma}_ZL^p}  +  (|Y_{\cdot}|)_{p;\lambda}  + \lambda^\gamma (|Y^\natural|)_{\alpha+\gamma;p;\lambda}.
        \end{aligned}
    \end{equation}
    Combining \eqref{eq:aprioriestimatedeltaY} with \eqref{eq:corollaryaprioiestimates2}, we obtain \begin{equation} \label{eq:corollaryaprioriestimates1}
        \begin{aligned}
            (|\delta Y|)_{\gamma;p;\lambda}
            &\lesssim_{M,\alpha,\gamma,p,T} \|(F,F')\|_{\mathscr{D}^{2\gamma}_ZL^{p}} + (|Y_{\cdot}|)_{p;\lambda} + \lambda^\gamma(|\delta Y|)_{\gamma;p;\lambda} \\
            &\lesssim \|(F,F')\|_{\mathscr{D}^{2\gamma}_ZL^{p}} + (|Y_{\cdot}|)_{p;\lambda} 
        \end{aligned}
    \end{equation} and the last inequality holds if $\lambda=\lambda(M,\alpha,\gamma,p,T)>0$ is sufficiently small.
    We insert the previous estimate in \eqref{eq:corollaryaprioriestimates4} to show that \begin{equation} \label{eq:corollaryaprioriestimates3}
       \begin{aligned}
           (|Y_{\cdot}|)_{p;\lambda} 
           &\lesssim_{M,\alpha,\gamma,p,T} \|\xi\|_p+ \|(F,F')\|_{\mathscr{D}^{2\gamma}_ZL^{p}} + \lambda^\gamma (|Y_{\cdot}|)_{p;\lambda} \lesssim  \|\xi\|_p + \|(F,F')\|_{\mathscr{D}^{2\gamma}_ZL^{p}},  
       \end{aligned}
   \end{equation}    where the last inequality follows whenever $\lambda=\lambda(M,\alpha,\gamma,p,T)>0$ is small enough.
   Putting everything together, we showed that there exists $\lambda=\lambda(M,\alpha,\gamma,p,T)>0$ such that \begin{equation} \label{eq:aprioriestimatewithweightednorms}
       (|Y_\cdot|)_{p;\lambda} + (|\delta{Y}|)_{\gamma;p;\lambda} + (|Y^\natural|)_{\alpha+\gamma;p;\lambda} \lesssim_{M,\alpha,\gamma,p,T} \|\xi\|_p + \|(F,F')\|_{\mathscr{D}^{2\gamma}_ZL^{p}}.
   \end{equation}
   We also observe that, from \eqref{eq:remainderofthesolution_linearRSDEs} it is possible to deduce that \begin{equation} \label{eq:weightednormoftheremainderofY}
       (|\E_\cdot R^Y|)_{2\gamma;p} \lesssim_{M,\alpha,\gamma,T} \|(F,F')\|_{\mathscr{D}_Z^{2\gamma}L^p} + (|Y_\cdot|)_{p;\lambda} + (|Y^\natural|)_{\alpha+\gamma}.
   \end{equation}
   Together with \eqref{eq:aprioriestimatewithweightednorms}, this implies that \begin{equation*}
       \begin{aligned}
           (|(Y,Y')|)_{Z;2\gamma;p;\lambda} \lesssim_{M,\alpha,\gamma,p,T} \|\xi\|_p + \|(F,F')\|_{\mathscr{D}^{2\gamma}_ZL^{p}}
       \end{aligned}
   \end{equation*} 
   and the assertion is proved taking into account Remark \ref{rmk:equivalenceforcontrolledroughpaths}.
\end{proof}

\begin{thm}[Well-posedness] \label{thm:wellposednesslinearRSDEs} For any $\xi \in L^p(\Omega,\mathcal{F}_0;W)$, there exists a unique up to indistinguishability $L^p$-integrable solution to \eqref{eq:linearRSDE_model} starting from $\xi$.
\end{thm}

\begin{proof}
    \textbf{Uniqueness.} Let $Y=(Y_t)_{t \in [0,T]}$ be an $L^p$-integrable solution of \eqref{eq:linearRSDE_model} starting from $\xi=0$ and with forcing term $(F,F')=(0,0)$. Being $Y$ continuous and applying Theorem \ref{thm:aprioriestimatelinearRSDEs}, it follows that, $\mathbb{P}$-a.s.\ and for any $t \in [0,T]$, $Y_t=0$.
    The proof of uniqueness can be concluded due to the linearity of the equation.

    \textbf{Existence} Let $\xi \in L^p(\Omega,\mathcal{F}_0;W)$ and let us run a Picard's iteration. For convenience, let us fix a positive constant $M>0$, independent of $\xi$, and such that \begin{equation*}
    |\mathbf{Z}|_\alpha + \|G\|_\infty + \|S\|_\infty + \|(f,f')\|_\infty + \|(f,f')\|_{\mathbf{D}_Z^{2\gamma}L^{p,\infty}\mathscr{L}} \le M.
    \end{equation*}

    \textit{Step 1. (Construction of the iterated processes)}
    Let us set $Y^{-2} = Y^{-1} \equiv 0$ and $Y^0= \xi + F$.
    Recursively on $n\ge 0$, given $Y^n, Y^{n-1}$ and $Y^{n-2}$, the two-parameter stochastic process 
    \begin{equation*}
        A^n_{s,t} := \delta F_{s,t} + \int_s^tG_rY^n_rdr + \int_s^tS_rY_r^ndB_r + f_sY_s^n\delta Z_{s,t} +(f'_sY^n_s+f_s(f_sY^{n-1}_s+F'_s))\mathbb{Z}_{s,t}.
    \end{equation*} 
    satisfies, for any $(s,u,t) \in \Delta^2_{[0,T]}$, $\|\delta A^n_{s,u,t}\|_p \lesssim|t-s|^{\alpha+\gamma}, \, \|\E_s(\delta A^n_{s,u,t})\|_p \lesssim |t-s|^{\alpha+2\gamma}$ and $\|\sup_{\tau \in [\frac{s+t}{2},t]} |\delta A^n_{s,\frac{s+t}{2},\tau}|\|_p \lesssim |t-s|^{\alpha+\gamma}$.
    By the stochastic sewing lemma \cite[Theorem 2.9]{FHL21}, there is a unique continuous and adapted $W$-valued stochastic processes $Y^{n+1}=(Y^{n+1}_t)_{t\in [0,T]}$ such that $Y^{n+1}_0 = \xi$ and  
    \begin{equation*}
        \delta Y^{n+1}_{s,t}=A_{s,t}^n + Y^{n+1,\natural}_{s,t} \qquad (s,t)\in\Delta_{[0,T]}
    \end{equation*} with $\|Y^{n+1,\natural}_{s,t}\|_{p}\lesssim |t-s|^{\alpha+\gamma}$ and $\|\E_s(Y^{n+1,\natural}_{s,t})\|_{p}\lesssim |t-s|^{\alpha+2\gamma}$. 
    Since $\|A^n_{s,t}\|\lesssim|t-s|^\gamma$ one can also deduce that $\sup_{t \in [0,T]} \|Y^{n+1}_t\|_p < +\infty$. \\

    \textit{Step 2. (Convergence)} Let us define a new sequence auxiliary processes \begin{equation*} \label{eq:definitionofJ^n}
        \begin{aligned}
            J^{-1} \equiv0, \qquad J^{n} := Y^{n} - Y^{n-1} \qquad n\ge 0
        \end{aligned}
    \end{equation*} 
    and let $J^{n+1,\natural}_{s,t}:=Y^{n+1,\natural}_{s,t} - Y^{n,\natural}_{s,t}$. 
    For any $(s,t)\in\Delta_{[0,T]}$ and for any $n\ge 0$,  \begin{equation*} \label{eq:deltaJ^ninexistence} \begin{aligned}
        \delta J^{n+1}_{s,t} 
        =  \int_s^tG_rJ_r^ndr + \int_s^tS_rJ_r^ndB_r + f_sJ_s^n\delta Z_{s,t} + (f'_sJ^n_s+f_s^2J^{n-1}_s)\mathbb{Z}_{s,t} + J^{n+1,\natural}_{s,t}. 
    \end{aligned} 
    \end{equation*}
    Moreover, for any $(s,u,t) \in \Delta_{[0,T]}^2$ and for any $n \ge 0$ it holds that \begin{equation*} \label{eq:deltaJ^n+1natural1}
        \begin{aligned}
            \delta J^{n+1,\natural}_{s,u,t} 
            &= \big(\delta (f_\cdot J^n_\cdot)_{s,u} -  (f'_sJ^n_s+f_s^2J^{n-1}_s)\delta Z_{s,u}\big) \delta Z_{u,t} + \delta (f'_\cdot J^n_\cdot +f_\cdot^2 J^{n-1}_\cdot)_{s,u} \mathbb{Z}_{u,t}
        \end{aligned}
    \end{equation*} and, provided that $n \ge 1$, we can write  \begin{multline*} \label{eq:deltaJ^n+1natural2}
            \delta (f_\cdot J^n_\cdot)_{s,u} -  (f'_sJ^n_s+f_s^2J^{n-1}_s)\delta Z_{s,u} 
            = R^f_{s,u} J_s^n +\delta{f}_{s,u}\delta{J}^n_{s,u}  \\ + f_s \Big(\int_s^tG_rJ^{n-1}_rdr + \int_s^tS_rJ^{n-1}_rdB_r +(f'_sJ^{n-1}_s+f^2_sJ^{n-2}_s)\mathbb{Z}_{s,u}+J^{n,\natural}_{s,u}\Big) .
    \end{multline*}
    Note also that $J^n_0=0$ for any $n \ge 1$. 
    This last fact implies, taking into account Proposition \ref{prop:comparison}, that \begin{equation*} \label{eq:comparisoninpicarditeration}
        (|J^n_\cdot|)_{p;\lambda} \lesssim \lambda^\gamma (|\delta{J}^n|)_{\gamma;p;\lambda} \qquad \text{for any $n \ge 1$ and for any $\lambda>0$.}
    \end{equation*} 
    Proceeding as in the proof of Theorem \ref{thm:aprioriestimatelinearRSDEs}, we can therefore conclude that there exists $\lambda>0$ sufficiently small such that, for any $n \ge 2$,   
    \begin{multline*}  
        (|J^{n-1}_\cdot|)_{p;\lambda} + (|\delta J^n|)_{\gamma;p;\lambda} + (|J^{n+1,\natural}|)_{\alpha+\gamma;p;\lambda}  \\ \lesssim_{M,\alpha,\gamma,p,T} \lambda^\gamma \big((|J^{n-2}_\cdot|)_{p;\lambda} + (|\delta J^{n-1}|)_{\gamma;p;\lambda} + (|J^{n,\natural}|)_{\alpha+\gamma;p;\lambda}\big).  
    \end{multline*} 
    As a consequence, 
    \begin{equation} \label{eq:keyinequalityfromPicard}
        (|J^{n-1}_\cdot|)_{p;\lambda} + (|\delta J^n|)_{\gamma;p;\lambda} + (|J^{n+1,\natural}|)_{\alpha+\gamma;p;\lambda} \lesssim_{M,\alpha,\gamma,p,T} \lambda^{\gamma n}.
    \end{equation}
    Let us now construct a solution as the limit of Picard's iterations.
    Assume $\lambda >0$ was chosen sufficiently small that $\lambda^{\gamma} < \frac{1}{2}$.
    From \eqref{eq:keyinequalityfromPicard} it is easy to deduce that, for any $t \in [0,T]$, $(Y^n_t)_n$ is a Cauchy sequence in $L^p(\Omega;W)$. Indeed, for any $n \ge 3$, \begin{equation*}
        \|Y^{n+1}_t - Y^n_t\|_p \le (| J^n_\cdot |)_{p;\lambda} e^{t/\lambda} \le \lambda^{\gamma(n+1)} e^{t/\lambda} 
    \end{equation*}
    and $\sum_{n=3}^\infty \lambda^{\gamma(n+1)} e^{t/\lambda} <+\infty$, due to the assumption on $\lambda$. 
    Let $Y=(Y_t)_{t \in [0,T]}$ be the adapted process defined as
    
    \begin{equation*} \label{eq:definitionofthesolutionforPicard}
        Y_t :=  \lim_{n \to +\infty} Y^n_t \qquad \text{in $L^p(\Omega;W)$}. 
    \end{equation*} 
    By construction $Y_0 = \xi$ and let us immediately observe that, uniformly in $t \in [0,T]$,  
    \begin{equation*} \label{eq:boundednessofthesolutionPicard}
        \begin{aligned}
            \|Y_t\|_p  &\le \|\xi\|_p+ \|L^p\text{-}\lim_{n \to +\infty} Y^n_t - Y_t^0\|_p = \|\xi\|_p + \lim_{n \to +\infty} \| \sum_{k=0}^{n-1}J^{k+1}_t\|_p \\
            &\le \|\xi\|_p + e^{T/\lambda} \lim_{n \to +\infty} \sum_{k=0}^{n-1} (|J^{k+1}_\cdot|)_{p;\lambda} \lesssim_{M,\alpha,\gamma,p,T} \|\xi\|_p + e^{T/\lambda}   \sum_{k=0}^{+\infty} \lambda^{\gamma(k+2)} < +\infty. 
        \end{aligned}
    \end{equation*}  
    Let us now show that the Davie expansion in Definition \ref{def:solutionlinearRSDEs}.ii) is satisfied. 
    By construction and for any $(s,t) \in \Delta_{[0,T]}$ we can write 
    \begin{equation*} \begin{aligned}
        Y^{\natural}_{s,t}  
        &:= \delta Y_{s,t} - \delta F_{s,t} -   \int_s^tG_rY_rdr - \int_s^tS_rY_rdB_r - f_sY_s\delta Z_{s,t}  - (f'_sY_s+ f_s(f_s Y_s+F'_s))\mathbb{Z}_{s,t} \\ 
        &= \lim_{n \to +\infty} \Big(\delta Y^{n+1}_{s,t} -   \delta F_{s,t} - \int_s^tG_rY_r^ndr - \int_s^tS_rY_r^ndB_r   - f_sY_s^n\delta Z_{s,t} \\ & \qquad  - (f'_sY^n_s+ f_s(f_sY^{n-1}_s+F'_s))\mathbb{Z}_{s,t} \Big) \\
        &= \lim_{n \to +\infty} Y^{n+1,\natural}_{s,t} \qquad \text{in $L^p(\Omega;W)$}. 
        \end{aligned}
    \end{equation*}
    Taking into account \eqref{eq:keyinequalityfromPicard}, for any $n \ge 3$ we apply the stochastic sewing lemma \cite[Theorem 2.9]{FHL21} to $J^{n,\natural}$ and, for any $(s,t) \in \Delta_{[0,T]}$ with $|t-s|\le \lambda$ we get \begin{equation*}
        \|J^{n+1,\natural}_{s,t}\|_p = \|\Lambda(J^{n+1,\natural})_{s,t}\|_p \lesssim_{M,\alpha,\gamma,p,T,\lambda} \lambda^{\gamma(n-1)} |t-s|^{\alpha+\gamma}
    \end{equation*} and \begin{equation*}
        \|\E_s(J^{n+1,\natural}_{s,t})\|_p = \|\E_s(\Lambda(J^{n+1,\natural})_{s,t})\|_p \lesssim_{M,\alpha,\gamma,p,T,\lambda} \lambda^{\gamma(n-1)} |t-s|^{\alpha+2\gamma}
    \end{equation*} where the implicit constants do not depend on $n$. Here $\Lambda(\cdot)$ denotes the sewing map defined in Proposition \ref{prop:sewingmap}.
    Thus, we can pass the two previous inequalities to the limit in $L^p$ and we can conclude that, for any $(s,t) \in \Delta_{[0,T]}$ with $|t-s|\le \lambda$
    \begin{equation*} \begin{aligned}
        \|Y^\natural_{s,t}\|_p &\le \|Y^{3,\natural}_{s,t}\|_p + \|L^p\text{-}\lim_{n \to +\infty} \sum_{k=3}^{n} J^{k+1,\natural}_{s,t}\|_p \lesssim \big(\|Y^{3,\natural}\|_{\alpha+\gamma;p} + \sum_{k=3}^{\infty} \lambda^{\gamma(k-1)} \big) |t-s|^{\alpha+\gamma} 
    \end{aligned}
    \end{equation*} for some implicit constant depending on ${M,\alpha,\gamma,p,T,\lambda}$. Similarly, $$\|\E_s(Y^\natural_{s,t})\|_p \lesssim_{M,\alpha,\gamma,p,T,\lambda} |t-s|^{\alpha+2\gamma}. $$
    By Proposition \ref{prop:characterizationsolutionslinearRSDEs}, $\mathbb{P}$-a.s.\ and for any $t \in [0,T]$, \begin{equation*}
        Y_t = \xi + \int_0^t G_r Y_r dr + \int_0^t S_r Y_r dB_r + \int_0^t (f_r,f'_r) Y_r d\mathbf{Z}_r
    \end{equation*}
    from which we deduce that $Y$ (admits a modification, still denoted by $Y$, which) is $\mathbb{P}$-a.s.\ continuous. 
    This is enough to conclude that $Y$ is a solution to \eqref{eq:linearRSDE_model} on the time interval $[0,\lambda]$. None of the computations of \textit{Step 2.}\ depend on the initial conditon $\xi$. Hence the choice of $\lambda$ is independent of $\xi$, and this allows the iteration of the procedure finitely many times, to prove that $Y$ is a solution on the entire interval $[0,T]$.
\end{proof}

{Assume $p(1-2\alpha)> 1$}, $(f,f') \in \mathbf{D}_Z^{2\gamma}L^{2p,\infty}\mathscr{L}(W,\mathscr{L}(\R^n,W))$ and $(F,F')\in \mathscr{D}^{2\gamma}_ZL^{2p}(W)$ with $\sup_{t \in [0,T]} \|F_t\|_{2p}<+\infty$.
Let $\bar{\mathbf{Z}}=(\bar{Z},\bar{\mathbb{Z}}) \in \mathscr{C}^\alpha([0,T];\R^n)$
, let $\bar{G}$ be another stochastic linear vector field from $W$ to $W$, let $\bar{S}$ be a stochastic linear vector field from $W$ to $\mathscr{L}(\R^m,W)$, and let us also consider $(\bar{f},\bar{f}') \in \mathbf{D}^{2\gamma}_{\bar{Z}} L^{2p,\infty}\mathscr{L}(W,\mathscr{L}(\R^n,W))$. Let  $(\bar{F},\bar{F}') \in~\mathscr{D}^{2\gamma}_{\bar{Z}} L^{2p}(\bar{W})$ such that $\bar{F}$ is $\mathbb{P}$-a.s.\ continuous, $\bar{F}_0=0$ and $\sup_{t \in [0,T]} \|\bar{F}_t\|_p <~+\infty$. For any $(s,t) \in \Delta_{[0,T]}$, we denote by $\bar{R}^{\bar{f}}_{s,t} = \delta \bar{f}_{s,t} - \bar{f}'_s\delta \bar{Z}_{s,t}$ and, similarly, $\bar{R}^{\bar{F}}_{s,t} = \delta \bar{F}_{s,t} -~\bar{F}'_s\delta \bar{Z}_{s,t}$.
 Define
 \begin{equation*}\label{eq:theta} \begin{aligned}
    \theta_p &:= { \E\Big(\int_0^T |G_r-\bar{G}_r|_\mathscr{L}^{2p} dr\Big)^\frac{1}{2p}+ \E\Big(\int_0^T |S_r-\bar{S}_r|_\mathscr{L}^{2p} dr\Big)^\frac{1}{2p} } +  \| f,f' ; \bar f , \bar f' \|_{\mathscr{L};2p}  \\
    & \quad + \| f,f' ; \bar f, \bar f' \|_{Z,\bar Z ; 2\gamma; 2p; \mathscr{L}} + \|F,F' ; \bar F , \bar F'\|_{Z,\bar Z;2\gamma;p} .
    \end{aligned}
\end{equation*}

\begin{thm}[Stability] \label{thm:stabilityforlinearRSDEs}
Let $\xi,\bar{\xi}\in L^{2p}(\Omega,\mathcal{F}_0;W)$. Let $Y=(Y_t)_{t \in [0,T]}$ be an $L^{2p}$-integrable solution to \eqref{eq:linearRSDE_model} starting from $\xi$ and, similarly, let $\bar{Y}=(\bar{Y}_t)_{t \in [0,T]}$ be an $L^{2p}$-integrable solution to the linear RSDE driven by $\bar{\mathbf{Z}}$, with coefficients $\bar{G},\bar{S},(\bar{f},\bar{f}')$ and forcing term $\bar{F}$, starting from $\bar{\xi}$. Let $M_p > 0$ be any constant such that \begin{equation*} \label{eq:constantMinstabilityestimatesforlinearRSDEs}
    \begin{aligned}
     &\|G\|_\infty + \|S\|_\infty + \|(f,f')\|_\infty + \|(f,f')\|_{\mathbf{D}_Z^{2\gamma}L^{2p,\infty}\mathscr{L}} + \|(F,F')\|_{\mathscr{D}^{2\gamma}_ZL^{2p}} + |\mathbf{Z}|_\alpha  + \|\xi\|_{2p}   \\ 
     & +\|\bar{G}\|_\infty + \|\bar{S}\|_\infty +  \|(\bar{f},\bar{f}')\|_\infty + \|(\bar{f},\bar{f}')\|_{\mathbf{D}_{\bar{Z}}^{2\gamma}L^{2p,\infty}\mathscr{L}} + \|(\bar{F},\bar{F}')\|_{\mathscr{D}^{2\gamma}_ZL^{2p}}  + |\bar{\mathbf{Z}}|_\alpha + \|\bar{\xi}\|_{2p} \le M_p.
    \end{aligned}
\end{equation*} 
Then \begin{align*} 
    \sup_{t \in [0,T]} \|Y_t - \bar Y_t\|_p +  \|Y,fY+F';\bar{Y},\bar{f}\bar{Y}+\bar{F}'\|_{Z,\bar{Z};\alpha,\gamma;p} \lesssim \|\xi - \bar{\xi}\|_p + \rho_{\alpha,\gamma}(\mathbf{Z},\bar{\mathbf{Z}}) + \theta_p
\end{align*}
 where the implicit constant depends only on $M_p,\alpha,\gamma,p,T$.
 

\end{thm}

\begin{proof} 
    The proof is very similar to the one of Theorem \ref{thm:aprioriestimatelinearRSDEs} and it strongly relies on the following elementary H\"older-type inequalities: \begin{equation}\label{eq:elementaryHolder}
        \|AB - \bar{A}\bar{B}\|_{p} \le \|A\|_{2p}\|B-\bar{B}\|_{2p}+\|A-\bar{A}\|_{2p}\|\bar{B}\|_{2p}.
    \end{equation} 
    According to \eqref{eq:aprioriestimatewithweightednorms} and from \eqref{eq:weightednormoftheremainderofY}, there exists a parameter $\lambda_0 = \lambda_0(M_p,\alpha,\gamma,p,T)>0$ such that \begin{equation*}
    (|Y_\cdot|)_{2p;\lambda_0} + (|\delta Y|)_{\alpha;2p;\lambda_0} + (|Y^\natural|)_{\alpha+\gamma;2p;\lambda_0} + (|\E_\cdot R^Y|)_{\alpha+\gamma;2p;\lambda_0}  \lesssim_{M_p,\alpha,\gamma,p,T} 1 \end{equation*} and \begin{equation*}    
    (|\bar{Y}_\cdot|)_{2p;\lambda_0} + (|\delta \bar{Y}|)_{\alpha;2p;\lambda_0} + (|\bar{Y}^\natural|)_{\alpha+\gamma;2p;\lambda_0} + (|\E_\cdot \bar{R}^{\bar{Y}} |)_{\alpha+\gamma;2p;\lambda_0} \lesssim_{M_p,\alpha,\gamma,p,T} 1, 
\end{equation*} for that particular choice of $\lambda_0$ and so for any $\lambda \in (0,\lambda_0]$ (cf.\ Remark \ref{rmk:weightednormswithdifferentlambda}). 
For any $t \in [0,T]$, let us define \begin{equation*}
    \tilde{Y}_t := Y_t - \bar{Y}_t.
\end{equation*} 
For simplicity, we also denote by \begin{align*}
    \tilde{Z}_t &= Z_t - \bar{Z}_t & \tilde{\mathbb{Z}}_{s,t} &= \mathbb{Z}_{s,t} - \bar{\mathbb{Z}}_{s,t} & \tilde{F}_t &= F_t - \bar{F}_t & 
    \tilde{G}_t &= G_t - \bar{G}_t  \\ \tilde{S}_t &= S_t - \bar{S}_t & \tilde{f}_t &= f_t - \bar{f}_t & \tilde{f}'_t &= f'_t - \bar{f}'_t &  \tilde{Y}^\natural_{s,t} &= Y^\natural_{s,t} - \bar{Y}^\natural_{s,t}.
\end{align*}
The strategy is to mimic the argument of the proof of Theorem \ref{thm:aprioriestimatelinearRSDEs}, adjusting it to the process $\tilde{Y}=(\tilde{Y}_t)_{t \in [0,T]}$. The conclusion follows again if we find a suitable estimate on $\|\tilde{Y}^\natural\|_{\alpha+\gamma;p}$.

As a preliminary fact, note that for any $(s,t) \in \Delta_{[0,T]}$ with $|t-s| \le \lambda$ we can write \begin{equation*}
    \begin{aligned}
        \delta{\tilde{Y}}_{s,t} &= 
        \delta \tilde{F}_{s,t} + \int_s^t G_r\tilde{Y}_rdr + \int_s^t S_r\tilde{Y}_rdB_r  +  f_s\tilde{Y}_s\delta \bar{Z}_{s,t}+ (f'_s\tilde{Y}_s+f_s(f_s\tilde{Y}_s+\tilde{F}'_s))\bar{\mathbb{Z}}_{s,t} \\ & \quad + \tilde{Y}^\natural_{s,t} + \tilde{P}_{s,t},
    \end{aligned}
\end{equation*} 
where, by construction,  $\|\tilde{P}_{s,t}\|_p \lesssim_{M_p,\alpha,\gamma,p,T} (\rho_{\alpha,\gamma}(\mathbf{Z},\bar{\mathbf{Z}}) + \theta_p) e^{t/\lambda}|t-s|^\alpha$.
Indeed, a double application of H\"older's inequality shows that \begin{equation*}
    \begin{aligned}
        \|\int_s^t \tilde{G}_r \bar{Y}_r dr\|_p & 
        \le |t-s|^\frac{p-1}{p} \E(\int_s^t |\tilde{G}_r|^{2p})^\frac{1}{2p} \E(\int_s^t|\bar{Y}_r|^{2p} dr)^\frac{1}{2p} \\ & \le |t-s|^\frac{2p-1}{2p} \sup_{t \in [0,T]} \|\bar{Y}_t\|_{2p} \ \E(\int_s^t |\tilde{G}_r|^{2p})^\frac{1}{2p}
    \end{aligned}
\end{equation*}
and, by also applying BDG inequality, \begin{equation*}
    \begin{aligned}
        \|\int_s^t \tilde{S}_r \bar{Y}_r dB_r\|_p & 
        \lesssim_p \E( (\int_s^t |\tilde{G}_r \bar{Y}_r|^2 dr)^{\frac{p}{2}})^\frac{1}{p} 
        \le |t-s|^\frac{p-1}{2p} \sup_{t \in [0,T]} \|\bar{Y}_t\|_{2p} \ \E(\int_s^t |\tilde{S}_r|^{2p})^\frac{1}{2p}.
    \end{aligned}
\end{equation*}
Notice that, being $p > \frac{1}{1-2\alpha}$, it holds $\frac{2p-1}{2p}>\alpha+\gamma$ and $\frac{p-1}{2p}>\alpha$.
Arguing as in \eqref{eq:aprioriestimatedeltaY}, we deduce \begin{equation*}
    (|\delta \tilde{Y}|)_{\gamma;p;\lambda} \lesssim_{M_p,\alpha,\gamma,p,T}  (|\tilde{Y}_\cdot|)_{p;\lambda} + \lambda^\gamma(|\tilde{Y}^\natural|)_{\alpha+\gamma;p;\lambda} + \rho_{\alpha,\gamma}(\mathbf{Z},\bar{\mathbf{Z}}) + \theta_p. 
\end{equation*} 
Similarly, we can write \begin{align*}
    \E_s (R^Y_{s,t}) - \E_s (\bar R ^{\bar Y}_{s,t}) &= \E_s(R^F_{s,t}) - \E_s(\bar R^{\bar F}_{s,t}) + \E_s(\int_s^t G_r \tilde Y_r dr)  +  (f'_s \tilde Y_s + f_s(f_s\tilde Y_s + \tilde F'_s)\bar{\mathbb{Z}}_{s,t} \\ & \quad + \E_s(\tilde Y^\natural_{s,t}) + \tilde Q_{s,t}
\end{align*}
with $\|\E_s(\tilde Q_{s,t})\|_p \lesssim_{M_p, \alpha,\gamma,p,T} (\rho_{\alpha,\gamma}(\mathbf{Z},\bar{\mathbf{Z}}) + \theta_p) e^{t/\lambda} |t-s|^{\alpha+\gamma}$. This allows to conclude that \begin{align} \label{eq:estimateontheremiander_stability}
    (| \E_\cdot R^Y - \E_\cdot \bar R^{\bar Y} |)_{\alpha + \gamma;p;\lambda}\lesssim_{M_p,\alpha,\gamma,\gamma',p,T}  (|\tilde{Y}_\cdot|)_{p;\lambda} 
    + (|\tilde{Y}^\natural|)_{\alpha+\gamma;p;\lambda} + \rho_{\alpha,\gamma}(\mathbf{Z},\bar{\mathbf{Z}}) + \theta_p 
\end{align}

We can now find a suitable estimate on $(|Y^\natural|)_{\alpha+\gamma;p;\lambda}$. As we argued in \eqref{eq:deltaYnaturalinaprioriestimates}, for any $(s,u,t) \in \Delta_{[0,T]}$ we can write  \begin{equation*}
    \begin{aligned}
        \delta \tilde{Y}^\natural_{s,u,t} = \delta{Y}^\natural_{s,u,t} - \delta{\bar{Y}^\natural}_{s,u,t} 
        &= R_{s,u}^{fY} \delta \tilde{Z}_{u,t}+ (R_{s,u}^{fY} -\bar{R}_{s,u}^{\bar{f} \bar{Y}})\delta \bar{Z}_{u,t} + \delta (f'Y + f(fY+F'))_{s,u} \tilde {\mathbb{Z}}_{u,t}  \\
        &\quad  + \delta (f'Y + f(fY+F') - \bar f'\bar Y - \bar f(\bar f\bar Y+\bar F'))_{s,u}\bar{\mathbb{Z}}_{u,t},
    \end{aligned}
\end{equation*} 
Similarly to what we did in \eqref{eq:deltaYnaturalestimate1} we deduce that 
\begin{align*}
    (|\delta \tilde{Y}^\natural|)_{\alpha+\gamma;p;\lambda} &\lesssim_{M_p,\alpha,\gamma,p,T} (|\tilde{Y}_\cdot|)_{p;\lambda} + (|\delta \tilde{Y}|)_{\alpha;p;\lambda} + \rho_{\alpha,\gamma}(\mathbf{Z},\bar{\mathbf{Z}}) + \theta_p .
\end{align*}
Arguing as \eqref{eq:deltaYnaturalestimate2} and combining \eqref{eq:estimateontheremiander_stability} with Proposition \ref{prop:stabilityundercomposition} we get \begin{align*}
    (|\E_\cdot \delta \tilde{Y}^\natural|)_{\alpha+\gamma+\eta;p;\lambda} &\lesssim_{M_p,\alpha,\gamma,p,T} (|\tilde{Y}_\cdot|)_{p;\lambda} + (|\delta \tilde{Y}|)_{\alpha;p;\lambda} + (|\tilde{Y}^\natural|)_{\alpha+\gamma;p;\lambda} + \rho_{\alpha,\gamma}(\mathbf{Z},\bar{\mathbf{Z}}) + \theta_p.
\end{align*}
By Proposition \ref{prop:sewingmap}, this implies that, for a suitable choice of $\lambda \in (0,\lambda_0]$, \begin{equation*}
    (|\tilde{Y}^\natural|)_{\alpha+\gamma;p;\lambda} \lesssim_{M_p,\alpha,\gamma,\gamma',p,T} (|\tilde{Y}_\cdot|)_{p;\lambda} + (|\delta \tilde{Y}|)_{\alpha;p;\lambda} + \rho_{\alpha,\gamma}(\mathbf{Z},\bar{\mathbf{Z}}) + \theta_p.
\end{equation*}
Applying Proposition \ref{prop:comparison}, we can conclude that there exists $\lambda=\lambda(M_p,\alpha,\gamma,p,T) \in (0,\lambda_0]$ sufficiently small such that  \begin{equation*}
        (|\tilde{Y}_\cdot|)_{p;\lambda} + (|\delta\tilde{Y}|)_{\alpha;p;\lambda} + (|\tilde{Y}^\natural|)_{\alpha+\gamma;p;\lambda} \lesssim \|\xi - \bar{\xi}\|_p + \rho_{\alpha,\alpha}(\mathbf{Z},\bar{\mathbf{Z}}) + \theta_p,
    \end{equation*} where the implicit constant depends only on $M_p,\alpha,\gamma,\gamma',p,T$. The assertion follows also taking into consideration Remark \ref{rmk:equivalenceforcontrolledroughpaths}.
\end{proof}

\section{Malliavin differentiability of the solutions}
\label{sec: malliavin}

In this section we consider a geometric rough path $\mathbf{Z}=(Z,\mathbb{Z}) \in \mathscr{C}^{0,\alpha}_g([0,T];\R^n)$, \linebreak $\alpha \in (\frac{1}{3},\frac{1}{2}]$, and some coefficients $b \in {C}^1_b(\R^d;\R^d)$, $\sigma \in {C}^1_b(\R^d;\mathscr{L}(\R^m,\R^d))$, and \linebreak $\beta \in {C}^3_b(\R^d;\mathscr{L}(\R^n,\R^d))$. 
Let $x_0 \in \R^d$ be fixed and denote by $X=(X_t)_{t \in [0,T]}$ the solution of the following RSDE: \begin{equation} \label{eq:solutionRSDEinmalliavinforRSDEs} \begin{aligned} &dX_t=b(X_t)dt+\sigma(X_t)dB_t+(\beta(X_t),D\beta(X_t)\beta(X_t))d\mathbf{Z}_t, \qquad t \in [0,T] \\ &X_0=x_0.\end{aligned}
\end{equation} 
Our aim is to prove the Malliavin-differentiability of such a solution {for $t \in (0,T]$}, under those standard regularity assumptions on the coefficients. We rely on the following result, coming from general Malliavin calculus theory (cf.\ \cite[Lemma 1.2.3]{nualart2006malliavin} and \cite[Proposition 1.5.5]{nualart2006malliavin}).
\begin{lemma} \label{lemma:Malliavin via approximation}
    Let $F\in L^2(\Omega)$. Assume that, for a certain $k \in \mathbb{N}_{\ge 1}$, there exists $(F_n)_n \subseteq~\mathbb{D}^{k,2}$ such that $F_n \to F$ in $L^2(\Omega)$ and \begin{equation*}
        \sup_{n\in \mathbb{N}} \E\Big(\int_0^T \dots \int_0^T |\mathbb{D}^k_{\theta_1,\dots,\theta_k}F_n|^2 d\theta_1 \dots d\theta_k\Big) <+\infty.
    \end{equation*}  
    Then $F \in \mathbb{D}^{k,2}$ and $\mathbb{D}^kF_n \rightharpoonup \mathbb{D}^kF$ in $L^2([0,T]^k\times\Omega;(\R^m)^{\times k})$. Moreover, if for some $p>2$ it holds that $F\in L^p(\Omega)$ and $\mathbb{D}^l F \in L^p([0,T]^l\times\Omega;(\R^m)^{\times l})$ for $l=1,\dots,k$, then $F \in \mathbb{D}^{k,p}$.
\end{lemma}

The following shows how it is possible to apply Malliavin calculus to RSDEs, to obtain the Malliavin differentiability of solutions.
Moreover, we also get that the Malliavin derivative of these solutions is itself the solution to a linear RSDE. 

\begin{thm}[]  \label{thm:malliavincalculusforRSDEs} 
Let $X=(X_t)_{t\in [0,T]}$ be as in \eqref{eq:solutionRSDEinmalliavinforRSDEs} and let us fix an arbitrary $t \in (0,T]$. Then, for any $i=1,\dots,d$, $$X_t^i \in \bigcap_{p \ge 1} \mathbb{D}^{1,p} \quad \text{}$$ Moreover, for $\lambda$-a.e.\ $\theta \in [0,T]$ and $\mathbb{P}$-a.s.,  \begin{equation*} \label{eq:malliavinderivativeequaltothesolutionofaRSDE}
        \mathbb{D}_\theta X_t = \begin{cases}
            Y^\theta_t & \text{if $\theta \le t$} \\
            0 & \text{otherwise,}
        \end{cases}
\end{equation*} where $Y^\theta=(Y^\theta_s)_{s \in [\theta,T]}$ is the unique $\R^{d \times m}$-valued solution of the following linear RSDE: \begin{equation} \label{eq:linearRSDEforthemalliavinderivative}\begin{aligned} Y^\theta_s&=\sigma(X_\theta) + \int_\theta^s Db(X_r)Y^\theta_r dr+ \int_\theta^s D\sigma(X_r)Y^\theta_r dB_r+ \\ & \quad + \int_\theta^s (D\beta(X_r),D^2\beta(X_r)\beta(X_r))Y^\theta_r d\mathbf{Z}_r \qquad s\in [\theta,T].
\end{aligned}
\end{equation}
\end{thm}

\begin{proof}
    We argue by approximation, considering Lemma \ref{lemma:Malliavin via approximation}. Let $(Z^n)_n \subseteq C^\infty([0,T];\R^n)$ be any sequence of smooth rough paths whose rough path lifts $(\mathbf{Z}^n=(Z^n,\mathbb{Z}^n))_n $ converge to $\mathbf{Z}$ in the rough path topology, i.e.\ $\rho_\alpha(\mathbf{Z}^n,\mathbf{Z}) \to 0$ as $n \to +\infty$. Such a sequence exists since $\mathbf{Z}$ is geometric, and the argument of this proof does not depend on the choice of such a sequence. By construction there is a constant $C>0$ such that \begin{equation} \label{eq:constantC}
    \sup_{n \in \mathbb{N}} |\mathbf{Z}^n|_\alpha  + |\mathbf{Z}|_\alpha  \le C.
    \end{equation} 
    For any $n \in \mathbb{N}$, we denote by $X^n=(X^n_t)_{t \in [0,T]}$ the solution of the following RSDE driven by the smooth rough path $\mathbf{Z}^n$: \begin{equation}  \label{eq:approximatingsolutionRSDE}\begin{aligned} dX^n_t&=b(X^n_t)dt+\sigma(X^n_t)dB_t+(\beta(X^n_t),D\beta(X^n_t)\beta(X^n_t))d\mathbf{Z}^n_t \qquad t \in [0,T] \\ X^n_0&=x_0.\end{aligned}
    \end{equation} 
    {
    From the uniqueness part of the stochastic sewing lemma \cite[Theorem 2.9]{FHL21} it is easy to see that, $\mathbb{P}$-a.s.\ and for any $t \in [0,T]$, \begin{equation*}
        \int_0^t (\beta(X^n_r),D\beta(X^n_r)\beta(X^n_r))d\mathbf{Z}^n_r = \int_0^t \beta(X^n_r) \dot{Z}^n_r dr
    \end{equation*}
    where the integral on the left-hand side is a rough stochastic integral in the sense of \cite{FHL21} and the one on the right-hand side is a classical Riemann-Stieltjes integral, with $\dot{Z}^n_t = \frac{d Z^n_t}{dt}$.
    In particular, $X^n$ is (indistinguishable from) the unique solution to the following purely stochastic differential equation: 
    \begin{equation*}  
    \begin{aligned} dX^n_t&=(b(X^n_t) + \beta(X^n_t) \dot{Z}^n_t)dt+\sigma(X^n_t)dB_t \qquad t \in [0,T] \\ X^n_0&=x_0.\end{aligned}
    \end{equation*} 
    The map $(t,x) \mapsto b(x) + \beta(x) \dot{Z}^n_t$ clearly satisfies assumptions (h1) - (h2) of \cite[Section~2.2]{nualart2006malliavin}. 
    Hence, by the classical result on Malliavin differentiability for solutions to SDEs \cite[Theorem 2.2.1]{nualart2006malliavin}, we have $(X^n_t)^i \in \mathbb{D}^{1,2}$ 
    for any $t \in (0,T]$ and for any $i=1,\dots,d$. 
    Moreover, for $\lambda$-a.e.\ $\theta \in [0,T]$ and $\mathbb{P}$-a.s., 
    \begin{equation*}
        \mathbb{D}_\theta X^n_t = \begin{cases}
            Y^{n,\theta}_t & \text{if $\theta \le t$} \\
            0 & \text{otherwise}
        \end{cases}
    \end{equation*} 
    where $Y^{n,\theta}=(Y^{n,\theta}_s)_{s   \in [\theta,T]}$ is the (unique) $\R^{d\times m}$-valued solution of the following linear (R)SDE: \begin{equation*} \label{eq:equationforY^n}\begin{aligned} dY^{n,\theta}_r& = (Db(X^n_r) + \beta(X^n_t) \dot{Z}^n_t) Y^{n,\theta}_r dt + D\sigma(X^n_r)Y^{n,\theta}_r dB_r \quad r \in [\theta,T] \\
    & =Db(X^n_r)Y^{n,\theta}_r dr+D\sigma(X^n_r)Y^{n,\theta}_r dB_r+(D\beta(X^n_r),D^2\beta(X^n_r)\beta(X^n_r))Y^{n,\theta}_r d\mathbf{Z}^n_r,  \\ Y^{n,\theta}_\theta &=\sigma(X^n_\theta) \in \R^{d\times m}. \end{aligned}
   \end{equation*}
   } 
   Let $t \in (0,T]$ and let $i=1,\dots,d$ be fixed.
    Stability results for RSDEs (see e.g.\ \cite[Theorem 4.9]{FHL21}) show that $X_t^{n,i} \to X_t^i$ in $L^2(\Omega;\R^d)$, since $\rho_\alpha(\mathbf{Z}^n,\mathbf{Z})\to 0$ as $n\to +\infty$. 
    Let $K_C>0$ be the constant defined as in Lemma \ref{lemma:consistencyandaprioriforstochasticlinearvectorfields} and let $M \ge 0$ be such that \begin{equation*}
        |b|_{{C}^1_b} + |\sigma|_{{C}^1_b} + (1+|\beta|_{{C}^3_b})^2 + K_C + C \le M.
    \end{equation*}
    By construction, $M$ is independent of $n$.
    Taking into account Theorem \ref{thm:aprioriestimatelinearRSDEs} with $F \equiv 0$, we can infer that 
    \begin{equation*}
        \begin{aligned}
            \E\big(\int_0^t |\mathbb{D}_\theta X_t^{n,i}|_{\R^m}^2 d\theta\big) &\le \E\big( \int_0^t |\mathbb{D}_\theta X_t^{n}|_{\R^{d\times m}}^2 d\theta\big) 
            = \int_0^t \|Y_t^{n,\theta}\|_2^2 d\theta \lesssim_{M,T} |\sigma|_\infty^2 < +\infty, 
            \end{aligned}
    \end{equation*} uniformly in $n \in \mathbb{N}$. This is enough to conclude that $X_t^i \in \mathbb{D}^{1,2}$ and that $\mathbb{D}X^n_t \rightharpoonup \mathbb{D}X_t$ in $L^2([0,T]\times \Omega;\R^{d\times m})$. \\
    Being $X_t$ $\mathcal{F}_t$-measurable, from standard properties of the Malliavin derivative (see e.g.\ \cite[Corollary 1.2.1]{nualart2006malliavin}) we immediately get that $\mathbb{D}_\theta X_t=0$, for $\lambda$-a.e.\ $\theta \in (t,T]$. 
    Combining Theorem~\ref{thm:stabilityforlinearRSDEs} with Lemma \ref{lemma:stabilityforstochasticlinearvectorfields} below, we can also deduce that
    
    \begin{equation*}
        Y^{n,\theta}_t \to Y^{\theta}_t   \qquad \text{in $L^2(\Omega;\R^{d\times m})$}
    \end{equation*}
    as $n \to +\infty$. 
    In particular this implies, for $\lambda$-a.e.\ $\theta \in [0,t]$,  \begin{equation*}
        Y^\theta_t =\lim_{n \to +\infty} \mathbb{D}_\theta X^{n}_t \qquad \text{in $L^2(\Omega;\R^{d\times m})$}
    \end{equation*} and Theorem \ref{thm:aprioriestimatelinearRSDEs} shows that $\|\mathbb{D}_\theta X^{n}_t\|_{2} = \|Y^{n,\theta}_t\|_{2} \lesssim_{M,\alpha,T}  |\sigma|_\infty$.
    We can now define the function $\mathcal{Y}_t:[0,T] \to L^2(\Omega;\R^{d\times m})$ by \begin{equation*}
        \mathcal{Y}_t (\theta) := \begin{cases}
            Y_t^\theta & \text{if $\theta \le t$} \\
            0 & \text{otherwise}
        \end{cases}.
    \end{equation*}  
    By applying dominated convergence, we obtain that $$\mathcal{Y}_t \in L^2([0,T];L^2(\Omega;\R^{d\times m})) \equiv L^2([0,T]\times\Omega;\R^{d \times m})$$ and $\mathbb{D}_\cdot X^{n}_t \to \mathcal{Y}_t$ in $L^2([0,T]\times\Omega;\R^{d \times m})$.
    By uniqueness of the weak limit, it follows that $\mathbb{D}X_t = \mathcal{Y}_t$ in  $L^2([0,T]\times\Omega;\R^{d \times m})$, namely \begin{equation*}
        \mathcal{Y}_t(\theta) = \mathbb{D}_\theta X_t \qquad \text{$\mathbb{P}$-a.s.}
    \end{equation*} for $\lambda$-a.e.\ $\theta \in [0,T]$. 
    Theorem \ref{thm:aprioriestimatelinearRSDEs} actually shows that, for any $p>2$, \begin{equation*}
        \E\big(\int_0^T|\mathbb{D}_\theta X_t^{n,i}|^p d\theta\big) = \int_0^T \|(Y_t^\theta)^i\|_p^p d\theta <+\infty.
    \end{equation*} 
    This implies that $X_t^i \in \mathbb{D}^{1,p}$ for any $p > 2$.
\end{proof}

The following lemma is needed in the previous proof. 
\begin{lemma}[] \label{lemma:stabilityforstochasticlinearvectorfields}
    Consider the following stochastic linear vector fields \begin{align*} 
    G&:= Db(X_\cdot) \qquad G^n:= Db(X^n_\cdot), \ n \in \mathbb{N}\\ \label{eq:vectorfieldsS} S&:= D\sigma(X_\cdot) \qquad S^n:= D\sigma(X^n_\cdot), \ n \in \mathbb{N}
\end{align*} and \begin{equation*} \begin{aligned}
    (f,f') &:= (D\beta(X_\cdot),D^2\beta(X_\cdot)\beta(X_\cdot)) \\
    (f^n,(f^n)') &:= (D\beta(X^n_\cdot),D^2\beta(X^n_\cdot)\beta(X^n_\cdot)), \ n \in \mathbb{N} . 
\end{aligned}
\end{equation*} 
For any $p \in [2,\infty)$, for any $\gamma \in (\frac{1}{3},\alpha)$ and for any $n \in \mathbb{N}$, let \begin{equation*} \begin{aligned}
        \theta^n_{p,\gamma} &:= \E\Big( \int_0^T |G^n_r-{G}_r|_{\mathscr{L}}^p dr \Big)^\frac1p + \E\Big( \int_0^T |S^n_r-{S}_r|_{\mathscr{L}}^p dr \Big)^\frac1p  + \| f^n , (f^n)' ; f, f' \|_{\mathscr{L};p}  \\ & \quad + \| f^n , (f^n)' ; f, f' \|_{Z^n,Z;2\gamma;p;\mathscr{L}} 
\end{aligned}
\end{equation*} where $R^{n,{f^n}}_{s,t}=\delta f^n_{s,t}-(f^n)'_s\delta Z^n_{s,t}$. Then $ \theta^n_{p,\gamma} \to 0$ as $n \to +\infty$.  
\end{lemma}
\begin{proof}
    The proof essentially follows from the fact that $(\mathbf{Z}^n)_n$ is a sequence of smooth rough paths approximating $\mathbf{Z}$ and that - by taking into account the continuous dependence result contained in \cite[Theorem 4.9]{FHL21} -
    {
    $X^n_t$ converges to $X_t$ as $n \to +\infty$ in $L^p(\Omega;\R^d)$ uniformly in $t \in [0,T]$. 

    Being $Db:\R^d \to \mathscr{L}(\R^d,\R^d)$ continuous, we get that $Db(X^n_t) \to Db(X_t)$ as $n \to +\infty$ in probability and, by dominated convergence, also in $L^p(\Omega \times [0,T]; \mathscr{L}(\R^d,\R^d))$. In other words, $\E(\int_0^T |Db(X^n_t) - Db(X_t)|_\mathscr{L}^p dr) \to 0$ as $n \to +\infty$. 
    The same argument can be applied if we replace $Db(\cdot)$ with $D\sigma(\cdot), D\beta (\cdot)$ or $D^2\beta(\cdot)\beta(\cdot)$. \\
    }
    By means of Lemma \ref{lemma:consistencyandaprioriforstochasticlinearvectorfields}, we deduce that $\|f^n,(f^n)'\|_{ \mathbf{D}^{2\alpha}_{Z^n} L^{p,\infty}\mathscr{L}} \le K_C$ uniformly in $n \in \mathbb{N}$. 
    Hence, arguing as in Corollary \ref{coroll:technicalconvergence}, we can reach the conclusion.
\end{proof}

In the case of smooth coefficients, solutions of RSDEs are infinitely many times Malliavin differentiable. 
\begin{thm} Let $X=(X_t)_{t\in [0,T]}$ be the solution of \eqref{eq:solutionRSDEinmalliavinforRSDEs}. Assume $b \in {C}^\infty_b(\R^d;\R^d)$, $\sigma\in {C}^\infty_b(\R^d;\mathscr{L}(\R^m,\R^d))$ and $\beta \in {C}^\infty_b(\R^d;\mathscr{L}(\R^n,\R^d))$.  Then, for any $t \in (0,T]$ and for any $i=1,\dots,d$, 
    \begin{equation*}
        X_t^i \in \bigcap_{k \in \mathbb{N}_{\ge 1}} \bigcap_{p \ge 1} \mathbb{D}^{k,p}.
    \end{equation*}
\end{thm}

\begin{proof}
  Let $t \in (0,T]$ and $i=1,\dots,d$. By strong induction on $k \in \mathbb{N}_{\ge 1}$
  we prove that \begin{equation*}
      X_t^i \in \mathbb{D}^{k, p} \quad \text{for any $p\ge 2$}.
  \end{equation*}
  Operatively, we first prove that $X_t^i \in \mathbb{D}^{k, 2}$, arguing by approximation. Then, knowing that lower order Malliavin derivatives satisfy suitable linear RSDEs, we prove $X_t^i \in
  \mathbb{D}^{k,p}$ for any $p > 2$.
  Let us fix any approximating sequence $(Z^n)_n$ of smooth paths, whose
  canonical rough path lifts converge to $\mathbf{Z}$ in the $\alpha$-rough path metric as $n\to +\infty$. Let $X^n = (X^n_t)_{t \in [0,T]}$ be the solution of the RSDE driven by $\mathbf{Z}^n = (Z^n, \mathbb{Z}^n)$, as defined in \eqref{eq:approximatingsolutionRSDE}. Recall that, by standard stability results for RSDEs, we know that $X_t^{n,i} \rightarrow X_t^i$ in $L^p
  (\Omega ; \mathbb{R}^d)$ for any $p \geqslant 2$. 
  By $\mathbb{D}^k X_t$ we denote the iterated derivative of order $k$.

  The case $k = 1$ has been already proved in Theorem \ref{thm:malliavincalculusforRSDEs}, under less regularity assumptions on the coefficients. In particular, we have also deduced that, for $\lambda$-almost every $\theta \in [0,t]$, $\mathbb{D}_\theta X_t = Y^{1,\theta}_t$, where $Y^{1,\theta}=(Y^{1,\theta}_t)_{t \in [\theta,T]}$ is the solution of the following linear RSDE: \begin{multline*}  Y^{1,\theta}_s =\sigma(X_\theta) + \int_0^s Db(X_r)Y^{1,\theta}_r dr+ \int_0^s D\sigma(X_r)Y^{1,\theta}_r dB_r \\ + \int_0^s(D\beta(X_r),D^2\beta(X_r)\beta(X_r))Y^{1,\theta}_r d\mathbf{Z}_r, \qquad s \in [\theta,T] .
\end{multline*}

  Assume the inductive hypothesis holds up to $k - 1$, for some $k \ge 2$.
  By \cite[Theorem 2.2.2]{nualart2006malliavin}, it follows that, for any $n \in \mathbb{N}$, $X^{n,i}_t \in \mathbb{D}^{k, 2}$. Moreover, for $\lambda^{\otimes k}$-a.e.\ $(\theta_1,\dots,\theta_k) \in [0,t]^k$,  $$\mathbb{D}^k_{\theta_1,\dots,\theta_k}X_t^n=Y_t^{n,k,\theta_1,\dots,\theta_k}$$ where $Y^{n,k,\theta_1,\dots,\theta_k}=(Y_s^{n,k,\theta_1,\dots,\theta_k})_{s\in [\theta_1\vee\dots\vee\theta_k,T]}$ denotes the (unique) $\R^{d \times km}$-valued solution of the following linear (R)SDE: \begin{align*}
    \label{eq:linear kth RSDE} Y_s^{n,k, \theta_1, \ldots, \theta_k} & = 
    \xi^{n,k, \theta_1, \ldots, \theta_k} + F_s^{n,k, \theta_1, \ldots, \theta_k}
    + \int_{\theta_1 \vee \ldots \vee \theta_k}^s  Db(X^n_r) Y_r^{n,k,
    \theta_1, \ldots, \theta_k} d r \\
    &\quad +   \int_{\theta_1 \vee \ldots \vee \theta_k}^s D \sigma (X^n_r) Y_r^{n,k,
    \theta_1, \ldots, \theta_k} d B_r  \\
    &\quad + \int_{\theta_1 \vee \dots \vee \theta_k}^s (D \beta (X^n_t), D^2
    \beta (X^n_t) \beta (X^n_t)) Y_r^{n,k, \theta_1, \ldots, \theta_k} d
    \mathbf{Z}^n_r . 
  \end{align*}
  {
  For the explicit (recursive) definition of $\xi^{n,k, \theta_1, \ldots, \theta_k}$ and $F^{n,k, \theta_1, \ldots, \theta_k}$, it is convenient to recall some notation introduced in the proof of \cite[Theorem 2.2.2]{nualart2006malliavin}. 
  For any subset \linebreak $I=\{i_1 < \dots < i_\kappa\}$ of $\{1,\dots,k\}$, denote by $\theta(I)=\theta_{i_1}, \dots, \theta_{i_\kappa}$ and by $|I| = \kappa$. 
  Define 
  
  \begin{align*}
      A^n_k (t,\theta_1,\dots,\theta_k) &= \sum_{ \substack{ I_1, \dots, I_K : \\ K \ge 2, \,  I_1 \cup \dots \cup I_K = \{1,\dots,k\} } } D^K \sigma(X^n_t) \cdot Y^{n,|I_1|,\theta(I_1)}_t \dots Y^{n,|I_K|,\theta(I_K)}_t
      \\
      B^n_k (t,\theta_1,\dots,\theta_k) &= \sum_{ \substack{ I_1, \dots, I_K : \\ K \ge 2, \,  I_1 \cup \dots \cup I_K = \{1,\dots,k\} } } D^K b(X_t^n) \cdot Y^{n,|I_1|,\theta(I_1)}_t \dots Y^{n,|I_K|,\theta(I_K)}_t
      \\
      \Gamma^n_k (t,\theta_1,\dots,\theta_k) &= \sum_{ \substack{ I_1, \dots, I_K : \\ K \ge 2, \,  I_1 \cup \dots \cup I_K = \{1,\dots,k\} } } D^K \beta(X^n_t) \cdot Y^{n,|I_1|,\theta(I_1)}_t \dots Y^{n,|I_K|,\theta(I_K)}_t
      . 
  \end{align*}
  Here $\cdot$ denotes a formal product; recall that higher order Fréchet derivatives of $b,\sigma,\beta$ can be seen as multilinear operators.
  From \cite[equation (2.54)]{nualart2006malliavin}, we can now write 
  \begin{equation} \label{eq:definitionofxi^nandF^n}
  \begin{aligned}
       \xi^{n,k,\theta_1,\dots,\theta_k} &= \sum_{\ell=1}^k \sum_{  \substack{ I_1, \dots, I_L : \\  I_1 \cup \dots \cup I_L = \{1,\dots,k\} \setminus \{\ell\} } } D^L \sigma (X^n_{\theta_\ell}) \cdot \Big(Y^{n,|I_1|,\theta(I_1)}_{\theta_\ell}  \mathbf{1}_{\max\{\theta(I_1)\} \le \theta_\ell} \\
       & \qquad \dots Y^{n,|I_L|,\theta(I_L)}_{\theta_\ell} \mathbf{1}_{\max\{\theta(I_L)\} \le \theta_\ell} \Big)  \\
        F_s^{n,k, \theta_1, \ldots, \theta_k} &= \int_{\theta_1 \vee \ldots \vee \theta_k}^s B^n_k(r,\theta_1,\dots,\theta_k) dr + \int_{\theta_1 \vee \ldots \vee \theta_k}^s A^n_k(r,\theta_1,\dots,\theta_k) dB_r  \\ 
        & \quad + \int_{\theta_1 \vee \ldots \vee \theta_k}^s (\Gamma^n_k(r,\theta_1,\dots,\theta_k), \Gamma^n_k(r,\theta_1,\dots,\theta_k) ') d\mathbf{Z}^n_r  
    \end{aligned}
    \end{equation}
    where \begin{align*}
    \Gamma^n_k(t,\theta_1,\dots,\theta_k) ' &= \sum D^{K+1} \beta(X^n_t) \beta(X^n_t) \cdot Y^{n,|I_1|,\theta(I_1)}_t \dots Y^{n,|I_K|,\theta(I_K)}_t \\
        & \quad + \sum D^K \beta(X^n_t) \cdot \Big( \sum_{j=1}^K Y^{n,|I_1|,\theta(I_1)}_t \dots Y^{n,|I_{j-1}|,\theta(I_{j-1})}_t \\
        & \qquad  \Big( D\beta(X^n_t) Y_t^{n,|I_j|,\theta(I_j)} + \Gamma^n_{|I_j|}(t,\theta(I_j)) \Big) Y^{n,|I_{j+1}|,\theta(I_{j+1})}_t \dots Y^{n,|I_K|,\theta(I_K)}_t \Big)
    \end{align*}
    and sums are taken over all the intervals $I_1, \dots, I_K$, $K \ge 2$, such that $I_1 \cup \dots \cup I_K = \{1,\dots,k\}$. 
    }
  Taking into account Lemma \ref{lemma:aprioriboundsinthesmoothcase} below and recalling that the constant therein defined is uniform over
  bounded sets of $\mathscr{C}^{\alpha}([0,T];\R^n)$, we have \begin{equation*}
      \sup_{n \in \mathbb{N}} \int_0^T \ldots \int_0^T \mathbb{E} (|
     \mathbb{D}^k_{\theta_1, \ldots, \theta_k} X_t^{n,i} |^2) d \theta_1 \ldots d
     \theta_k < + \infty .
  \end{equation*}
  From Lemma \ref{lemma:Malliavin via approximation} it follows that $X_t^i \in
  \mathbb{D}^{k, 2}$ and $\mathbb{D}^k_{\cdot} X^n_t \rightharpoonup
  \mathbb{D}^k_{\cdot} X_t$ in $L^2 ([0, T]^k \times \Omega; \mathbb{R}^{d
  \times km}) \equiv L^2 ([0, T]^k ; L^2 (\Omega ;
  \mathbb{R}^{d \times km}))$. Thanks to this
  identification, it is possible to apply dominated convergence and prove
  that the Malliavin derivative solves the expected linear RSDE.
  For any $(\theta_1,\dots,\theta_k)\in[0,T]^k$, let us define $Y^{k,\theta_1,\dots,\theta_k}=(Y^{k,\theta_1,\dots,\theta_k}_s)_{s\in [\theta_1\vee\dots\vee\theta_k,T]}$ to be the solution of the following linear RSDE: \begin{align*}
     Y_s^{k, \theta_1, \ldots, \theta_k} & = 
    \xi^{k, \theta_1, \ldots, \theta_k} + F_s^{k, \theta_1, \ldots, \theta_k}
    + \int_{\theta_1 \vee \ldots \vee \theta_k}^s  Db(X_r) Y_r^{k,
    \theta_1, \ldots, \theta_k} d r  \\
    &\quad + \int_{\theta_1 \vee \ldots \vee \theta_k}^s D \sigma (X_r) Y_r^{k,
    \theta_1, \ldots, \theta_k} d B_r +\\
    & \quad  + \int_{\theta_1 \vee \dots \vee \theta_k}^s (D \beta (X_t), D^2
    \beta (X_t) \beta (X_t)) Y_r^{k, \theta_1, \ldots, \theta_k} d
    \mathbf{Z}_r . 
  \end{align*}
  The explicit definitions of $\xi^{k, \theta_1, \ldots, \theta_k}$ and $F^{k, \theta_1, \ldots, \theta_k}$ can be obtained from \eqref{eq:definitionofxi^nandF^n} by removing all superscripts $n$ and by considering $\mathbf{Z}$ instead of $\mathbf{Z}^n$ in the rough stochastic integral. 
  We define the function $\mathcal{Y}=\mathcal{Y}^{k}_t : [0,
  T]^k \rightarrow L^2 (\Omega ; \mathbb{R}^{d \times km})$ by \begin{equation*}
      \mathcal{Y}^k_t ((\theta_1, \ldots, \theta_k)) := \begin{cases}
          Y^{k, \theta_1,
     \ldots, \theta_k}_t & \text{if $\theta_1\vee\dots\vee\theta_k \le t$} \\
     0 & \text{otherwise}
      \end{cases}.
  \end{equation*}
  From Lemma \ref{lemma:stabilityforlinearsmoothRSDEs} below, we deduce that 
  $\mathbb{D}_{\cdot}^k X_t^n \to \mathcal{Y}^k_t (\cdot)$ in $L^2
  (\Omega ; \mathbb{R}^{d \times km})$ $\lambda^{\otimes k}$-a.s.\ in $[0,T]^k$. Moreover, Lemma \ref{lemma:aprioriboundsinthesmoothcase} implies that \begin{equation*}
      \| \mathbb{D}_{\theta_1, \ldots, \theta_k}^k X_t^n \|_2 \le
     \Lambda_{k, 2} < +\infty
  \end{equation*}
  for $\lambda^{\otimes k}$-almost every $(\theta_1, \ldots, \theta_k) \in [0,T]^k$.
  From dominated convergence we conclude that $\mathcal{Y}^k_t \in
  L^2 ([0, T]^k ; L^2 (\Omega ; \mathbb{R}^{d \times km}))$ and \begin{equation*}
      \mathbb{D}^k X_t^n \rightarrow \mathcal{Y}^k_t \quad \text{in $L^2 ([0, T]^k ;
  L^2 (\Omega ; \mathbb{R}^{d \times km}))$}.
  \end{equation*}  
  By uniqueness of the weak limit, \begin{equation*}
      \mathbb{D}^k X_t =\mathcal{Y}^k_t \in L^2 ([0, T]^k ; L^2
     (\Omega ; \mathbb{R}^{d \times km})) \equiv L^2
     (\Omega ; L^2 ([0, T]^k ; \mathbb{R}^{d \times km})) .
  \end{equation*}
  This last equality in particular implies that - taking also into account the
  inductive hypothesis - $\mathbb{D}^l X_t^i \in L^p (\Omega ; L^2 ([0, T]^l ; \mathbb{R}^{d \times km}))$ for any $p > 2$ and for
  any $l = 1, \dots, k$. This allows us to conclude that $X_t^i \in
  \mathbb{D}^{k, p}$ for any $p > 2$.
\end{proof}

The following two lemmas are needed in the previous proof. 

\begin{lemma}\label{lemma:aprioriboundsinthesmoothcase}
     For any $k \in \mathbb{N}_{\ge 1}$, for any $n\in \mathbb{N}$ and for any $p \ge 2$, it holds \begin{equation*}
         \sup_{t \in [\theta_1 \vee \ldots \vee \theta_k, T]} \|Y_t^{n,k, \theta_1,\ldots, \theta_k} {\|_p}  +\| (Y_{\cdot}^{n,k, \theta_1, \ldots,\theta_k}, (Y_{\cdot}^{n,k, \theta_1, \ldots, \theta_k})') \|_{\mathcal{D}_{Z^n}^{2 \alpha} L^p} \le \Lambda_{k, p}  < + \infty,
     \end{equation*}
  where the constant $\Lambda_{k, p} = \Lambda_{k, p} (k, p, \alpha, T, C, | b |_{C^k_b}, | \sigma |_{C^k_b}, | \beta |_{C^{3 + k}_b})
  \ge 0$ is independent of $n$ and on $(\theta_1, \dots, \theta_k) \in [0, T]^k$, and $C\ge0$ is defined as in \eqref{eq:constantC}.   
\end{lemma}

\begin{proof} From Theorem \ref{thm:aprioriestimatelinearRSDEs} it is clear that a priori estimates for linear RSDEs only depend on the length of the time interval and they are uniform over bounded sets in $\mathscr{C}^\alpha([0,T];\R^n)$.
Therefore, we expect some a priori bounds on $Y^{n,k,\theta_1,\dots,\theta_k}_\cdot$ that are uniform in $(\theta_1, \ldots, \theta_k)$ and in $n$. \\
Going more into detail, the case $k = 1$ is an easy application of Theorem \ref{thm:aprioriestimatelinearRSDEs} without forcing term. \\
Assume, by induction, that the conclusion holds for $k - 1$. As a general fact, every time we consider a rough It\^o process in $L^p$ (see Definition \ref{def: rough ito process}) of the form  \begin{equation*}
     F_t^{\theta} = \int_{\theta}^t \eta_r d r + \int_{\theta}^t \nu_r d B_r + \int_{\theta}^t (\gamma_r, \gamma'_r) d \mathbf{Z}_r, 
\end{equation*}
we can define ${F_t^{\theta}}':=\gamma_t$ and, from the estimates coming from the stochastic sewing lemma \cite[Theorem 2.9]{FHL21}
, it holds \begin{equation*}
     \| ( F^{\theta}_t {, F^{\theta}_t}' ) \|_{\mathscr{D}_Z^{2\alpha} L^p} \lesssim \sup_{t \in [\theta, T]} \| \eta_t \|_p + \sup_{t \in [\theta, T]} \|\nu_t \|_p + \sup_{t \in [\theta, T]} \| \gamma_t \|_p +\| (\gamma, \gamma')
     \|_{\mathscr{D}_Z^{2 \alpha} L^p}.
\end{equation*}
    The implicit constant depends on $\alpha, p, T$ and it is uniform on $\{ | \mathbf{Z} |_{\alpha} \leqslant R \}$, for any $R > 0$. We can now fix an arbitrary $p\ge 2$. From Theorem \ref{thm:aprioriestimatelinearRSDEs} it holds that \begin{align*}
        &\sup_{t \in [\theta_1 \vee \dots \vee \theta_k, T]} \|Y_t^{n,k, \theta_1,
     \dots, \theta_{k + 1}} \|_p  +\| (Y^{k, \theta_1, \dots,\theta_k}, (Y^{n,k,\theta_1, \ldots, \theta_k})')\|_{\mathscr{D}_{Z^n}^{2 \alpha}L^p} \\
     &\quad \lesssim \| \xi^{n,k, \theta_1, \dots,\theta_k} \|_p + \| (F^{n,k, \theta_1, \dots, \theta_k}, (F^{n,k,\theta_1, \ldots, \theta_k})') \| {_{\mathscr{D}_{Z^n}^{2 \alpha} L^p}}
    \end{align*}
  where the implicit constant is independent of $\theta_1,\dots,\theta_k$. The conclusion follows by the inductive hypothesis, taking also into account the explicit form of $\xi^{n,k, \theta_1, \dots,\theta_k}$ and $F^{n,k, \theta_1, \dots, \theta_k}$ in \eqref{eq:definitionofxi^nandF^n}, and possibily applying H\"older's inequality.
  \end{proof}

\begin{lemma}\label{lemma:stabilityforlinearsmoothRSDEs}
    For any $k \in \mathbb{N}_{\ge 1}$, for any $p\ge 2$ and for any $(\theta_1, \ldots, \theta_k)\in [0, T]^k$, it holds that \begin{equation*}
         \sup_{t \in [\theta_1 \vee \dots \vee \theta_k, T]} \|Y_t^{n, k,\theta_1, \ldots, \theta_k} - Y_t^{k, \theta_1, \ldots, \theta_k} \|_p \to 0 \qquad \text{as $n \rightarrow + \infty$.}
    \end{equation*}
\end{lemma}

\begin{proof}
  The case $k=1$ easily holds, combining Theorem \ref{thm:stabilityforlinearRSDEs} and Lemma \ref{lemma:stabilityforstochasticlinearvectorfields}.
  Let us fix $k \in \mathbb{N}_{\ge 1}$ and let us assume that the statement holds up to
  $k - 1$. Let us fix $p \ge 2$ and $(\theta_1,
  \dots, \theta_k) \in [0, T]^k$.
  By Theorem \ref{thm:stabilityforlinearRSDEs} and considering also Lemma \ref{lemma:stabilityforstochasticlinearvectorfields}, it is
  sufficient to prove that \begin{equation*}
      \| \xi^{n, k, \theta_1, \ldots, \theta_k} - \xi^{k, \theta_1, \ldots,
     \theta_k} \|_p \rightarrow 0
  \end{equation*}
  and \begin{align*}
     &\|F^{n, k, \theta_1, \dots, \theta_k}, \Gamma_k^n(\cdot,\theta_1,\dots,\theta_k) ; F^{k, \theta_1, \dots,\theta_k}, \Gamma_k(\cdot,\theta_1,\dots,\theta_k) \|_{Z^n,Z;2\gamma;p} \\
     & \quad + \sup_{t\in [0,T]} \|\Gamma_k^n(t,\theta_1,\dots,\theta_k)-\Gamma_k(t,\theta_1,\dots,\theta_k)\|_p
     \to 0 \quad \text{as $n\to \infty$}.
  \end{align*}
  Both
  the convergences hold thanks to the inductive assumptions and recalling the
  explicit form of those quantity. The first convergence is just a trivial
  application of Holder's inequality, while the second one holds also
  considering that
  \[ \sup_n \| (F^{n, k, \theta_1, \ldots, \theta_k}, \Gamma_k^n(\cdot,\theta_1,\dots,\theta_k)) \|_{\mathscr{D}_{Z^n}^{2 \alpha} L^p} < + \infty \]
  and applying Corollary \ref{coroll:technicalconvergence}. 
\end{proof}

\section{Hörmander-type results}
\label{sec: hoermander}

\subsection{Existence of densities under H\"ormander condition} \label{sec:existencedensity_Hoermander}
In this section we prove the existence of densities for the law of the solution of equation \eqref{eq:RSDE} under H\"ormander condition. 

For a random vector $X = (X^1, \dots, X^d)$, such that  $X^i \in \mathbb{D}^{1,1}$ for every $i=1,\dots,d$, we define the Malliavin matrix as 
\begin{equation}
\label{eq: malliavin matrix}
    \gamma_X := ( \langle \mathbb{D}_{\cdot}X^i, \mathbb{D}_{\cdot}X^j \rangle_{L^2([0,T])} )_{i,j=1,\dots,d}.
\end{equation}

From classical Malliavin theory, if $X = (X^1, \dots, X^d)$ is such that $X^i\in \mathbb{D}^{1,p}$, for every $i=1,\dots,d$ and for some $p>1$, then the law of $X$ admits a density with respect to the Lebesgue measure when the Malliavin matrix $\gamma_X$ is a.s. invertible, see \cite[Theorem 2.1.2]{nualart2006malliavin}. Moreover, the density is smooth whenever $X$ is Malliavin smooth and the inverse of $\gamma_X$ has moments of all orders, see \cite[Theorem 3.2]{hairer2011malliavins} or \cite[Theorem 2.1.4]{malliavin2015stochastic}.

When we work with $(X_t)_{t\in[0,T]}$ solution to \eqref{eq:RSDE}, it is convenient to introduce the reduced Malliavin matrix of $X_t$. In order to define it, we first formally derive equation \eqref{eq:RSDE} in the initial condition.
Let $J=(J_t)_{t \in [0,T]}$ be the continuous and adapted $\mathscr{L}(\R^d,\R^d)$-valued stochastic process which solves the following linear RSDE: \begin{equation} \label{eq:equationforJ}\begin{aligned} &dJ_t=Db(X_t)J_tdt+D\sigma(X_t)J_tdB_t+(D\beta(X_t),D^2\beta(X_t)\beta(X_t))J_td\mathbf{Z}_t, \quad t \in [0,T] \\ &J_0=id_{\R^d}.\end{aligned}\end{equation} 
Moreover, let us denote by $I=(I_t)_{t \in [0,T]}$ the $\mathscr{L}(\R^d,\R^d)$-valued solution of the following linear RSDE:
\begin{equation} \label{eq:equationforI}\begin{aligned} &dI_t=-I_t(Db(X_t)-D\sigma^2(X_t))dt-I_tD\sigma(X_t)dB_t-I_t(D\beta(X_t),D^2\beta(X_t)\beta(X_t))d\mathbf{Z}_t, \, t \in [0,T] \\ &I_0=id_{\R^d}. \end{aligned}\end{equation}
Notice that both equation \eqref{eq:equationforJ} and \eqref{eq:equationforI} are well-posed by Theorem \ref{thm:wellposednesslinearRSDEs}.
As an application of the rough It\^o formula, equation \eqref{eq: product formula}, we have that $I_t$ is the matrix inverse of $J_t$. Indeed,
for any $t \in [0,T]$ we have, $\mathbb{P}$-a.s.,
\begin{equation*}
    \begin{aligned}
        d(I J)_t &= I_t dJ_t + dI_t J_t + d\langle I,J\rangle_t =  \\ &= I_tDb(X_t)J_tdt+I_tD\beta(X_t))J_td\mathbf{Z}_t+I_t D\sigma(X_t)J_tdB_t \\ & \quad -I_t(Db(X_t)-D\sigma(X_t)^2)J_t dt+D\beta(X_t)J_td\mathbf{Z}_t-I_tD\sigma(X_t)J_tdB_t - I_t D\sigma(X_t)^2 J_t \\ &= 0. 
    \end{aligned}
\end{equation*} 
We define the reduced Malliavin matrix of $X_t$ as
\begin{equation}
    \label{eq: reduced malliavin matrix}
    C_t := \int_{0}^t I_s \sigma(X_s)\sigma^{\top}(X_s) I_s^{\top} \rd s, \qquad t\in[0,T].
    \end{equation}
As in the classical case, we have $\gamma_{X_t} = J_tC_tJ_t^{\top}$. Indeed, we first multiply equation \eqref{eq:equationforJ} to the right by $I_{\theta}\sigma(X_{\theta})$, for $\theta \in [0,t]$, to obtain that $J_tI_{\theta}\sigma(X_{\theta})$ solves the same linear equation as the Malliavin derivative $\mathbb{D}_{\theta}X_t$, equation \eqref{eq:linearRSDEforthemalliavinderivative}. By uniqueness of solutions to linear RSDEs, Theorem \ref{thm:wellposednesslinearRSDEs}, we have that, $\mathbb{P}$-a.s.
\begin{equation*}
    \mathbb{D}_{\theta}X_t = J_t I_{\theta}\sigma(X_{\theta}),
    \qquad 0\leq \theta \leq t \leq T.
\end{equation*}
Now, by definition, for $t\in [0,T]$,
\begin{equation*}
    \gamma_{X_t} = \int_0^t \mathbb{D}_{s}X_t(\mathbb{D}_{s}X_t)^{\top} \rd s
    = \int_0^t J_t I_{s}\sigma(X_s) \sigma(X_s)^{\top}I_{s}^{\top} J_t^{\top}\rd s
    = J_t C_t J_t^{\top}.
\end{equation*}

In order to prove that the Malliavin matrix is invertible, and the inverse has moments of all order, it is enough to prove the same properties for the reduced Malliavin matrix.
{To this end, notice that $J$ has finite moments of all orders because it solves a linear RSDE with initial condition $J_0 \in \bigcap_{p \ge 2} L^p$ (cf.\ Theorem \ref{thm:aprioriestimatelinearRSDEs}). }

Before diving into the proof of existence and smoothness of densities, we need an extra assumption on the rough path driving the equation. In short, if the rough path is just the lift of a (relatively) smooth process, it could be assimilated to a drift and it would help much less in spreading the diffiusion in all directions. Instead, we will work under the assumption that the rough path is truly rough in the sense of the following definition (see \cite[Definition 6.3]{FH20}).

\begin{defn}
    \label{def: truly rough}
    Let $\alpha \in (\frac{1}{3},\frac{1}{2}]$. Let $s\in [0,T)$, we say that $Z \in C^\alpha([0,T];\R^n)$
    is \textit{rough at time s} if for every $v \in \R^n \setminus \{0\}$ we have
    \begin{equation*}
        \overline{\lim}_{t\downarrow s} \frac{|v\cdot \delta Z_{s,t}|}{|t-s|^{2\alpha}} = \infty.
    \end{equation*}
    If ${Z}$ is rough at $s$ for every $s$ in a dense subset of $[0,T]$ we say that ${Z}$ is \textit{truly rough}.
\end{defn}

It is classical that, if a path is truly rough, then the Gubinelli derivative of any controlled path is unique, see \cite[Proposition 6.4]{FH20}. This is the case also for stochastically controlled path, as shown in the following proposition, which works for any Banach space $W$. However, we will only apply it in the finite dimensional case.

\begin{prop}
    \label{prop: uniqueness truly rough}
    Let $Z \in C^\alpha([0,T];\R^n)$
    be truly rough, for some $\alpha \in (\frac{1}{3},\frac{1}{2}]$. If \linebreak $(X,X^{\prime}) \in \mathscr{D}_{Z}^{2\alpha}L^{p,q}(W)$ and $X^{\prime}$ has continuous sample paths, then
    \begin{enumerate}
        \item If $\mathbb{E}_{\cdot}\delta X \in C^{2\alpha}L^q(W)$ then $X^{\prime}$ is indistinguishable from $0$.
        \item \textit{(Uniqueness of Gubinelli derivative)} If $X^1$ and $X^2$ are both Gubinelli derivatives of $X$, then they are indistinguishable.
    \end{enumerate}
\end{prop}

\begin{proof}

    In order to prove the first statement, we first notice that
    \begin{equation*}
        \sup_{0\leq s<t\leq T}\frac{\| X^{\prime}_s Z_{s,t}\|_{q}}{|t-s|^{2\alpha}}
        = \sup_{0\leq s<t\leq T}\frac{\|\mathbb{E}_s[X_{s,t} - R^X_{s,t}]\|_q}{|t-s|^{2\alpha}}
        < +\infty.
    \end{equation*}
    Let $1/q + 1/q^{\prime} = 1$ and $W^{\ast}$ be the dual of $W$. We take $w^{\ast} \in L^{q^{\prime}}(W^{\ast})$, with $\|w^{\ast}\|_{q^{\prime}} = 1$.
    Let us now apply the definiton of rough at time $s\in[0,T)$ with the vector $v := \mathbb{E}[(X_s^{\prime})^{\ast}w^{\ast}] \in \R^n$
    \begin{equation*}
        \overline{\lim}_{t\downarrow s} \frac{|v\cdot Z_{s,t}|}{|t-s|^{2\alpha}}
        =
        \overline{\lim}_{t\downarrow s} \frac{|\mathbb{E}[w^{\ast}(X^{\prime}_sZ_{s,t})]|}{|t-s|^{2\alpha}}
        \leq
        \|w\|_{q^{\prime}} 
        \sup_{0\leq s<t\leq T}\frac{\| X^{\prime}_s Z_{s,t}\|_{q}}{|t-s|^{2\alpha}}
        < +\infty.
    \end{equation*}
    This implies that $v = 0 \in \R^{n}$ and hence 
    \begin{equation*}
        \|X^{\prime}_s\|_{q} = \sup_{\|w^{\ast}\|_{q^{\prime}}=1} |\mathbb{E}[(X^{\prime}_s)^{\ast}w^{\ast}]| = 0.
    \end{equation*}
    Hence, $X^{\prime}_s = 0$, $\mathbb{P}$-a.s., for $s$ in a dense subset of $[0,T]$, without loss of generality, we can assume that the dense is countable and, by continuity, we have that $X^{\prime}$ is indistinguishable from $0$.

    The second statement is a direct consequence of the first. Indeed, assume that there are two Gubinelli derivatives for $Y$, say $Y^1$, $Y^2$, then
    \begin{equation*}
        0 = (X^1-X^2)_{s}Z_{s,t} + R_{s,t}
    \end{equation*}
    for some $R \in C^{2\alpha}_2L^q(W)$. This implies that $X^1-X^2$ is the Gubinelli derivative of $0$, which is $0$, by uniqueness.
\end{proof}

\begin{rmk}
    Proposition \ref{prop: uniqueness truly rough} 
    is a slight generalization of \cite[Remark 3.3]{FHL21}.
    Notice that in that paper they only assume that $X$ is a martingale, whereas,
    under assumptions, $X$ can also be the sum of a martingale and a finite variation process.
    
    Moreover, the assumption of path continuity for $X^{\prime}$ is needed if we want uniqueness up to indistinguishability. Without uniqueness one obtains that any two Gubinelli derivatives coincide on a dense subset of times, up to a null set, which possibly depends on time.
   
\end{rmk}

As an immediate consequence, we have uniqueness in the decomposition of a path as the sum of a stochastic integral, a rough integral and a Lebesgue integral. Assume $W$ is a finite dimensional Banach space.

\begin{coroll}
\label{cor: weak norris}
     Let $\mathbf{Z}=(Z,\mathbb{Z}) \in \mathscr{C}^{\alpha}([0,T];\R^n)$, with $\alpha \in (\frac{1}{3},\frac{1}{2}]$, and assume $Z$ truly rough. Let $(X_t)_{t\in[0,T]}$ be a rough It\^o process with coefficients $b,\sigma$. Assume that $X \equiv 0$, $\mathbb{P}$-almost surely, then 
     $b\equiv 0$, $\sigma \equiv 0$ and $(X^{\prime},X^{\prime\prime}) \equiv 0$, $\mathbb{P}$-almost surely.
     
\end{coroll}
\begin{proof}
    We define the path
    \begin{equation}
    \label{eq: rough stoch leb equal 0}
        Y_t :=
        \int_0^t b_s \mathrm{d}s
         + \int_0^t \sigma_s \mathrm{d} B_s
         = -\int_0^t (X^{\prime},X^{\prime\prime})_s \mathrm{d}\mathbf{Z}_s.
    \end{equation}
    Then $(Y,X^{\prime}) \in \mathscr{D}_{\mathbf{Z}}^{2\alpha}L^{p}(U)$ for every $p\geq 0$. However, $\mathbb{E}_{\cdot}\delta Y\in C^{2\alpha}_2L^p(U)$. By Proposition \ref{prop: uniqueness truly rough} we have that $(X^{\prime},X^{\prime\prime}) \equiv (0,0)$, which also implies that the rough integral and hence each term in \eqref{eq: rough stoch leb equal 0} is $0$. Since the central term in \eqref{eq: rough stoch leb equal 0} is $0$, by classical stochastic calculus also $b$ and $\sigma$ vanish.
\end{proof}

We are now ready to prove existence of densities under Hörmander condition. 


\begin{thm}
\label{thm: existence under hoermander}
Let $\alpha \in (\frac{1}{3},\frac{1}{2}]$ and $\mathbf{Z}=(Z,\mathbb{Z}) \in \mathscr{C}^{0,\alpha}_g([0,T];\R^n)$ be a geometric rough path with $Z$ truly rough.
    Assume that the coefficients $b,\sigma,\beta$ satisfy H\"ormander's condition \ref{def: heormander condition} { at $x_0 \in \R^d$}. Then, {for any $t \in (0,T]$}, the solution $X_t$ to equation \eqref{eq:solutionRSDEinmalliavinforRSDEs} starting from $X_0 = x_0$ admits a density with respect to the Lebesgue measure.
\end{thm}

\begin{proof}
    The proof is an adaptation of the classical proof for SDEs (see \cite[Theorem 2.3.2]{nualart2006malliavin}). The goal is to apply \cite[Theorem 2.1.2]{nualart2006malliavin} by showing that the Malliavin matrix $\gamma_{X_t}$, defined in \eqref{eq: malliavin matrix}, is almost surely invertible. Equivalently, we check that the reduced Malliavin matrix $C_t$ is almost surely invertible. 

    First we notice that we can actually replace the spaces $\{\mathcal{S}_i\}_{i\in\mathbb{N}}$ with $\bar{\mathcal{S}}_0 = \mathcal{S}_0$ and
    \begin{align} 
    \label{eq: def of bar S spaces}
    \begin{aligned}
    \bar{\mathcal{S}}_{i+1} := &\bar{\mathcal{S}}_{i} \cup
    \{
        [F,V] \mid F\in\{\sigma^1,\dots,\sigma^m,\beta^1,\dots,\beta^n\},
        V\in \bar{\mathcal{S}}_{i}
    \}\\
    &\cup
    \left\{
    [\tilde b,V] + \frac{1}{2}\sum_{k=1}^m [\sigma^k,[\sigma^k,V]] \mid V \in \bar{\mathcal{S}}_{i}
    \right\},
    \qquad
    i\in \mathbb{N}.
    \end{aligned}
\end{align}
    Indeed, $\{V(x) \mid V\in\bar{\mathcal{S}}_i, \, i \in \mathbb{N} \}$ spans the same space as $\{V(x) \mid V\in\mathcal{S}_i, \, i \in  \mathbb{N}\}$, for every $x\in\R^d$.
    
    Fix $t \in (0,T]$. Assume that there exist $\xi \in \R^d$ such that, $\mathbb{P}$-a.s.
    \begin{align*}
        0=\xi^{\top} C_t \xi
        :=\xi^{\top}\left(
        \int_0^t I_s \sigma\sigma^{\top}(X_s)I_s^{\top} \rd s
        \right)\xi
        = \sum_{k=1}^{m} \int_0^t |\xi^{\top} I_s \sigma^k(X_s)|^2 \rd s,
    \end{align*}
    where $I$ is a solution to the linear equation \eqref{eq:equationforI}. By the time-continuity of the processes $X$ and $I$, we have that, $\mathbb{P}$-a.s.
    \begin{equation*}
        \sup_{s\in[0,t]}|\xi^{\top} I_s V(X_s)| = 0,
        \qquad
        \forall V \in \mathcal{S}_0.
    \end{equation*}
    If we can prove that, for every $i\in \mathbb{N}$, $\mathbb{P}$-a.s.
    \begin{equation}
    \label{eq: supremum zero}
        \sup_{s\in[0,t]}|\xi^{\top} I_s V(X_s)| = 0,
        \qquad
        \forall V \in \bar{\mathcal{S}}_i.
    \end{equation}
    we can then take $s=0$ and we have that $\xi$ is orthogonal to $V(x_0)$ for every $V\in \bar{\mathcal{S}}_i$. By H\"ormander's condition, there exists $i\geq 0$ such that $\operatorname{span}\{V(x_0) \mid V\in \bar{\mathcal{S}}_i\} = \R^d$, and hence $\xi = 0$. This implies that $C_t$ is almost surely invertible.

    It is only left to prove \eqref{eq: supremum zero}. We proceed by induction.

    Let $F\in \bar{\mathcal{S}}^{i}$ be a smooth  bounded vector field. We define \begin{equation}
    \label{eq: definition of W(F)}
    W_t(F) := \xi^{\top} I_t F(X_t).
    \end{equation}
    
    It follows from the rough It\^o formula \cite[Theorem 4.13]{FHL21} that $F(X_t)$ is a rough It\^o process satisfying the assumptions of Proposition \ref{prop: product formula}.
    By Proposition \ref{prop: product formula} we have
     \begin{align}
    \label{eq: equation for xi I F}
        \rd W_t(F)
        &= 
        \left(W_t([\tilde b,F]) + \frac{1}{2} \sum_{k=1}^{m} W_t([\sigma^k,[\sigma^k,F]])
        \right) \rd t 
         + W_t([\sigma,F]) \rd B_t
        + W_t([\beta,F]) \rd \mathbf{Z}_t,
    \end{align}
    where $[\sigma,F](x) := ([\sigma^k,F](x))_{k=1,\dots,m}$, $[\beta,F](x) := ([\beta^k,F](x))_{k=1,\dots,n}$ and $\tilde b$ is defined as in \eqref{eq:Stratonovich corrected drift}.
    By inductive hypothesis, we have that $\sup_{s\in[0,T]}|W_s(F)| = 0$, $\mathbb{P}$-a.s. By Corollary \ref{cor: weak norris} it follows that each integrand in \eqref{eq: equation for xi I F} vanishes. This completes the inductive argument and the proof.
\end{proof}


\subsection{Smoothness of densities under H\"ormander condition}

In order to obtain smoothness of densities, it is crucial to obtain an analogue of Norris' Lemma for a stochastic controlled rough path.

The starting point is to give a quantitative definition of truly rough (see Definition \ref{def: truly rough}). We have the following from \cite[Definition 6.7]{FH20}

\begin{defn}
    \label{def: theta hoelder rough}
    We say that { $Z:[0,T] \to \R^n$}
    is $\theta$-H\"older rough for $\theta \in (0,1)$ on scale $\epsilon_0 > 0$ if there exists a constant $L:=L(\theta,\epsilon_0,T,Z) > 0$ such that for every $v\in \R^n$, $s\in [0,T]$ and $\epsilon\in (0,\epsilon_0]$ there exists $t\in [0,T]$ with
    \begin{equation*}
        |t-s| < \epsilon,
        \quad
        |v\cdot Z_{s,t}|
        \geq L\epsilon^{\theta} |v|.
    \end{equation*}
\end{defn}

    Notice that, if ${Z} \in {C}^{\alpha}([0,T],\R^n), \, \alpha \in (\frac13, \frac12],$ is $\theta$-H\"older rough for some $\theta < 2\alpha$ then it is truly rough. In particular this implies that every stochastic controlled rough path $(X,X^{\prime}) \in \mathscr{D}_{{Z}}^{2\alpha}L^{p,q}(W)$ has a unique Gubinelli derivative $X^{\prime}$.

    Moreover, the notion of $\theta$-rough allows us to give a quantitative estimate of the Gubinelli derivative in terms of the path and {a suitable} remainder, as shown in Lemma \ref{lem: key estimate norris} below.
    {
    In order to do so, we first need to recall the Doob-Meyer decomposition for stochastic controlled rough paths in \cite[Theorem 3.3]{FHL21} and prove a generalized version of Kolmogorov's continuity theorem (cf.\ Lemma \ref{lemma:kolmogorov} below).
    }

    \begin{thm}\label{thm:DoobMeyerdecomposition}
        (see \cite[Theorem 3.3]{FHL21}) 
        Let $Z \in C^\alpha([0,T],\R^n)$ for some $\alpha \in (\frac14,1]$.
        Let $q \in [2,\infty)$ and let $(X, X') \in \mathscr{D}_Z^{2 \alpha} L^q(W)$ be a stochastic controlled rough path. 
        Then there exists a uniquely characterized pair of processes $(M, \phi)$ such that:
        \begin{enumerate} \renewcommand{\labelenumi}{\roman{enumi}.}
            \item $M$ is a martingale with $M_0 = 0$;
            \item $\phi$ is an adapted process with $\| \delta \phi_{s, t} - X'_s    \delta Z_{s, t} \|_q \lesssim | t - s |^{2 \alpha}$ for any $s,t \in [0,T]$;
            \item $X_t = M_t + \phi_t$ $\mathbb{P}$-a.s.\ for every $t \in [0,T]$.
        \end{enumerate}
        In particular, if $(X,X')$ is a rough It\^o process as in Definition \ref{def: rough ito process}, $M_t = \int_0^t \sigma_s dB_s$ and $\phi_t = \int_0^t b_s ds + \int_0^t (X'_s,X''_s) d\mathbf{Z}_s$. 
    \end{thm}

    \begin{lemma} \label{lemma:kolmogorov}
    In the setting of Theorem \ref{thm:DoobMeyerdecomposition}, define the two-parameter process \begin{equation*}
        R^{\phi}_{s, t} :=  \delta \phi_{s, t} - X'_s \delta Z_{s, t} \qquad s,t \in [0,T]. 
    \end{equation*}
    Then, for any $\gamma \in [ 0, 2 \alpha - \frac{2}{q})$, there exist a  modification of $R^{\phi}$ (also denoted by $R^{\phi}$) and a random  variable $K = K (\gamma, R^{\phi}) \in L^q(\Omega;\R)$ such that 
    \begin{equation*}
        | R^{\phi}_{s, t} | (\omega) \leqslant K (\omega)  | t - s |^\gamma \qquad \text{for any  $s, t \in [0, T]$} .
    \end{equation*}
    \end{lemma}
    
    \begin{proof}
        The proof is a straightforward modification of the Kolmogorov criterion for rough paths in \cite[Theorem 3.1]{FH20}, where Chen's relation for the two-parameter process $R^\phi$ is substituted by the following 
        \begin{equation*}
            \delta R^{\phi}_{s, u, t} := R^{\phi}_{s, t} - R^{\phi}_{s, u} - R^{\phi}_{u, t} = \delta X'_{s, u} \otimes \delta Z_{u, t}, \qquad \text{for any  $s, u, t \in [0, T]$} .
        \end{equation*}
        The complete proof is given in Appendix \ref{appendixKolmogorov}.      
    \end{proof}

\begin{lemma}
    \label{lem: key estimate norris}
    Let $\alpha \in (\frac{1}{3}, \frac{1}{2}]$ and $\theta < 2\alpha$. If ${Z} \in {C}^{\alpha}([0,T];\R^n)$ is $\theta$-H\"older rough on scale $\epsilon_0>0$ with modulus $L>0$ and $(X,X^{\prime})\in \mathscr{D}_{Z}^{2\alpha}L^{q}(W)$ { for some $q \in  [2,\infty)$} then
    \begin{equation*}
        \forall \epsilon \in (0,\epsilon_0],
        \qquad
        L \epsilon^{\theta}
        \sup_{t\in[0,T]} 
        |X^{\prime}_t|
        \leq \sup_{|t-s|\leq \epsilon}
        |
            \mathbb{E}_{s \wedge t}(\delta X_{s,t}) -  R^\phi_{s,t}
        |
    \end{equation*} 
    $\mathbb{P}$-a.s.
    Moreover, for every { $\bar{\alpha} \in [0,\alpha-\frac{1}{q})$}
    we have $\mathbb{P}$-a.s.
    \begin{equation*}
    \sup_{t\in[0,T]}|X^{\prime}_t| 
    {
     \le \frac{3}{L} \sup_{t \in [0,T]} |X_t|
    }
    \left(|R^\phi|_{2\bar{\alpha}}^{\theta/2\bar{\alpha}} \sup_{t\in[0,T]}|X_t|^{-\theta/2\bar{\alpha}} \vee \epsilon_0^{-\theta}
        \right), 
    \end{equation*}
    where $R^\phi$ is defined as in Lemma \ref{lemma:kolmogorov}. 
\end{lemma}

\begin{proof}
If $\sup_{t\in[0,T]}|X_t| = 0$, then $\sup_{t\in [0,T]} |X_t^{\prime}| = 0$ by Proposition \ref{prop: uniqueness truly rough} and the statement becomes trivial. We thus give the proof under the assumption that $\sup_{t\in[0,T]}|X_t|>0$, $\mathbb{P}$-a.s.

Throughout the proof we keep track of the space whenever we take the relevant norm (unless it is the usual Euclidean norm in $\R$).

    We first prove the first statement.
    Let $s\in [0,T]$ and $\epsilon \in (0,\epsilon_0]$. Moreover, in the following we can fix $\omega \in \Omega$ as the argument is completely deterministic. We take $w^{\ast} \in W^{\ast}$ and, by calling $X^{\ast} \in \mathscr{L}(W,\R^n)$ the adjoint operator of $X$, we notice that $(X_s)^{\ast}w^{\ast} \in \R^n$. There exists a random time $\bar{t} = \bar{t}(s,\epsilon,w^\ast) > 0$ such that $|\bar{t}-s| \leq \epsilon$ and
    \begin{equation}
    \label{eq: application of theta rough}
     L \epsilon^{\theta}|(X^\prime_s)^{\ast}w^{\ast}|_{\R^n}
     < |(X^\prime_s)^{\ast}w^{\ast}\cdot \delta Z_{s,\bar{t}}| 
     \leq 
     |w^{\ast}(X^\prime_s \delta Z_{s,\bar{t}})|
     \leq \sup_{t: |t-s|\leq \epsilon} |w^{\ast}(X^\prime_s \delta Z_{s,t})|.
    \end{equation}
    
    We take the supremum in $s$ and $|w^{\ast}|_{W^{\ast}} \leq 1$ in \eqref{eq: application of theta rough} to get 
    \begin{align*}
        L \epsilon^{\theta} \sup_s |X^{\prime}_s|_{\mathscr{L}(\R^n;W)}
        & = L \epsilon^{\theta} \sup_s \sup_{|w^{\ast}|_{W^{\ast}} \leq 1}|(X^{\prime}_s)^{\ast}w^{\ast}|_{\R^n} \\
        & \leq 
        \sup_{s,t: |t-s|\leq \epsilon}\sup_{|w^{\ast}|_{W^{\ast}} \leq 1}|w^{\ast}(X^\prime_s \delta  Z_{s,t})|\\
        & \leq 
        \sup_{|t-s|\leq \epsilon} |X^\prime_s \delta Z_{s,t}|_W.
    \end{align*}
    In order to conclude the proof notice that, by { the Doob-Meyer decomposition for stochastic controlled rough path in Theorem \ref{thm:DoobMeyerdecomposition}}, we have, $\mathbb{P}$-a.s.
    \begin{equation*}
        \sup_{|t-s|\leq \epsilon} |X^\prime_s \delta Z_{s,t}|_W
        = \sup_{|t-s|\leq \epsilon}
          |
            \delta \phi_{s,t} - R^\phi_{s,t}
        |_W.
    \end{equation*}
    { and, by martingale property and for any $s,t \in [0,T]$, $\delta \phi_{s,t} = \E_{s \wedge t} (\delta X_{s,t})$ $\mathbb{P}$-a.s. }

    The second statement follows from the first by choosing 
     \begin{equation*}
     \epsilon 
     :=|R^\phi|_{2\bar{\alpha}}^{-1/2\bar{\alpha}}\sup_{t\in[0,T]}|X_t|_W^{1/2\bar{\alpha}}\wedge\epsilon_0.
     \end{equation*}
     As the previous argument is completely deterministic, we can actually choose $\epsilon$ random. From Lemma \ref{lemma:kolmogorov} we know that $\epsilon(\omega) \in [0,\epsilon_0]$ is well defined $\omega$-a.s., provided that \linebreak $2\bar{\alpha} \in [0,2\alpha - \frac{2}{q})$. 
     

    With our choice of $\epsilon$, we have
    \begin{equation}
    \label{eq: where sup X comes from}
        L\epsilon^{\theta} \sup_{s}|X^{\prime}_s|_{\mathscr{L}(\R^n;W)}
        \leq \sup_{|t-s|\leq \epsilon}
          \left|
            \E_{s \wedge t} (\delta X_{s,t})
            \right|_W
            + \left|\mathbb{E}_\cdot R
        \right|_{2\bar{\alpha}} \epsilon^{2\bar{\alpha}}
        \leq { 3 \sup_{t \in [0,T]} |X_t|_W}
    \end{equation}
    where in the last inequality we used the definition of $\epsilon$ as follows
    \begin{equation*}
    |R^\phi|_{2\bar{\alpha}} \epsilon^{2\bar{\alpha}}
        = \sup_{t\in[0,T]}|X_t|_W\wedge\epsilon_0^{2\bar{\alpha}} |R^\phi|_{2\bar{\alpha}}
        \leq \sup_{t\in[0,T]}|X_t|_W.
    \end{equation*}
    From \eqref{eq: where sup X comes from} we can conclude by multiplying on both side by $L^{-1} \epsilon^{-\theta}$.
\end{proof}

We are now ready to prove a version of Norris's lemma for stochastic controlled rough paths.
\begin{prop}
    \label{pro: Norris lemma}
    Let $\alpha \in (\frac{1}{3}, \frac{1}{2}]$ and $\theta < 2\alpha$. Let $\mathbf{Z} = (Z, \mathbb{Z}) \in \mathscr{C}^{\alpha}([0,T],\R^n)$ with $Z$ $\theta$-H\"older rough on scale $\epsilon_0>0$ with modulus $L>0$. Let $(X_t)_{t\in[0,T]}$ be a rough-It\^o process with coefficients $b,\sigma$ such that, for $\beta = 1/3$ and for every $p\geq 0$,
    \begin{equation}
    \label{eq: assumption Norris lemma Hairer}
        \| |b|_{\beta} \|_{p} < \infty
        \qquad
        \| |\sigma|_{\beta} \|_{p} < \infty.
    \end{equation}
    
    There exists $r\in (0,1)$ such that for every $\gamma>0$ and $\epsilon > 0$
    \begin{equation*}
        \mathbb{P}\left(
        \sup_{t\in [0,T]}|X_t| < \epsilon,
        \quad
        \sup_{t\in[0,T]}|X^{\prime}_t| \vee \sup_{t\in[0,T]}|\sigma_t| \vee \sup_{t\in[0,T]}|b_t| > \epsilon^r
        \right)
        \leq C_{\gamma}\epsilon^{\gamma}.
    \end{equation*} 
\end{prop}

\begin{proof}
We call $A := \{\sup_{t\in [0,T]}|X_t| < \epsilon\}$.
If $\mathbb{P}(A) = 0$ the statement is trivial, so we assume that $\mathbb{P}(A) > 0$.

    \textbf{Step 1} As a first step we prove that there exists $r\in (0,1)$ such that for every $\gamma > 0$ there exists $C_{\gamma} > 0$ such that
    \begin{equation*}
        \mathbb{P}\left(
        \sup_{t\in[0,T]}|X^{\prime}_t| > \epsilon^r,
        A
        \right)
        \leq C_{\gamma}\epsilon^{\gamma}.
    \end{equation*} 
    Notice that $(X,X^{\prime}) \in \mathscr{D}_{Z}^{2\alpha}L^{q}(W)$ { for any $q \in [2,\infty)$ and define $R^\phi_{s,t}$ as in Lemma \ref{lemma:kolmogorov}}.
    By Markov inequality and Lemma \ref{lem: key estimate norris} we have for $\theta < 2\bar{\alpha} < 2\alpha - {2}/q$,
    \begin{align*}
        \mathbb{P}\left(
        \sup_{t\in[0,T]}|X^{\prime}_t| > \epsilon^r,
        A
        \right)
        & \leq 
        \epsilon^{-qr}\mathbb{E}\left[\sup_{t\in[0,T]}|X^{\prime}_t|^q \ 1_A \right]\\
        & \leq \frac{3^q}{L^q}\epsilon^{q(1-r-\theta/2\bar{\alpha})}\mathbb{E}\left[
         \left(
         |R^\phi|_{2\bar{\alpha}}^{\theta/2\bar{\alpha}}
         \vee \epsilon_0^{-\theta}\epsilon^{\theta/2\bar{\alpha}}
        \right)^q
        \right]\\
        & \leq C_q \left(
        \mathbb{E}[
        |R^\phi|_{2\bar{\alpha}}^{\theta q / 2\bar{\alpha}}
        ]+1
        \right)
        \epsilon^{q(1-r-\theta/2\bar{\alpha})}\\
        & \leq C_q 
        \epsilon^{q(1-r-\theta/2\bar{\alpha})},
    \end{align*} 
    where in the last line we used that, since $\theta  / 2\bar{\alpha} < 1$, we have by Lemma \ref{lemma:kolmogorov}
    \begin{equation*}
        \mathbb{E}[
        |R^\phi|_{2\bar{\alpha}}^{\theta q / 2\bar{\alpha}}
        ]
        \leq 
        \||R^\phi|_{2\bar{\alpha}}\|_q^{q\theta / 2\bar{\alpha}}
        { \leq
        \| K(2 \bar\alpha,R^\phi) \|_q^{q\theta / 2\bar{\alpha}}
        }
        \leq C_q < \infty.
    \end{equation*}
    We can choose $r \in (0,1-\theta/2\bar{\alpha})$ and $q$ such that $\gamma =q(1-r-\theta/2\bar{\alpha})$.

    \textbf{Step 2} By step $1$, applied to $(X,X^{\prime})$ and and then to $(X^{\prime},X^{\prime\prime})$ we have that there exists $r \in (0,1)$ such that for every $\gamma > 0$ and $\epsilon > 0$ we have
    \begin{align*}
        \mathbb{P}\left(\sup_{t\in[0,T]}|X^{\prime\prime}_t| > \epsilon^{r^2}, A\right)
        &\leq 
        \mathbb{P}(\sup_{t\in[0,T]}|X^{\prime\prime}_t| > \epsilon^{r^2},
        \sup_{t\in[0,T]}|X^\prime_t| \leq \epsilon^{r}) + C_{\gamma}\epsilon^{\gamma}
        \leq C_{\gamma} \epsilon^{r\gamma}
    \end{align*}

    By possibly replacing $r$ with $r^2$, we conclude, that there exists $r\in (0,1)$, such that, for every $\gamma > 0$ there exists $C_{\gamma} > 0$ such that
    \begin{equation*}
        \mathbb{P}\left(
        \sup_{t\in[0,T]}|X_t| \le \epsilon,
        \sup_{t\in[0,T]}|X^{\prime}_t| \vee \sup_{t\in[0,T]}|X^{\prime\prime}_t| > \epsilon^r
        \right)
        \leq C_{\gamma}\epsilon^{\gamma}.
    \end{equation*} 
    \textbf{Step 3.} We now show that the rough integral is large with very small probability when $X$ is small. 
    We must prove that there exists $r\in (0,1)$ such that for every $\gamma > 0$ there exists $C_{\gamma} > 0$ such that
    \begin{equation}
    \label{eq: small rough integral when rough-ito is small}
    \mathbb{P}\left(\sup_{t\in[0,T]}|\int_0^t (X^{\prime},X^{\prime\prime})_s\mathrm{d}\mathbf{Z}_s| > \epsilon^r,
     \sup_{t\in[0,T]}|X_t| \le \epsilon
        \right)
        \leq C_{\gamma}\epsilon^{\gamma}.
    \end{equation} 
    For $r^{\prime} \in (0,1)$ we define the set
    \begin{equation*}
        A^{\prime} :=\left\{
        \sup_{t\in[0,T]}|X^{\prime}_t| \le \epsilon^{r^{\prime}},
                \sup_{t\in[0,T]}|X^{\prime\prime}_t| \le \epsilon^{r^{\prime}},
                \sup_{t\in[0,T]}|X_t| \le \epsilon
        \right\}
    \end{equation*}
    and we notice that, by Steps 1 and 2, for every $\gamma > 0$, there exists $C_{\gamma}$ such that, for every $\epsilon > 0$,
    \begin{multline*}
        \mathbb{P}\left(\sup_{t\in[0,T]}|\int_0^t (X^{\prime},X^{\prime\prime})_s\mathrm{d}\mathbf{Z}_s| > \epsilon^r,
        \sup_{t\in[0,T]}|X_t| \le \epsilon
        \right) \\ \leq C_{\gamma}\epsilon^{\gamma}
        + 
               \mathbb{P}\left(\sup_{t\in[0,T]}|\int_0^t (X^{\prime},X^{\prime\prime})_s\mathrm{d}\mathbf{Z}_s| > \epsilon^r,
               A^{\prime}
        \right).
    \end{multline*}
    Notice that the Riemann sum approximation of the rough integral is given by $$\Xi_{s,t}:= X^{\prime}_s \delta Z_{s,t} + X^{\prime\prime}_s \mathbb{Z}_{s,t}.$$ For every $\omega \in A^{\prime}$, we have the deterministic bound
    \begin{equation*}
        |\Xi_{s,t}(\omega)| \leq C \epsilon^{r^{\prime}} |t-s|^{\alpha}.
    \end{equation*}
    We call $\pi$ a partition of $[0,T]$ with mesh size $|\pi| = \epsilon^{r^{\prime}/2}$. For every $t\in [0,T]$, we define by $\pi[0,t]$ the restriction of $\pi$ to the sub interval $[0,t]$.
    From the stochastic sewing lemma (see \cite[(2.23) in Theorem 2.9]{FHL21}) we have that there exists $\eta \in (0,1)$ and $C>0$ such that
     \begin{equation}
     \label{eq: rieman sum approx tu rough integral}
        \left\|
            \sup_{t\in[0,T]}|\int_0^t (X^{\prime},X^{\prime\prime})_s\mathrm{d}\mathbf{Z}_s- \sum_{[u,v]\in \pi[0,t]} \Xi_{u,v}|
        \right\|_{q} 
        \leq C \epsilon^{\eta r^{\prime}/2}.
    \end{equation}
    We then use Markov inequality and \eqref{eq: rieman sum approx tu rough integral} to obtain that for every $r\in (0,1)$, there exists $C>0$ such that
    \begin{align*}
        &\mathbb{P}\left(\sup_{t\in[0,T]} |\int_0^t (X^{\prime},X^{\prime\prime})_s\mathrm{d}\mathbf{Z}_s| > \epsilon^r,
               A^{\prime}
        \right)^{\frac{1}{q}}
        \leq \epsilon^{-r}
        \left\|
            \sup_{t\in[0,T]}|\int_0^t (X^{\prime},X^{\prime\prime})_s\mathrm{d}\mathbf{Z}_s| \ 1_{A^{\prime}}
        \right\|_{q} \\
        &\quad \leq \epsilon^{-r}
        \left\|
            \sup_{t\in[0,T]}|\int_0^t (X^{\prime},X^{\prime\prime})_s\mathrm{d}\mathbf{Z}_s- \sum_{[u,v]\in \pi[0,t]} \Xi_{u,v}|
        \right\|_{q} 
        + \epsilon^{-r}
        \left\|
            \sup_{t\in[0,T]}|\sum_{[u,v]\in \pi[0,t]} \Xi_{u,v}| 1_{A^{\prime}}
        \right\|_{q} \\
        &\quad \leq C \epsilon^{\eta r^{\prime}/2 - r}
        + \epsilon^{-r}
        \sum_{[u,v]\in \pi} \left\|
             |\Xi_{u,v}| 1_{A^{\prime}}
        \right\|_{q} 
        \leq C (\epsilon^{\eta r^{\prime}/2 - r}
        + \epsilon^{r^{\prime}-r-(1-\alpha)r^{\prime}/2})\\
        &\quad \leq C (\epsilon^{(\eta r^{\prime}/2 - r) \wedge 
        (r^{\prime}-r-r^{\prime}/2)}).
    \end{align*}
    We can now choose $r = \eta r^{\prime}/4$ and $q\geq 2$ to obtain that $\gamma := q ((\eta r^{\prime}/2 - r) \wedge (r^{\prime}/2-r)) > 0$. This concludes the proof of \eqref{eq: small rough integral when rough-ito is small}.
    
    \textbf{Step 4.} We only need to prove now that $b$ and $\sigma$ are small with high probability, when $X$ is small. Let
    \begin{equation*}
        Y_t := X_t - \int_{0}^{t}(X^{\prime},X^{\prime\prime}) \rd \mathbf{Z}_s = X_0 + \int_{0}^{t} b_s \rd s + \int_{0}^{t} \sigma_s \rd B_s.
    \end{equation*}
    It follows from step 4 that,
    \begin{align*}
        \mathbb{P}(\sup_{t\in[0,T]}|Y_t| > \epsilon + \epsilon^{r}, A) 
        \leq C_{\gamma} \epsilon^{\gamma}
    \end{align*}
    From the previous inequality it follows that, for every $z\in(0,1)$,
    \begin{align}
        \mathbb{P}(\sup_{t\in[0,T]}|b_t|\vee  \sup_{t\in[0,T]}|\sigma_t| > \epsilon^{z}, A) \leq 
        C_{\gamma} \epsilon^{\gamma}
        + \mathbb{P}(\sup_{t\in[0,T]}|b_t|\vee  \sup_{t\in[0,T]}|\sigma_t| > \epsilon^{z}, \sup_{t\in[0,T]}|Y_t| \leq \epsilon + \epsilon^{r})
        \label{eq: b and sigma in Norris lemma}
    \end{align}
   From the classical Norris' lemma for stochastic differential equations, for instance \cite[Lemma 4.11]{hairer2011malliavins}, and assumption \eqref{eq: assumption Norris lemma Hairer} we have that there exists $z\in (0,1)$ such that the last term in \eqref{eq: b and sigma in Norris lemma} is bounded by $C_{\gamma}\epsilon^{\gamma}$ for every $\gamma > 0$.
    

\end{proof}

Thanks to the previous analogous of Norris's Lemma we can finally prove the main result of this section.

    \begin{thm}
    \label{thm: smoothness of densities}
    Let $\alpha \in (\frac{1}{3}, \frac{1}{2}]$ and $\theta < 2\alpha$. Let $\mathbf{Z} \in \mathscr{C}^{0,\alpha}_g([0,T];\R^n)$ with $Z$ $\theta$-H\"older rough on scale $\epsilon_0>0$ with modulus $L>0$.
    Assume that the coefficients $b,\sigma,\beta$ satisfy H\"ormander's condition \ref{def: heormander condition} at {$x_0 \in \R^d$}. Then, {for any $t \in (0,T]$}, the solution $X_t$ to equation \eqref{eq:solutionRSDEinmalliavinforRSDEs} starting from $X_0=x_0$ admits a smooth density with respect to the Lebesgue measure.
\end{thm}
\begin{proof}
    Let $t \in (0,T]$. 
    Following classical results for stochastic differential equations we just need to prove that the determinant of the inverse Malliavin matrix of $X_t$ has moments of all order.
    {As discussed at the beginning of Section \ref{sec:existencedensity_Hoermander} and}
    by \cite[Lemma 2.3.1]{nualart2006malliavin}, it is enough to prove that for every $p\geq 2$ and for all $\epsilon >0$,
    \begin{equation}
    \label{eq: criterion for smoothness}
        \sup_{|\xi| = 1} \mathbb{P}(\xi^{\top} C_t \xi \leq \epsilon ) \leq C_p \epsilon^p,
    \end{equation}
    where $C_t$ is the reduced Malliavin matrix of $X_t$, defined in \eqref{eq: reduced malliavin matrix}.

    Proceeding as in the proof of Theorem \ref{thm: existence under hoermander}, we have
   \begin{align}
        \label{eq: low bonud malliavin matrix}
        \xi^{\top} C_t \xi
        = \sum_{k=1}^{m} \int_0^t |W_r(\sigma^k)|^2 \rd r
        \geq \frac{1}{T}
        \sum_{k=1}^{m} 
        \left(\int_0^t |W_r(\sigma^k)| \rd r\right)^2,
    \end{align}
    We would like to bound the right hand side of \eqref{eq: low bonud malliavin matrix} from below by the supremum norm to use our version of Norris's lemma, Proposition \ref{pro: Norris lemma}. In particular notice that $W_t(\sigma^k)$ is a rough-It\^o process in the sense of Definition \ref{def: rough ito process} that satisfies \eqref{eq: assumption Norris lemma Hairer}. In particular $|W_\cdot(\sigma^k)|_{\beta} \in L^p$ for every $p\geq 0$ and for $\beta=\frac{1}{3}$.
    
    From \cite[Lemma 4.10]{hairer2011malliavins} applied with $\alpha = \beta = \frac{1}{3}$, we have
    \begin{equation}
        \label{eq: bound supremum}
        \sup_{s\in[0,t]}|W_s(\sigma^k)| \leq 4 \left(\int_0^t |W_r(\sigma^k)| \rd r\right)^{\frac{1}{4}} \max \left\{
        \left(\int_0^t |W_r(\sigma^k) |
        \rd r \right)^{\frac{3}{4}},  |W(\sigma^k)|_{\beta}
        \right\}.
    \end{equation}
    Using that $|W(\sigma^k)|_{\beta}$ has moments of all order and applying Markov inequality, we have for any $p > 0$ and $q > 0$,
    \begin{align*}
         \mathbb{P} (\xi^{\top} C_t \xi \leq \epsilon)
         & \leq \mathbb{P} (\xi^{\top} C_t \xi \leq \epsilon, |W(\sigma^k)|_{\beta} \leq \epsilon^{-q}\}
         + C_p \epsilon^p\\
         & \leq \mathbb{P}(
         \sup_{s\in[0,t]}|W_s(F)| \leq 4 T^{\frac{1}{8}}\epsilon^{\frac{1}{8}} \max \{
            T^{\frac{3}{8}}\epsilon^{\frac{3}{8}}, \epsilon^{-q}\}, \ \forall F \in \mathcal{S}_0
         ) + C_p \epsilon^p,
    \end{align*}
    where in the last inequality we used the bound \eqref{eq: bound supremum}. Also, recall the definition of the spaces $\mathcal{S}_i$, $i\in \mathbb{N}$, given in equation \eqref{eq: def of S spaces} and the definition of the spaces $\bar{\mathcal{S}}_i$ given in \eqref{eq: def of bar S spaces}.

    When $\epsilon > 0$ is small enough and by also choosing $q>0$ small, we have {for any $p > 0$}
    \begin{equation*}
        \mathbb{P} (\xi^{\top} C_t \xi \leq \epsilon)
        \leq \mathbb{P}(\sup_{s\in [0,t]}|W_s(F)| \leq \epsilon^{\frac{1}{9}} , \, \forall F \in \mathcal{S}_0 ) { + C_p \epsilon^p}.
    \end{equation*}
    Now, for any vector field $F \in \mathcal{S}_0$, we have that $W_t(F)$ satisfies equation \eqref{eq: equation for xi I F}. In particular $W_t(F)$ is a rough-It\^o process that satisfies Proposition \ref{pro: Norris lemma}, whose coefficients are of the form $W_s(G)$, for $G\in \bar{\mathcal{S}}_1$, and satisfy \eqref{eq: assumption Norris lemma Hairer}. Indeed, the processes $I$ and $X$ are in $C^{\alpha}_tL^p_{\omega}$ for any $p\geq 2$ and some $\alpha \in (\frac{1}{3}, \frac{1}{2}]$. Notice in particular, that $I$ and $X$ are rough-It\^o process. We can thus apply rough-It\^o formula, Proposition \ref{prop: product formula}, in order to derive the explicit identify \eqref{eq: equation for xi I F} as a rough It\^o process for $W_s(F)$.

    We thus have by Proposition \ref{pro: Norris lemma} that there exists $r\in(0,1)$ such that for all $\gamma > 0$, there exists $C_{\gamma} > 0$ with
    \begin{multline*}
         \mathbb{P}(\sup_{s\in [0,t]}|W_s(F)| \leq \epsilon , \forall F\in\mathcal{S}_0) \\ \leq 
        \mathbb{P}(\sup_{s\in [0,t]}|W_s(F)|  \leq \epsilon, \forall F\in\mathcal{S}_0,
        \sup_{s\in [0,t]}|W_s(G)| \leq \epsilon^r, \forall G\in \bar{\mathcal{S}}_1) + C_{\gamma}\epsilon^{\gamma}.
    \end{multline*}
    Noting that $W_t(G)$ satisfies again Proposition \ref{pro: Norris lemma}, we can iterate the previous argument $k$ times and obtain that there exists $r_k \in (0,1)$ such that, for every $\gamma >0$ there exists $C = C(\gamma,k)$ such that
    \begin{equation*}
        \mathbb{P} (\xi^{\top} C_t \xi \leq \epsilon)
        \leq \mathbb{P}(
         \sup_{s\in [0,t]}|W_s(G)| \leq \epsilon^{r_k}, \forall G \in \bar{\mathcal{S}}_k) + C \epsilon^{\gamma}.
    \end{equation*}
    Notice that $W_0(G) = \xi^{\top} I_0 G(X_0) = \xi^{\top} G(x_0)$. Moreover, by H\"ormander condition, there exists a $k>0$ such that the vectors in $\bar{\mathcal{S}}_k(x_0)$ span the entire space $\R^d$. Consequently, since $|\xi|=1$, we have that
    \begin{equation*}
        \{\sup_{s\in [0,t]}|W_s(G)| \leq \epsilon^{r_k}, \forall G \in \bar{\mathcal{S}}_k
        \}
        \subset
        \{|\xi^{\top} G(x_0)| \leq \epsilon^{r_k},
        \forall G\in \bar{\mathcal{S}}_k
        \}
        =\emptyset.
    \end{equation*}
    This implies that \eqref{eq: criterion for smoothness} is satisfied and the proof is complete.
\end{proof}

\begin{appendices}

\section{Weighted norms} \label{appendix weighted norms}
We introduce a new class of \lq\lq weighted\rq\rq \, Hölder norms, which can be seen as the stochastic counterpart of the ones introduced in \cite[Section 2]{bailleul2017unbounded}. 
The definition of these norms, and consequently all the related results, depends on the choice of a parameter $\lambda \in (0,+\infty)$. In some cases we want to control this parameter and to be able to choose it small enough. Throughout this subsection, all the stochastic processes are assumed to take values in a finite dimensional real Hilbert space $(W,|\cdot|)$.  

\begin{defn}[Weighted norms] \label{def:weightednorms} Let $\lambda >0$, $\alpha \in (0,1]$ and $p \in [1,\infty)$. Let $Y=(Y_t)_{t \in [0,T]}$ be a $L^p$-integrable stochastic process. We define \begin{equation*}  \label{eq:weigthednorm1}
    (|Y_{\cdot}|)_{p;\lambda} := \sup_{t \in [0,T]} \frac{\|Y_t\|_p}{e^{t/\lambda}}.
\end{equation*} Let $A=(A_{s,t})_{(s,t)\in\Delta_{[0,T]}}$ be a $L^p$-integrable two-paramenter stochastic process. We define
\begin{align*}
    (|A|)_{\alpha;p;\lambda} &:= \sup_{\underset{|t-s|\le \lambda}{0\le s<t\le T}} \ \frac{\|A_{s,t}\|_p}{e^{t/\lambda}|t-s|^\alpha}. \label{eq:weightednorm2}
\end{align*} 
Let $\Xi=(\Xi_{s,u,t})_{(s,u,t) \in \Delta_{[0,T]}^2}$ be an $L^p$-integrable three-parameter stochastic process.  We define \begin{equation*}
     (|\Xi|)_{\alpha;p;\lambda} := \sup_{\underset{|t-s|\le \lambda}{0\le s<u<t\le T}} \ \frac{\|\Xi_{s,u,t}\|_p}{e^{t/\lambda}|t-s|^\alpha}.
\end{equation*}
\end{defn}

\begin{rmk} \label[rmk]{rmk:weightednormswithdifferentlambda}
    Let $\lambda_1,\lambda_2 >0$ be such that $\lambda_1 \le \lambda_2$. For any $t \in [0,T]$, it holds \begin{equation*}
        \frac{\|Y_t\|_p}{e^{t/\lambda_1}} = \frac{\|Y_t\|_p}{e^{t/\lambda_2}} e^{(t/\lambda_2-t/\lambda_1)}  \le  \frac{\|Y_t\|_p}{e^{t/\lambda_2}} \le (|Y_\cdot|)_{p;\lambda_2}
    \end{equation*} and, therefore, $(|Y_\cdot|)_{p;\lambda_1} \le (|Y_\cdot|)_{p;\lambda_2}$. With the same technique, one case also prove that $(|\cdot|)_{\alpha;p;\lambda_1} \le (|\cdot|)_{\alpha;p;\lambda_2}$.
\end{rmk}

 The following result shows the link between the weighted norm of a process and the weighted norm of its increment. The presence of that $\lambda^\alpha$ factor will appear as a fundamental motivation for choosing this kind of norms to handle linear RSDEs.
 
\begin{prop}[Comparison result] \label{prop:comparison}
    Let $\lambda>0$, $\alpha \in (0,1]$ and $p \in [1,\infty)$ be fixed. Then, for any $L^p$-integrable stochastic process $Y=(Y_t)_{t \in [0,T]}$, it holds \begin{equation*}
    (|Y_{\cdot}|)_{p;\lambda} \le \|Y_0\|_p+e^2\lambda^\alpha(|\delta Y|)_{\alpha;p;\lambda}.
\end{equation*}
\end{prop}
\begin{proof}
    Let $t \in [0,T]$. Let $N=N(t)\in\mathbb{N}_{\ge1}$ be such that $(N-1)\lambda < t\le N\lambda$ and consider the partition $\pi^{[0,t],\lambda}=\{t_0=0,t_1=\lambda,\dots,t_{N-1}=(N-1)\lambda,t_{N}=t\}$ of the interval $[0,t]$. Notice that $|\pi^{[0,t],\lambda}|\le\lambda$. Then, for any $t \in [0,T]$, the following holds: \begin{equation*}
        \begin{aligned}
        \|Y_t\|_p &= \|Y_0+\sum_{k=0}^{N-1}(Y_{t_{k+1}}-Y_{t_k})\|_p 
        \le \|Y_0\|_p e^{t/\lambda}+(|\delta Y|)_{\alpha;p;\lambda}  \left(\sum_{k=0}^{N-1}e^{t_{k+1}/\lambda} \right)\lambda^\alpha.
    \end{aligned}
    \end{equation*}
    By construction $(N+1)\lambda < t+2\lambda$. Therefore, by noticing that
    \begin{equation*} \label{eq:weightednormsinequalitywiththeexponentials}
        \begin{aligned}
            \sum_{k=0}^{N-1}e^{\frac{t_{k+1}}{\lambda}} &\le \sum_{k=0}^{N-1}e^{\frac{(k+1)\lambda}{\lambda}} = \frac{e^{N+1}-e}{e-1} \le e^{N+1} < e^\frac{t+2\lambda}{\lambda} = e^2 e^{\frac{t}{\lambda}}
        \end{aligned} 
    \end{equation*} we reach the conclusion. 
\end{proof}

\begin{rmk}\label[rmk]{rmk:equivalencewiththeclassicalnorms} 
Consider a $L^p$-integrable stochastic process $Y=(Y_t)_{t \in [0,T]}$, for $p \in [1,\infty)$, and let $\lambda>0$. By definition of weighted norm, it is straightforward to see that \begin{equation*} \label{eq:inequalityforweightednorm1}
    (|Y_\cdot|)_{p;\lambda}\le \sup_{t \in [0,T]}\|Y_t\|_p \le e^{T/\lambda}(|Y_\cdot|)_{p;\lambda}.
\end{equation*} 
Let $\gamma \in (0,1]$. As a general fact $(|\cdot|)_{\gamma;p;\lambda} \le \|\cdot\|_{\gamma;p}$, being $e^{t/\lambda}\ge 1$ for any $t \in [0,T]$. By definition, given an $L^p$-integrable stochastic process $A=(A_{s,t})_{[s,t]\in\Delta_{[0,T]}}$ it holds that \begin{equation*}
    (|A|)_{\gamma;p;\lambda} \le \|A\|_{\gamma;p}.
\end{equation*}
If $A$ is additive (i.e.\ $\delta A=0$), we also get that
\begin{equation*} \label{eq:equivalenceholdernormspt1}
    \|A\|_{\gamma;p} \le e^2e^{T/\lambda} (|A|)_{\gamma;p;\lambda}.
\end{equation*} 
Indeed, the inequality follows by applying the same argument of the proof of Proposition \ref{prop:comparison} over any time interval $[s,t]$ and by writing $A_{s,t}$ as a telescopic sum of increments of $A$ along $\pi^{[s,t],\lambda}$. In particular, the norm $(|\cdot|)_{p;\lambda}+(|\delta \cdot|)_{\gamma;p;\lambda}$ is equivalent to the standard norm on ${C}^\gamma L^p(W)$.
\end{rmk}

\begin{rmk} \label[rmk]{rmk:equivalenceforcontrolledroughpaths}
    Let $(Y,Y') \in \mathscr{D}_Z^{2\gamma}L^p(W)$ be a stochastic controlled rough path, in the sense of Definition \ref{def:stochasticcontrolledpaths}. For $\lambda>0$, we denote by \begin{equation*} \begin{aligned}
        (|(Y,Y')|)_{Z;2\gamma;p;\lambda} &= (|Y_\cdot|)_{p;\lambda} + (|\delta Y|)_{\gamma;p;\lambda} + (|Y'_\cdot|)_{p;\lambda} + (|\delta Y'|)_{\gamma;p;\lambda} + (|\E_\cdot R^Y|)_{2\gamma;p} \\
        &=: (|Y_\cdot|)_{p;\lambda} + (|(Y,Y')|)_{\mathscr{D}_Z^{2\gamma}L^p;\lambda}.
    \end{aligned}
    \end{equation*}
    It is straightforward to see that \begin{equation*}
        (|(Y,Y')|)_{Z;2\gamma;p;\lambda} \le \|(Y,Y')\|_{Z;2\gamma;p}.
    \end{equation*}
    Notice that the two parameter process $(\E_s(R^Y_{s,t}))_{(s,t)\in\Delta_{[0,T]}}$ is not additive. However, given any time interval $[s,t]$ and given the partition $\pi^{[s,t],\lambda}=\{t_i, i=0,\dots,N\}$ of $[s,t]$, we can always write \begin{equation*}
        \begin{aligned}
            R^Y_{s,t} 
            = \sum_{i=0}^{N-1} \delta Y_{t_i,t_{i+1}} - Y'_s \sum_{i=0}^{N-1} \delta Z_{t_i,t_{i+1}} 
            &= \sum_{i=0}^{N-1}R^Y_{t_i,t_{i+1}} +  \sum_{i=0}^{N-1} \delta Y'_{s,t_i} \delta Z_{t_i,t_{i+1}}.
        \end{aligned}
    \end{equation*}
    Applying similar arguments as above, it is possible to deduce that \begin{equation*} \begin{aligned}
        \|\E_s(R^Y_{s,t})\|_p &
        \lesssim_{\lambda,\alpha,\gamma,T} (|\E_\cdot R^Y|)_{2\gamma;p;\lambda} e^{t/\lambda}|t-s|^{2\gamma} + e^2e^{T/\lambda} (|\delta Y'|)_{\gamma;p;\lambda} |\delta Z|_{\alpha} e^{t/\lambda} |t-s|^{2\gamma}.
    \end{aligned}
    \end{equation*}
    Combining the previous estimate with Proposition \ref{prop:comparison} and \ref{rmk:equivalencewiththeclassicalnorms}, we obtain that \begin{equation*}
        \|(Y,Y')\|_{Z;2\gamma;p} \lesssim_{\lambda,\alpha,\gamma,T} (1+|\delta Z|_\alpha) (|(Y,Y')|)_{Z;2\gamma;p;\lambda}.
    \end{equation*}
    Analogously, one can deduce that $\|(Y,Y')\|_{\mathscr{D}_Z^{2\gamma}L^p} \lesssim_{\lambda,\alpha,\gamma,T,|\delta Z|_\alpha} (|(Y,Y')|)_{\mathscr{D}_Z^{2\gamma}L^p;\lambda}$.
\end{rmk}

By means of the stochastic sewing lemma \cite[Theorem 2.9]{FHL21}, it is also possible to deduce an estimate involving weighted norms of two-parameter processes. The following result shows that we can use stochastic sewing, not only to construct the rough stochastic integral, but also to deduce some useful bounds.

\begin{prop} \label{prop:sewingmap}
    Let $p \in [2,\infty),$ $\alpha \in (\frac{1}{3},\frac{1}{2}],$ $\gamma \in (\frac{1}{3},\alpha]$ and $\lambda>0$. Let $A=(A_{s,t})_{(s,t)\in\Delta_{[0,T]}}$ be a $L^p$-integrable stochastic process such that $A_{t,t}=0$, $A_{s,t}$ is $\mathcal{F}_t$-measurable, \begin{equation*}
        \|\delta A_{s,u,t}\|_p\lesssim|t-s|^{\alpha+\gamma} \quad \text{and} \quad \|\E_s(\delta A_{s,u,t})\|_p\lesssim|t-s|^{\alpha+2\gamma}
    \end{equation*} for any $(s,u,t)\in\Delta^2_{[0,T]}$. 
    By applying the stochastic sewing lemma in \cite[Theorem 2.9]{FHL21} 
    with $m=n=p$ and being $\frac{1}{3}<\gamma \le \alpha \le \frac{1}{2}$, the following sewing map is well defined: \begin{equation*} \Lambda(A): \Delta_{[0,T]} \ni (s,t) \longmapsto \Lambda(A)_{s,t} := \delta \mathcal{A}_{s,t}-A_{s,t} \in L^p(\Omega;W). \end{equation*}  Then there exist some constants $c_{\alpha,\gamma},c'_{\alpha,\gamma,p} >0$ such that 
    \begin{equation} \label{eq:sewingmap}
        (|\Lambda(A)|)_{\alpha+\gamma;p;\lambda} \le c_{\alpha,\gamma} \lambda^\gamma(|\E_\cdot \delta A|)_{\alpha+2\gamma;p;\lambda}+c'_{\alpha,\gamma,p}(|\delta A|)_{\alpha+\gamma;p;\lambda}.
    \end{equation}
\end{prop}
\begin{proof}
Let $(s,t)\in\Delta_{[0,T]}$ such that $|t-s|\le\lambda$. For any $h \in \mathbb{N}$, consider the $h$-th dyadic partition $\pi^{h}=\{s=t_0^h<t_1^h<\dots<t_{2^h}^h=t\}$ of $[s,t]$ and denote by $u_i^h$ the middle point of any interval $[t_i^h,t_{i+1}^h]$, $i=0,\dots,2^h-1$. Then the stochastic sewing lemma \cite[Theorem 2.9]{FHL21} and an application of BDG inequality - as shown in \cite[Section 2]{LE20} - lead to 
\begin{align*}
        \|\Lambda(A)_{s,t}\|_p &= \|\delta \mathcal{A}_{s,t}-A_{s,t}\|_p = \|L^p\text{-}\lim_{h\to +\infty}\sum_{i=0}^{2^h-1} A_{t_i^h,t_{i+1}^h}-A_{t^0_0,t^0_1}\|_p = \\ 
        &= \|L^p\text{-}\lim_{h\to +\infty} \sum_{k=0}^{h-1} \big(\sum_{i=0}^{2^{k+1}-1} A_{t_i^{k+1},t_{i+1}^{k+1}}- \sum_{i=0}^{2^k-1} A_{t_i^k,t_{i+1}^k}\big)\|_p = \\ 
        &\le \lim_{h \to +\infty} \sum_{k=0}^{h-1}\|\sum_{i=0}^{2^k-1} -\delta A_{t_i^k,u_i^k,t_{i+1}^k}\|_p  \\ 
        &\le \lim_{h \to +\infty} \sum_{k=0}^{h-1} \Big(\sum_{i=0}^{2^k-1} \|\E_{t_i^k}(\delta A_{t_i^k,u_i^k,t_{i+1}^k})\|_p + 2c_p^{1/p}\big(\sum_{i=0}^{2^k-1} \|\delta A_{t_i^k,u_i^k,t_{i+1}^k}\|_p^2\big)^{\frac{1}{2}} \Big). 
\end{align*}
To conclude we can now note that, for every $h\ge 0$,\begin{align*}
        \sum_{k=0}^{h-1} \sum_{i=0}^{2^k-1} \|\E_{t_i^k}(\delta A_{t_i^k,u_i^k,t_{i+1}^k})\|_p &\le \sum_{k=0}^{h-1} \sum_{i=0}^{2^k-1} (|\E_\cdot\delta A|)_{\alpha+2\gamma;p;\lambda} e^{t_{i+1}^k/\lambda} |t_{i+1}^k-t_i^k|^{\alpha+2\gamma}  \\ 
        &\le (|\E_\cdot\delta A|)_{\alpha+2\gamma;p;\lambda}e^{t/\lambda} \sum_{k=0}^{h-1} \sum_{i=0}^{2^k-1} (2^{-k}|t-s|)^{\alpha+2\gamma} \\
        &\lesssim_{\alpha,\gamma} (|\E_\cdot\delta A|)_{\alpha+2\gamma;p;\lambda}e^{t/\lambda}\lambda^\gamma |t-s|^{\alpha+\gamma}  
    \end{align*}
and \begin{align*}
    \sum_{k=0}^{h-1} \big(\sum_{i=0}^{2^k-1} \|\delta A_{t_i^k,u_i^k,t_{i+1}^k}\|_p^2\big)^{\frac{1}{2}} 
    &\le \sum_{k=0}^{h-1} \big(\sum_{i=0}^{2^k-1} (|\delta A|)_{\alpha+\gamma;p;\lambda}^2 e^{2t_{i+1}^k/\lambda} |t_{i+1}^k-t_i^k|^{2(\alpha+\gamma)}\big)^{\frac{1}{2}} \\ 
    &\lesssim_{\alpha,\gamma} (|\delta A|)_{\alpha+\gamma;p;\lambda}e^{t/\lambda}|t-s|^{\alpha+\gamma}.  
\end{align*}
\end{proof}

\section{Proof of Lemma \ref{lemma:kolmogorov}} \label{appendixKolmogorov}

We give a full proof of Lemma \ref{lemma:kolmogorov}.
The argument generalises the Kolmogorov criterion for rough paths and follows the same ideas as \cite[Theorem 3.1]{FH20}. 

\begin{proof}[Proof of Lemma \ref{lemma:kolmogorov}]
    Note that, for any $s, u , t \in [0,T]$, \begin{equation*}
        R^{\phi}_{s, t} - R^{\phi}_{s, u} - R^{\phi}_{u, t} = \delta X'_{s, u} \otimes \delta Z_{u, t} .
    \end{equation*}
    In particular, for any $s < t$ and for any partition $s = \tau_0 < \tau_1 < \cdots < \tau_N = t$ of $[s,t]$,
    \begin{equation}
        R^{\phi}_{s, t} = \sum_{i = 0}^{N - 1} R^{\phi}_{\tau_i, \tau_{i + 1}} +
        \delta X'_{s, \tau_i} \otimes \delta Z_{\tau_i, \tau_{i + 1}} .
        \label{eq:additivity}
    \end{equation}
Denote by $\pi^n$ the $n$-th dyadic partition of $[0, T]$, that is $\pi^n
:= \left\{ \frac{k T}{2^n}, k = 0, \ldots, 2^n \right\}$. Define the
random variable
\[ \mathbb{K}^n := \sup_{t \in \pi^n} | R^{\phi}_{t, t + 2^{- n}} | . \]
It holds, denoting by $\| R^{\phi} \|_{2 \alpha ; q} := \sup_{s<t} \frac{\| \delta \phi_{s,t} - X'_s \delta Z_{s,t}\|_q}{|t-s|^{2\alpha}} < +\infty$,
\begin{equation*}
    \mathbb{E} (| \mathbb{K}^n |^q) \le \mathbb{E} \Big( |
  \sum_{t \in \pi^n} R^{\phi}_{t, t + T 2^{- n}} |^q \Big) \le 2^n  \| R^{\phi} \|_{2 \alpha ; q}^q T^{2 \alpha q} (2^{- n})^{2  \alpha q} = \| R^{\phi} \|_{2 \alpha ; q}^q T^{2 \alpha q}  (2^{- n})^{2  \alpha q - 1} . 
\end{equation*}

Notice that we can apply classical Kolmogorov's continuity theorem to $X'$
because we know that $\mathbb{E} (| \delta X'_{s, t} |^q) \leq \| \delta
X' \|_{\alpha ; q}^q | t - s |^{\alpha q}$. 
Hence, for any $\gamma' \in [0, \alpha - \frac{1}{q} )$, there is a modification of $X'$ (also denoted by $X'$) and a random variable $C = C (\gamma', X') \in L^q(\Omega;\R)$ such that $| \delta X'_{s, t} | (\omega) \leq C (\omega) | t - s |^{\gamma'}$ for any $s,t \in [0,T]$. 
In particular, by applying classical Kolmogorov continuity criterion  also to $\phi$,  $(s, t) \mapsto R^{\phi}_{s, t}$ admits a pathwise continuous
modification and it is therefore sufficient to prove the claim on dyadic points and extend the result by continuity (and up to modifications) on the entire $[0, T]$.

Let $s < t$ be dyadic points in $\bigcup_{n \geqslant 0} \pi^n$. Choose $m
\geqslant 0$ such that
\[ T 2^{- (m + 1)} < t - s \leq T 2^{- m} . \]
Therefore, there exists a partition $s = \tau_0 < \tau_1 < \cdots < \tau_N =
t$ of $[s, t]$ such that \linebreak $(\tau_i , \tau_{i + 1}) \in \pi^n$ for some $n = n
(i) \geqslant m + 1$ and for each fixed $n \geqslant m + 1$ there are
at most two such intervals taken from $\pi^n$. 
By means of
\eqref{eq:additivity} and for any $\gamma' \in [ 0, \alpha - \frac{1}{q})$ we can write
\begin{align*}
    | R^{\phi}_{s, t} | & = \left| \sum_{i = 0}^{N - 1} R^{\phi}_{\tau_i, \tau_{i + 1}} + \delta X'_{s, \tau_i} \otimes \delta Z_{\tau_i, \tau_{i + 1}}  \right| 
    \le  \sum_{i = 0}^{N - 1} | R^{\phi}_{\tau_i, \tau_{i + 1}} | +  \sum_{i = 0}^{N - 1} | \delta X'_{s, \tau_i} \otimes
    \delta Z_{\tau_i, \tau_{i + 1}} | \\
    & \le 2 \sum_{n \geqslant m + 1} | \mathbb{K}^n | + \max_{i = 0, \ldots, N - 1} | \delta X'_{s, \tau_i} |  \Big( 2 \sum_{n \geqslant m + 1} \sup_{t \in \pi^n} | \delta  Z_{t, t + T 2^{- n}} | \Big) \\
    & \le 2 \sum_{n \geqslant m + 1} | \mathbb{K}^n | + \Big( 2  \sum_{n \geqslant m + 1} \sup_{t \in \pi^n} | \delta X'_{t, t + T 2^{- n}} | \Big) \Big( 2 \sum_{n \geqslant m + 1} \sup_{t \in \pi^n} | \delta Z_{t,  t + T 2^{- n}} | \Big) \\
    & \le 2 \sum_{n \geqslant m + 1} | \mathbb{K}^n | + 4 C(\gamma', X') | \delta Z_{\alpha}
  | T^{\alpha - \gamma'} \Big( \sum_{n \geqslant m + 1} | T 2^{- n} |^{\gamma'}  \Big)^2
\end{align*}
Let $\gamma \in [ 0, 2 \alpha - \frac{2}{q} )$ be arbitrary but
fixed and let us choose any $\gamma' \in [ 0, \alpha - \frac{1}{q})$ such that $\gamma < 2 \gamma'$. From the previous inequalities we
deduce that
    \begin{align*}
  \frac{| R^{\phi}_{s, t} |}{| t - s |^{\gamma}}  &<  \frac{| R^{\phi}_{s, t}
  |}{T^\gamma 2^{- \gamma (m + 1)}} \\
  &\lesssim_{T, \gamma}  \frac{2 }{2^{- \gamma (m + 1)}} 
   \sum_{n \geqslant m + 1} |
  \mathbb{K}^n | + 4  C | \delta Z_{\alpha} | T^{\alpha - \gamma'}  \Big(
  \frac{1}{2^{- \frac{\gamma}{2} (m + 1)}} \sum_{n \geqslant m + 1} | T 2^{- n}
  |^{\gamma'} \Big)^2 .
\end{align*}
Note that (by monotone convergence)
\begin{eqnarray*}
  \Big\| \frac{2}{2^{- \gamma (m + 1)}} \sum_{n \geqslant m + 1} |
  \mathbb{K}^n | \Big\|_q  \leq \sum_{n \geqslant 0} \frac{2}{2^{-
  \gamma n}} \| \mathbb{K}^n \|_q 
   \leq  2 \| R^{\phi} \|_{2 \alpha ; q}^q T^{2 \alpha q}  \sum_{n
  \geqslant 0} (2^{- n})^{2 \alpha - \frac{1}{q} - \gamma} < + \infty,
\end{eqnarray*}
since $2 \alpha - \frac{1}{q} - \gamma > 0 \Leftrightarrow \gamma < 2 \alpha -
\frac{1}{q}$.
Moreover, recall that $C \in L^q(\Omega;\R)$ and
\begin{equation*}
  \frac{1}{2^{- \frac{\gamma}{2} (m + 1)}} \sum_{n \geqslant m + 1} | T 2^{- n}
  |^{\gamma'}  \leqslant  \sum_{n \geqslant 0} \frac{1}{2^{- \frac{\gamma}{2}
  n}}  | T 2^{- n} |^{\gamma'} < + \infty \quad \Leftrightarrow \quad  \gamma < 2
  \gamma' . \qedhere
\end{equation*}
\end{proof}

\end{appendices}

\printbibliography

@article{CNNR23,
author = {Michele Coghi and Torstein Nilssen and Nikolas N{\"u}sken and Sebastian Reich},
title = {{Rough McKean–Vlasov dynamics for robust ensemble Kalman filtering}},
volume = {33},
journal = {The Annals of Applied Probability},
number = {6B},
publisher = {Institute of Mathematical Statistics},
pages = {5693 -- 5752},
year = {2023},
doi = {10.1214/23-AAP1957},
URL = {https://doi.org/10.1214/23-AAP1957}
}

@article{FHL21,
  author  = {Friz, Peter K. and Hocquet, Antoine and L{\^e}, Khoa},
  title   = {Rough stochastic differential equations},
  journal = {arXiv preprint arXiv:2106.10340, v6},
  year    = {2021}, 
  doi  = {https://doi.org/10.48550/arXiv.2106.10340},
}

@book{BC09,
   author     = {Bain, Alan and Crisan, Dan},
  doi        = {10.1007/978-0-387-76896-0},
  isbn       = {978-0-387-76895-3},
  mrclass    = {60G35 (60-02 93E11)},
  mrnumber   = 2454694,
  mrreviewer = {Monique Pontier},
  pages      = {xiv+390},
  publisher  = {Springer, New York},
  series     = {Stochastic Modelling and Applied Probability},
  title      = {{Fundamentals of stochastic filtering}},
  url        = {https://doi.org/10.1007/978-0-387-76896-0},
  volume     = 60,
  year       = 2009
}

@book{FH20,
  title        = {{A course on rough paths. With an introduction to regularity structures}},
  author       = {Friz, Peter K. and Hairer, Martin},
  year         = 2020,
  publisher    = {Springer, Cham},
  series       = {Universitext},
  edition      = {Second},
  doi={https://doi.org/10.1007/978-3-030-41556-3},
}

@article{Lyo98,
  author     = {Lyons, Terry J.},
  doi        = {10.4171/RMI/240},
  fjournal   = {Revista Matem\'{a}tica Iberoamericana},
  issn       = {0213-2230},
  journal    = {Rev. Mat. Iberoamericana},
  mrclass    = {60H10 (34A99 34F05 37H99)},
  mrnumber   = 1654527,
  mrreviewer = {Ben Hambly},
  number     = 2,
  pages      = {215--310},
  title      = {{Differential equations driven by rough signals}},
  url        = {https://doi.org/10.4171/RMI/240},
  volume     = 14,
  year       = 1998
}

@Book{CARMONADELARUE_volumeII,
 Author = {Carmona, Ren{\'e} and Delarue, Fran{\c{c}}ois},
 Title = {Probabilistic theory of mean field games with applications {II}. {Mean} field games with common noise and master equations},
 FSeries = {Probability Theory and Stochastic Modelling},
 Series = {Probab. Theory Stoch. Model.},
 ISSN = {2199-3130},
 Volume = {84},
 ISBN = {978-3-319-56435-7; 978-3-319-56436-4; 978-3-319-59820-8},
 Year = {2018},
 Publisher = {Cham: Springer},
 Language = {English},
 DOI = {10.1007/978-3-319-56436-4},
 zbMATH = {6721977},
 Zbl = {1422.91015}
}

@article{DFS17,
     author = {Diehl, Joscha and Friz, Peter K. and Stannat, Wilhelm},
     title = {{Stochastic partial differential equations: a rough paths view on weak solutions via {Feynman{\textendash}Kac}}},
     journal = {Annales de la Facult\'e des sciences de Toulouse : Math\'ematiques},
     pages = {911--947},
     publisher = {Universit\'e Paul Sabatier, Toulouse},
     volume = {Ser. 6, 26},
     number = {4},
     year = {2017},
     doi = {10.5802/afst.1556},
     language = {en},
     url = {http://www.numdam.org/articles/10.5802/afst.1556/}
}

@article{CN21,
title = {Rough nonlocal diffusions},
journal = {Stochastic Processes and their Applications},
volume = {141},
pages = {1-56},
year = {2021},
issn = {0304-4149},
doi = {https://doi.org/10.1016/j.spa.2021.07.002},
url = {https://www.sciencedirect.com/science/article/pii/S0304414921001083},
author = {Michele Coghi and Torstein Nilssen}
}

@article{LE20,
  title={A stochastic sewing lemma and applications},
  author={L{\^e}, Khoa},
  journal={Electron. J. Probab},
  volume={25},
  number={38},
  pages={1--55},
  year={2020},
  doi = {https://doi.org/10.1214/20-ejp442},
}

@article{CDFO13,
author = {Dan Crisan and Joscha Diehl and Peter K. Friz and Harald Oberhauser},
title = {{Robust filtering: Correlated noise and multidimensional observation}},
volume = {23},
journal = {The Annals of Applied Probability},
number = {5},
publisher = {Institute of Mathematical Statistics},
pages = {2139 -- 2160},
year = {2013},
doi = {10.1214/12-AAP896},
URL = {https://doi.org/10.1214/12-AAP896}
}

@article{DOR13,
 author     = {Diehl, Joscha and Oberhauser, Harald and Riedel, Sebastian},
  doi        = {10.1016/j.spa.2014.08.005},
  fjournal   = {Stochastic Processes and their Applications},
  issn       = {0304-4149},
  journal    = {Stochastic Process. Appl.},
  mrclass    = {60H10 (60G15 60G17 60H05 60H30)},
  mrnumber   = 3274695,
  mrreviewer = {Keisuke Hara},
  number     = 1,
  pages      = {161--181},
  title      = {{A L\'evy area between Brownian motion and rough paths with applications to robust nonlinear filtering and rough partial differential equations}},
  url        = {https://doi.org/10.1016/j.spa.2014.08.005},
  volume     = 125,
  year       = 2015
}

@article{kurtz1999particle,
	title={Particle representations for a class of nonlinear {SPDEs}},
	author={Kurtz, Thomas G and Xiong, Jie},
	journal={Stochastic Processes and their Applications},
	volume={83},
	number={1},
	pages={103--126},
	year={1999},
	publisher={Elsevier},
    doi={https://doi.org/10.1016/S0304-4149(99)00024-1},
}

@article {coghi2016propagation,
    AUTHOR = {Coghi, Michele and Flandoli, Franco},
     TITLE = {Propagation of chaos for interacting particles subject to
              environmental noise},
   JOURNAL = {Ann. Appl. Probab.},
  FJOURNAL = {The Annals of Applied Probability},
    VOLUME = {26},
      YEAR = {2016},
    NUMBER = {3},
     PAGES = {1407--1442},
      ISSN = {1050-5164},
   MRCLASS = {60K35 (60H15 60K37 82C22)},
  MRNUMBER = {3513594},
MRREVIEWER = {Xinghua Zheng},
       DOI = {10.1214/15-AAP1120},
       URL = {https://doi.org/10.1214/15-AAP1120},
}

@book{friz2020course,
  title={A course on rough paths},
  author={Friz, Peter K and Hairer, Martin},
  year={2020},
  publisher={Springer}
}

@article{davie2008differential,
  author = {Davie, Alexander M.},
    title = {{Differential Equations Driven by Rough Paths: An Approach via Discrete Approximation}},
    journal = {Applied Mathematics Research eXpress},
    year = {2008},
    month = {01},
    issn = {1687-1200},
    doi = {10.1093/amrx/abm009},
    url = {https://doi.org/10.1093/amrx/abm009},
    eprint = {https://academic.oup.com/amrx/article-pdf/doi/10.1093/amrx/abm009/648318/abm009.pdf},
}

@book{nualart2006malliavin,
  title={The Malliavin calculus and related topics},
  author={Nualart, David},
  volume={1995},
  year={2006},
  publisher={Springer Science \& Business Media},
  edition={Second},
  doi={https://doi.org/10.1007/3-540-28329-3},
  isbn       = {978-3-540-28328-7},
}

@article{hairer2011malliavins,
  title={{On Malliavin's proof of H{\"o}rmander's theorem}},
  author={Hairer, Martin},
  journal={Bulletin des sciences mathematiques},
  volume={135},
  number={6-7},
  pages={650--666},
  year={2011},
  publisher={Elsevier},
  doi={https://doi.org/10.1016/j.bulsci.2011.07.007},
}

@inproceedings{bailleul2017unbounded,
    TITLE = {{Unbounded rough drivers}},
  AUTHOR = {Bailleul, Isma{\"e}l and Gubinelli, Massimiliano},
  URL = {https://hal.science/hal-01278756},
  NOTE = {29 pages},
  JOURNAL = {{Annales de la Facult{\'e} des Sciences de Toulouse. Math{\'e}matiques.}},
  PUBLISHER = {{Universit{\'e} Paul Sabatier \_ Cellule Mathdoc }},
  VOLUME = {26},
  NUMBER = {4},
  PAGES = {795-830},
  YEAR = {2017},
  DOI = {10.5802/afst.1553},
  HAL_ID = {hal-01278756},
  HAL_VERSION = {v1},
}

@article{cass2010densities,
  title={{Densities for rough differential equations under H{\"o}rmander's condition}},
  author={Cass, Thomas and Friz, Peter},
  journal={Annals of mathematics},
  pages={2115--2141},
  year={2010},
  publisher={JSTOR},
  doi={http://doi.org/10.4007/annals.2010.171.2115},
}

@article{hairer2013regularity,
  title={{Regularity of laws and ergodicity of hypoelliptic SDEs driven by rough paths}},
  author={Hairer, Martin and Pillai, Natesh S},
  journal={The Annals of Probability},
  volume={41},
  number={4},
  pages={2544--2598},
  year={2013},
  publisher={Institute of Mathematical Statistics},
  doi={DOI:10.1214/12-AOP777},
}

@article{cass2015smoothness,
  title={{Smoothness of the density for solutions to Gaussian rough differential equations}},
  author={Cass, Thomas and Hairer, Martin and Litterer, Christian and Tindel, Samy},
  journal={Annals of Probability},
  volume={43},
  number={1},
  pages={188--239},
  year={2015},
  doi={https://doi.org/10.1214/13-AOP896},
}

@article{hu2013smooth,
  title={{Smooth density for some nilpotent rough differential equations}},
  author={Hu, Yaozhong and Tindel, Samy},
  journal={Journal of Theoretical Probability},
  volume={26},
  pages={722--749},
  year={2013},
  publisher={Springer},
  doi={DOI: 10.1007/s10959-011-0388-x},
}

@article{driscoll2010smoothness,
  title={{Smoothness of density for the area process of fractional Brownian motion}},
  author={Driscoll, Patrick},
  journal={arXiv preprint arXiv:1010.3047},
  year={2010}, 
  doi={https://doi.org/10.48550/arXiv.1010.3047},
}

@book{malliavin2015stochastic,
  title={{Stochastic analysis}},
  author={Malliavin, Paul},
  volume={313},
  year={2015},
  publisher={Springer},
  doi={https://doi.org/10.1007/978-3-642-15074-6},
  isbn={978-3-642-15074-6},
}

@article{friz2020existence,
  author     = {Friz, Peter K. and Nilssen, Torstein and Stannat, Wilhelm},
  doi        = {10.1016/j.jde.2019.09.033},
  fjournal   = {Journal of Differential Equations},
  issn       = {0022-0396},
  journal    = {J. Differential Equations},
  mrclass    = {60H15 (60L20)},
  mrnumber   = 4042356,
  number     = 4,
  pages      = {1686--1721},
  title      = {{Existence, uniqueness and stability of semi-linear rough partial differential equations}},
  url        = {https://doi.org/10.1016/j.jde.2019.09.033},
  volume     = 268,
  year       = 2020
}

@article{gubinelli2004controlling,
  author     = {Gubinelli, M.},
  doi        = {10.1016/j.jfa.2004.01.002},
  fjournal   = {Journal of Functional Analysis},
  issn       = {0022-1236},
  journal    = {J. Funct. Anal.},
  mrclass    = {60H05 (26A42 60H10)},
  mrnumber   = 2091358,
  mrreviewer = {Vigirdas Mackevi\v cius},
  number     = 1,
  pages      = {86--140},
  title      = {{Controlling rough paths}},
  url        = {https://doi.org/10.1016/j.jfa.2004.01.002},
  volume     = 216,
  year       = 2004
}

\end{document}